\documentclass[12pt]{amsart}
\usepackage{graphicx} 

\usepackage[T1]{fontenc}
\usepackage[utf8]{inputenc}
\usepackage[english]{babel}
\usepackage{amsmath,amsthm,amssymb,latexsym,amsfonts, mathtools}
\usepackage{mathrsfs} 
\usepackage{tikz-cd}
\usepackage{framed}
\usepackage{listings}
\usepackage{shuffle}
\usepackage{faktor} 
\usepackage{nicefrac}
\usepackage{color,soul}
\usepackage{setspace}
\usepackage{graphicx}
\usepackage{faktor}
\usepackage[normalem]{ulem} 
\usepackage[a4paper, margin=1in]{geometry}
\usepackage[alphabetic]{amsrefs}
\usepackage[colorlinks,linkcolor=black,anchorcolor=blue,citecolor=blue]{hyperref}

\usepackage{xpatch}
\makeatletter   
\xpatchcmd{\@tocline}
{\hfil\hbox to\@pnumwidth{\@tocpagenum{#7}}\par}
{\ifnum#1<0\hfill\else\dotfill\fi\hbox to\@pnumwidth{\@tocpagenum{#7}}\par}
{}{}
\makeatother

\usepackage{etoolbox}
\makeatletter
\patchcmd\@thm
  {\let\thm@indent\indent}{\let\thm@indent\noindent}%
  {}{}
\makeatother

\makeatletter
\patchcmd{\@maketitle}
  {\ifx\@empty\@dedicatory}
  {\ifx\@empty\@date \else {\vskip3ex \centering\footnotesize\@date\par\vskip1ex}\fi
   \ifx\@empty\@dedicatory}
  {}{}
\patchcmd{\@adminfootnotes}
  {\ifx\@empty\@date\else \@footnotetext{\@setdate}\fi}
  {}{}{}
\makeatother

\makeatletter
\patchcmd{\@thm}{\thm@headfont{\scshape}}{\thm@headfont{\scshape\normalfont\bfseries}}{}{}
\patchcmd{\@thm}{\thm@notefont{\fontseries\mddefault\upshape}}{}{}{}
\makeatother
\usepackage{subfiles}

\DeclareFontFamily{U}{wncy}{}
\DeclareFontShape{U}{wncy}{m}{n}{<->wncyr10}{}
\DeclareSymbolFont{mcy}{U}{wncy}{m}{n}
\DeclareMathSymbol{\sh}{\mathord}{mcy}{"78}

\renewcommand{\epsilon}{\varepsilon}
\renewcommand{\theta}{\vartheta}
\renewcommand{\rho}{\varrho}
\let\temp\phi
\let\phi\varphi
\let\varphi\temp
\newcommand{\insubfile}[1]{\ifx\@onlypreamble\@notprerr\else#1\fi}
\newcommand{\f}[1]{\mathfrak{#1}}

\newcommand{\Z}{\mathbb{Z}}
\newcommand{\Q}{\mathbb{Q}}
\newcommand{\R}{\mathbb{R}}
\newcommand{\C}{\mathbb{C}}
\newcommand{\F}{\mathcal{F}}
\newcommand{\A}{\mathcal{A}}
\newcommand{\G}{\mathcal{G}}
\renewcommand{\S}{\mathcal{S}}

\theoremstyle{definition}
\newtheorem{definition}{Definition}[subsection]
\theoremstyle{plain}
\newtheorem{proposition}[definition]{Proposition}
\theoremstyle{plain}
\newtheorem{theorem}[definition]{Theorem}
\theoremstyle{plain}
\newtheorem{lemma}[definition]{Lemma}
\theoremstyle{plain}
\newtheorem{corollary}[definition]{Corollary}
\theoremstyle{plain}

\theoremstyle{definition}
\newtheorem{lemma*}{Lemma}
\theoremstyle{definition}
\newtheorem{remark}[definition]{Remark}
\theoremstyle{definition}
\newtheorem{example}[definition]{Example}
\theoremstyle{plain}
\newtheorem{intro}{Theorem}

\AtBeginDocument{%
   \def\MR#1{}
}

\title{Periods of hyperplane arrangements and multiple polylogarithms}
\author{Riccardo Tosi}
\date{\today}

\address[Riccardo Tosi]{ University of Duisburg-Essen, Fakultät für Mathematik, Thea-Leymann-Str. 9, 45127 Essen, Germany.}

\email[]{riccardo.tosi@uni-due.de}

\begin{document}

    \begin{abstract}
        
    We compute the periods associated with a special class of hyperplane arrangements. In particular, we exhibit a combinatorial condition on the intersection lattice of a hyperplane arrangement that ensures that its associated periods are linear combinations of special values of multiple polylogarithms. Our method generalizes Brown's approach to the periods of moduli spaces of curves of genus zero. We apply this result to the reflection arrangement of the full monomial group, whose periods are shown to be linear combinations of values of multiple polylogarithms at roots of unity.

    \end{abstract}
    \noindent
     
    \maketitle

    \noindent
    {\small 2020 \textit{Mathematics Subject Classification.} \\
    Primary 11M32, 14N20, 32G20. \\
    Secondary  11G55, 32S22, 33B30. 
    }
    \tableofcontents

    \newpage

    \section*{Introduction}
    
    The connection between zeta values and the moduli space $\overline{\f{M}}_{0,n}$ of stable curves of genus zero with $n$ marked points was first outlined by Goncharov and Manin \cite{Goncharov-Manin-Multiple_zeta_motive_and_moduli_spaces_of_curves} and further explored by Brown \cite{Brown-Multiple_zeta_values_and_periods_of_moduli_spaces_of_curves}, who showed that the periods of $\overline{\f{M}}_{0,n}$ are $\Q[2\pi i]$-linear combinations of multiple zeta values. His argument was effective and was later implemented into an algorithm by Panzer~\cite{Panzer-Algortihms_for_the_symbolic_integration_of_hyperlogarithms_with_applications_to_Feynman_integrals}. Brown \cite{Brown-Mixed_Tate_motives_over_Z} also gave a motivic version of this fact, showing that all periods of mixed Tate motives unramified over $\Z$ are generated by multiple zeta values together with $2\pi i$. \\
    In view of these results, there has been an extensive search for linear combinations of zeta values among period integrals over moduli spaces of curves, in the hope that the geometry of these varieties may lead to new irrationality proofs. Promising families of these periods have been studied via the theory of the so-called \emph{cellular integrals} \cite{Brown-Irrationality_proofs_for_zeta_values_moduli_spaces_and_dinner_parties}, which have recovered many classical irrationality results \cites{Apery-Irrationalité_de_zeta2_et_zeta3, Rhin-Viola-On_a_permutation_group_related_to_zeta2,
    Ball-Rivoa-Irrationalité_d'une_infinité_de_valuers_de_la_fonction_zeta_aux_entiers_impairs, Rhin-Viola-The_group_structure_for_zeta3} and have led, for example, to the best known rational approximations of $\zeta(5)$ \cite{Brown-Zudilin-Cellular_rational_approximations_to_zeta5}. \\
    Many of the key insights of \cite{Brown-Multiple_zeta_values_and_periods_of_moduli_spaces_of_curves} rely on the fact that $\overline{\f{M}}_{0,n}$ is the compactification of the complement of a hyperplane arrangement in the affine space, namely
    \[
        \f{M}_{0,n}=\text{Spec}\; \Z[t_1,\dots, t_l]\left[ \frac{1}{t_i},\frac{1}{t_i-1}, \frac{1}{t_i-t_j} \,\middle\vert\, i,j=1,\dots, l, \; i\neq j \right].
    \]
    The aim of this work is to understand which hyperplane arrangements have periods that can be computed by means of Brown's techniques. We will prove that such techniques do not depend on the special symmetries of $\overline{\f{M}}_{0,n}$, but rather only on purely combinatorial properties of the associated arrangement. \\ 
    In particular, we isolate a class of hyperplane arrangements, which we call \emph{with enough modular elements}, whose periods are generated by special values of multiple polylogarithms and can be effectively computed with the Brown-Panzer algorithm. This new property of arrangements depends only on the lattice-isomorphism class of the arrangement and is strictly stronger than supersolvability. \\
    More precisely, given a hyperplane arrangement $\A$ defined over a subfield $k$ of $\C$, let $Y_\A$ be its complement in the affine space. Write $l=\dim Y_\A$ and denote by $\overline{Y}_\A$ a De Concini-Procesi compactification of $Y_\A$ \cite{De_Concini-Procesi-Wonderful_models_of_subspace_arrangements}. The condition for $\A$ to have enough modular elements is satisfied when the modular lines in the intersection lattice of $\A$ generate the ambient vector space of $\A$. Geometrically, this ensures the existence of retractions from $\overline Y_\A$ onto each irreducible component of the boundary divisor $\overline Y_\A\setminus Y_\A$.\\
    Roughly speaking, our main result reads as follows.
    \begin{intro}[see Theorem~\ref{theo: periods of hyperplane arrangements with enough modular elements}]
        \label{theo: intro periods of arrangements with enough modular elements}
        Assume that $\A$ has enough modular elements. Let $A,B\subseteq \overline Y_\A\setminus Y_\A$ be two divisors without common irreducible components. Consider an admissible relative homology class $\Delta\in H_l(\overline Y_\A\setminus A, B\setminus (A\cap B))$ and an algebraic differential $l$-form $\omega$ of $\overline Y_\A \setminus A$ over $k$.
        Then the period integral 
        \[
            \int_\Delta \omega
        \]
        lies in the $k$-subalgebra of $\C$ generated by
        \begin{enumerate}
            \item values of multiple polylogarithms at a specific finite set of points of $k$, which depends only on the one-dimensional hyperplane arrangements obtained from $\A$ by iterated restrictions and deletions;
            \item a finite set of logarithms of elements of $k$ which depends explicitly on $\Delta$.
        \end{enumerate}
        Moreover, this $k$-algebra can be endowed with a natural filtration and $\int_\Delta \omega$ lies in the $l$-th term of this filtration.
        \end{intro}
        \noindent
        In the case of $\overline{\f{M}}_{0,n}$, Brown introduces a differential algebra of generalized multiple polylogarithms that is closed under taking primitives. This enables him to apply Stokes' theorem and reduce the computation of the integral in the statement to an integral on the boundary of the homology class. The irreducible boundary divisors of $\overline{\f{M}}_{0,n}$ are isomorphic to products of the form $\overline{\f{M}}_{0,n_1}\times \overline{\f{M}}_{0,n_2}$ with $n_1+n_2=n-1$. One may therefore argue inductively; the one-dimensional case involves hyperlogarithms over $\mathbb{P}^1\setminus\{0,1,\infty\}$, which finally yield multiple zeta values when evaluated at $1$.  \\
        To handle the case of a general hyperplane arrangement, we also introduce an analogous algebra of generalized multiple polylogarithms that is closed under primitives. We do this by considering the solutions to the holonomy equation associated with $\A$ \cite{De_Concini-Procesi-Hyperplane_arrangements_and_holonomy_equations}, and we prove some specific properties of these functions that are needed in our situation. \\
        The irreducible boundary divisors of $\overline{Y}_\A$ are isomorphic to products of the form $\overline{Y}_{\A_1}\times \overline{Y}_{\A_2}$ for some arrangements $\A_1$, $\A_2$ of smaller dimension, which makes it possible to adapt Brown's argument. However, Stokes' theorem requires the primitive that we have found to have no poles along the boundary of the homology class. To take care of this, one needs a regularization procedure, which is where the assumption of $\A$ having enough modular elements enters the picture. As we will see, the retractions from $\overline{Y}_\A$ onto each irreducible boundary divisor will play a crucial role in this regard and substitute the morphisms between moduli spaces of curves corresponding to forgetting a marked point.\\

        A motivating example of our exposition is the reflection arrangement $\A_{l,n}$ of the full monomial group, whose complement is given by
        \[
            Y_{\A_{l,n}}= \text{Spec}\; \Q(\mu)[t_1,\dots, t_l]
            \left[ \frac{1}{t_i},\frac{1}{t_i-\mu^m}, \frac{1}{t_i-\mu^mt_j} \,\middle\vert\, i\neq j, \; m=1,\dots, n \right],
        \]
        where $\mu$ is an $n$-th root of unity. This arrangement is dear to the literature \cite{Orlik-Terao-Arrangements_of_hyperplanes} due to its large group of symmetries and, for $n=2$, it contains several subarrangements of Coxeter type. $\A_{l,n}$ can be shown to have enough modular elements, so Theorem~\ref{theo: intro periods of arrangements with enough modular elements} enables us to effectively compute its periods.
        \begin{intro}[see Theorem~\ref{theo: full monomial group}]
            \label{theo: intro full monomial group}
            Period integrals of $\overline{Y}_{\A_{l,n}}$ relative to the boundary divisor are $\Q(\mu) [2\pi i]$- linear combinations of multiple polylogarithmic values of weight at most $l$, which are numbers of the form
            \[
                \sum_{1\le k_1<\dots<k_r} \frac{\mu^{k_1n_1+\dots +k_rn_r}}{k_1^{s_1}\dots k_r^{s_r}}
            \]
            for integers $n_i,s_i\ge 1$ with $s_r\ge 2$, $s_1+\dots +s_r\le l$.
        \end{intro}
        \noindent
        Multiple polylogarithmic values include special values of Dirichlet $L$-functions and of the Hurwitz zeta function. Some recent irrationality proofs \cites{Fischler-Sprang-Zudilin-Many-odd-zeta_values_are_irrational, Lai-Yu-A_note_on_the_number_of_irrational_odd_zeta_values} construct linear combinations of these numbers with related coefficients, then combine them in a clever way to obtain linear forms in zeta values only. Linking these irrationality proofs to the periods of the arrangement $\A_{l,n}$ constitutes the first step towards a geometric interpretation of these more recent results. \\

        We now elaborate on further possible developments motivated by the results of this work. It would be interesting to understand more about the properties characterizing hyperplane arrangements with enough modular elements. At the moment, for example, it is not clear when an arrangement can be embedded into one with enough modular elements. For such arrangements, the conclusion of Theorem~\ref{theo: intro periods of arrangements with enough modular elements} still holds by pulling integrals back to the larger arrangement. \\
        Moreover, arrangements with enough modular elements may provide a good framework for extending the work of Brown \cite{Brown-Irrationality_proofs_for_zeta_values_moduli_spaces_and_dinner_parties} on cellular integrals over moduli spaces of curves to more general arrangements, with possible implications to irrationality proofs. In particular, we will explore the case of the reflection arrangement of the full monomial group in upcoming work.\\
        There are also some motivic implications of Theorem~\ref{theo: intro periods of arrangements with enough modular elements} worth investigating. Assume now that $k$ is a number field. The motives associated with relative cohomology groups of $\overline{Y}_\A$ for a hyperplane arrangement $\A$ lie in the category $\text{MT}(k)$ of mixed Tate motives over $k$ \cite{Levine-Tate_motives_and_the_vanishing_conjectures_for_algebraic_K_theory}. 
        Deligne and Goncharov \cite{Deligne-Goncharov-Groupes_fondamentaux_motiviques_de_Tate_mixte} have defined a subcategory of $\text{MT}(k)$, the periods of whose objects are generated by values of multiple polylogarithms. This subcategory is generated by the motivic avatar of the pro-unipotent completion of the fundamental groupoid of $\mathbb{P}^1$ with suitable points removed; in general, it does not coincide with the whole $\text{MT}(k)$. Some candidates for the missing generators have been found in motivic versions of Aomoto polylogarithms \cite{Beilinson-Varchenko-Goncharov-Shekhtman-Projective_geometry_and_K_theory}. \\
        Theorem~\ref{theo: intro periods of arrangements with enough modular elements} seems to suggest that, when $\A$ has enough modular elements, these motives actually belong to the subcategory defined by Deligne and Goncharov. Giving a motivic version of this result may enable us to have a combinatorial criterion for which an arrangement lies in this smaller category of motives.\\

        The paper is structured as follows. The first part introduces hyperplane arrangements with enough modular elements and their basic properties. We first set up the notation and recall the geometry of compactifications of complements of hyperplane arrangements. Then, we turn to arrangements with enough modular elements and construct the retractions mentioned above. \\
        In the second part, we start by recalling the holonomy equation and then we prove several properties of its solutions. We use the condition of having enough modular elements to establish a regularization procedure for primitives, then conclude with the proof of Theorem~\ref{theo: intro periods of arrangements with enough modular elements}.
        In the end, we sketch how to deduce Theorem~\ref{theo: intro full monomial group} from Theorem~\ref{theo: intro periods of arrangements with enough modular elements}. For the sake of the length of the present exposition, we defer further investigations of the periods of the reflection arrangement of the full monomial group to future work. \\

        \noindent
        \textbf{Acknowledgements.} The author expresses his heartfelt gratitude to Johannes Sprang for his careful supervision, unceasing encouragement, and constant guidance. The author also thanks Clément Dupont for his valuable suggestions and Jochen Heinloth for several helpful discussions. Moreover, the author is grateful for the stimulating work environment at the University of Duisburg-Essen and expresses his gratitude to all his friends and colleagues for their precious support. \\
        The financial support of the DFG Research Training Group 2553 \emph{Symmetries and classifying spaces: analytic, arithmetic and derived} is gratefully acknowledged.

    \section{Hyperplane arrangements and modular elements}
        
    We start by recalling some key facts in the theory of hyperplane arrangements. We also introduce the main class of hyperplane arrangements whose associated periods will be computed in Theorem~\ref{theo: periods of hyperplane arrangements with enough modular elements}.\\
    In the first section, we recall the definition of a hyperplane arrangement and some basic combinatorial invariants attached to it. We also introduce the De Concini-Procesi compactification of the complement of an arrangement of hyperplanes and explain some crucial aspects of its geometry. In the second section, we focus on supersolvable arrangements. The main references for this part are \cites{Orlik-Terao-Arrangements_of_hyperplanes, De_Concini-Procesi-Wonderful_models_of_subspace_arrangements}.\\
    In the last two sections, we introduce a new class of arrangements, which we call \emph{with enough modular elements}. More precisely, in the third section we define these arrangements, listing their basic properties and a few simple examples. In the last section, we illustrate the geometric implications of the condition of having enough modular elements on the compactification of the complement of an arrangement. 

    \subsection{Compactifications of complements of hyperplane arrangements}
        
    \subsubsection{Hyperplane arrangements} 
    Let $k$ be a subfield of $\C$.
    \begin{definition}
        A \emph{hyperplane arrangement} is a pair $(V,\A)$ where $V$ is a finite-dimensional $k$-vector space and $\A$ is a non-empty finite set of one-dimensional linear subspaces of the dual space $V^*$. If $\dim V=l$, we say that $(V,\A)$ is a \emph{hyperplane $l$-arrangement}.
    \end{definition}
    \noindent
    By abuse of notation, we will at times drop $V$ from the notation when the ambient vector space is clear from the context. Given a hyperplane arrangement $(V,\A)$, for all subspaces $H$ of $V^*$ we denote by $H^\perp$ the kernel of $H$; in particular, if $H\in \A$, then $H^\perp$ is a hyperplane in $V$.\\
    A morphism of hyperplane arrangements $f\colon (V,\A)\to (V',\A')$ is a linear homomorphism $f\colon V\to V'$ such that for all $H\in \A$ the image of $H^\perp$ via $f$ is contained in $K^\perp$ for some $K\in H'$.
    \begin{remark}
        In the literature, hyperplane arrangements are typically defined by a finite set of affine hyperplanes $\A$ of an affine space $V$. The arrangement $(V,\A)$ is then said to be \emph{central} if all elements of $\A$ are linear subspaces of $V$, so if they contain the origin of $V$. With our definition, all hyperplane arrangements are central according to this typical definition in the literature. When considering complements of hyperplane arrangements, we will recover non-central hyperplane arrangements by looking at the complement of a central hyperplane arrangement in a projective space, as we will explain later.
    \end{remark}
    \begin{definition}
        Let $(V,\A)$ be a hyperplane arrangement. The \emph{intersection lattice} $L(\A)$ of $\A$ is the set of all subspaces of $V^*$ that are sums of elements of $\A$. We agree that the zero subspace belongs to $L(\A)$.\\
        The set $L(\A)$ carries a partial order induced by set-theoretic inclusion.
    \end{definition}
    \noindent
    The reason for this terminology for $L(\A)$ is that this set can be endowed with a structure of a geometric lattice \cite[Lemma 2.3]{Orlik-Terao-Arrangements_of_hyperplanes}. This will not play a significant role in the exposition, as we will mainly focus on the structure of partially ordered set of $L(\A)$.\\
    As is common in the literature, if $\A$ is defined as a set of hyperplanes in $V$, one lets $L(\A)$ be the set of all intersections of the hyperplanes in $\A$, ordered by reverse inclusion. This is thus dual to our definition of $L(\A)$, so one obtains isomorphic partially ordered sets anyway. Our convention agrees with that of \cite{De_Concini-Procesi-Wonderful_models_of_subspace_arrangements}. This should allow for a more natural description of nested sets and building families later in the discussion.  
    \begin{definition}
        A hyperplane arrangement $(V,\A)$ is \emph{essential} if $V^*\in L(\A)$.
    \end{definition}
    \noindent
    Almost all hyperplane arrangements considered in this work are essential. If $(V,\A)$ is not essential, let $W$ be the sum of all lines in $\A$. One then obtains an induced essential arrangement on $V/W^\perp$. \\
    There are two main operations that can be performed starting from a hyperplane arrangement $(V,\A)$ to obtain more arrangements. We collect them in the following definition.
    \begin{definition}
        Let $(V,\A)$ be a hyperplane arrangement and fix $X\in L(\A)$.
        \begin{enumerate}
            \item The \emph{restriction arrangement} of $(V,\A)$ to $X$ is defined as the hyperplane arrangement $(X^\perp, \A/X)$ over $X^\perp$ with
            \[
                \A/X=\{ (H+X)/X\subseteq V^*/X \mid H\in \A, H\not\subseteq X\},
            \]
            where we identify $(X^\perp)^*$ with $V^*/X$.
            \item The \emph{quotient arrangement} of $(V,\A)$ modulo $X$ is defined as the hyperplane arrangement $(V/X^\perp, \A\vert_X)$ over $V/X^\perp$ with 
            \[
                \A\vert_X=\left\{ H\in \A \mid H\subseteq X \right\},
            \]
            where we identify $(V/X^\perp)^*$ with $X$.
        \end{enumerate}
    \end{definition}
    \noindent
    For reference, the arrangement $(X^\perp, \A/X)$ is denoted by $\A^{X^\perp}$ in \cite{Orlik-Terao-Arrangements_of_hyperplanes}. On the other hand, the so-called \emph{deletion arrangement} $\A_{X^\perp}$ of \cite{Orlik-Terao-Arrangements_of_hyperplanes} corresponds in our notation to $(V,\A\vert_X)$. We prefer to consider $(V/X^\perp, \A\vert_X)$ to have an essential arrangement. Throughout our exposition, we will sometimes loosely refer to $(V/X^\perp, \A\vert_X)$ with the term \emph{deletion arrangement} in any case. Notice that 
    \[
        L(\A/X)=\{ (Y+X)/X\subseteq V^*/X \mid Y\in L(\A)\}, \qquad L(\A\vert_X)=\{Y\in L(\A)\mid Y\subseteq X\}.
    \]
   For these and similar properties, we refer to \cite[Lemma 2.11]{Orlik-Terao-Arrangements_of_hyperplanes}.\\
    The main geometric objects of study connected with a hyperplane arrangement are the complement thereof in the affine and projective space.
    \begin{definition}
        Let $(V,\A)$ be a hyperplane arrangement.
        \begin{enumerate}
            \item The complement of $(V,\A)$ in the affine space is denoted by 
            \[
                X(V,\A)= V \setminus \bigcup_{H\in \A} H^\perp.
            \]
            \item The complement of $(V,\A)$ in the projective space is denoted by
            \[
                Y(V,\A)= \mathbb{P}(V)\setminus \bigcup_{H\in \A} \mathbb{P}(H^\perp).
            \]
        \end{enumerate}
    \end{definition}
    \noindent
    Since $\A$ is not empty, both $X(V,\A)$ and $Y(V,\A)$ are affine varieties over $k$. If $\dim V=l+1$, then $X(V,\A)$ has dimension $l+1$, while $Y(V,\A)$ has dimension $l$. Notice that, if we fix an element of $\A$ as a hyperplane at infinity, $Y(V,\A)$ is the complement in the affine space of a finite set of affine hyperplanes. One can therefore view $Y(V,\A)$ as the complement of what is usually defined in the literature as a non-central arrangement.\\
    As a notational remark, the global sections of the structure sheaf of an affine variety $X$ over $k$ shall be denoted by $\mathcal{O}_X$, while those of the sheaf of Kähler differentials by $\Omega^1_{X/k}$. We take this decision to simplify the notation when working with the affine varieties $X(V,\A)$ and $Y(V,\A)$.\\
    In the sequel, we will focus mainly on the study of the projective complement $Y(V,\A)$ because the geometry of its De Concini-Procesi compactification, which will be explained later, is somewhat more symmetric. One can still realize $X(V,\A)$ as the projective complement $Y(\widetilde{V}, \widetilde{\A})$ of a hyperplane arrangement $(\widetilde{V}, \widetilde{\A})$, as we now explain.
    \begin{definition} $ $
        \begin{enumerate}
            \item Let $(V,\A)$ be a hyperplane arrangement. The \emph{cone over $(V,\A)$} is the hyperplane arrangement $(\widetilde{V}, \widetilde{\A})$ given by the product of $(V,\A)$ with a one-dimensional arrangement. This means that there is a one-dimensional $k$-vector space $L$ such that $\widetilde{V}=V\oplus L$ and $\widetilde{\A}=\A\sqcup \{L^*\}$.
            \item We say that a hyperplane arrangement $(\widetilde{V}, \widetilde{\A})$ is \emph{central} if it is the cone over some hyperplane arrangement.
        \end{enumerate}
    \end{definition}
    \noindent
    If $(\widetilde{V}, \widetilde{\A})$ is the cone over $(V,\A)$, then $X(V,\A)=Y(\widetilde{V}, \widetilde{\A})$.
    \begin{remark}
        We point out once again the difference of our notion of central arrangement with the usual one given in the literature. To recover the latter from our definition, one should always think of the complement of the arrangement in the projective space, and not in the affine one.
    \end{remark}
    \noindent
    Our next goal is to describe the explicit compactification of $Y(V,\A)$ via a simple normal crossings divisor constructed in \cite{De_Concini-Procesi-Wonderful_models_of_subspace_arrangements}. In short, one first constructs a smooth quasi-projective variety $\overline X$ in which $X(V,\A)$ embeds with complement given by a simple normal crossings divisor. Then, a certain proper closed subvariety $\overline Y$ of $\overline X$ is shown to be the desired compactification for $Y(V,\A)$. In order to obtain $\overline X$, one performs a chain of blow-ups of $\mathbb{A}^{l+1}$ with centers the successive strict transforms of a certain subset $\G$ of $L(\A)$. To make sense of this, we start by describing what subsets of $L(\A)$ will play this role; these will lead to the notion of a \emph{building set}. \\
    We first need some preliminary combinatorial concepts. We follow the conventions of \cite{De_Concini-Procesi-Wonderful_models_of_subspace_arrangements}. Fix a hyperplane arrangement $(V,\A)$.
    \begin{definition}
        A \emph{decomposition} of $X\in L(\A)$ is a subset $\{X_1,\dots, X_r\}$ of $L(\A)\setminus\{ 0\}$ with $X=X_1\oplus \dots \oplus X_r$ such that for all $Y\in L(\A)$ with $Y\subseteq X$ we have $Y\cap X_i\in L(\A)$ for all $i=1,\dots, r$ and
        \[
        Y=(X_1\cap Y)\oplus \dots \oplus (X_r\cap Y).
        \]
        $X$ is called \emph{irreducible} if $\{X\}$ is the only decomposition of $X$.
    \end{definition}
    \noindent
    In the case where $\{X_1,\dots, X_r\}$ is a decomposition of $X$, we will often loosely write that $X=X_1\oplus \dots \oplus X_r$ is a decomposition of $X$. Every element of $\A$ is irreducible and we will typically assume that $V^*$ is irreducible. The set of irreducible elements of $L(\A)$ is denoted by $\F(\A)$.
    \begin{lemma}[{\cite[Section 2.1]{De_Concini-Procesi-Wonderful_models_of_subspace_arrangements}}]
        $ $
        \begin{enumerate}
            \item Every $X\in L(\A)$ admits a unique decomposition $X=X_1\oplus \dots \oplus X_r$ into irreducible elements.
            \item If $X,Y\in \F(\A)$, then either $ X+Y\in \F(\A)$ or $X\oplus Y$ is a decomposition.
        \end{enumerate}
    \end{lemma}
    \noindent
    The family $\F(\A)$ of irreducible elements is the first example of subset of $L(\A)$ which provides suitable centers for a chain of blow-ups on $\mathbb{A}^{l+1}$ to obtain a variety $\overline X$ in which $X(V,\A)$ embeds with boundary given by a simple normal crossings divisor. The more general subsets of $L(\A)$ that we can allow for this purpose are introduced in the next definition.
    \begin{definition}
        A subset $\G\subseteq L(\A)\setminus \{0\}$ is called a \emph{building set} for $L(\A)$ if the following two conditions are satisfied:
        \begin{enumerate}
            \item $\F(\A)\subseteq \G$;
            \item if $X,Y\in \G$ and $X+Y\notin \G$, then $X\oplus Y$ is a decomposition.
        \end{enumerate}
        Equivalently, $\G$ is a building set for $L(\A)$ if every $X\in L(\A)$ is the direct sum of the set of maximal elements of $\G$ contained in $X$. 
    \end{definition}
    \noindent
    The equivalence of these two definitions requires some justification, for which we refer to \cite[Theorem 2.3]{De_Concini-Procesi-Wonderful_models_of_subspace_arrangements}. If $\G$ is a building set and $X\in L(\A)$, the maximal elements of $\G$ contained in $X$ are a decomposition for $X$, called the \emph{$\G$-decomposition} of $X$.\\
    Notice that both $\F(\A)$ and $L(\A)\setminus \{0\}$ are building sets. These are, respectively, the minimal and maximal building sets with respect to set-theoretic inclusion. Building sets behave well with respect to intersections, restrictions and quotients, as the following lemmas clarify.
    \begin{lemma}
        Let $(V,\A)$ be a hyperplane arrangement and $\G_1, \G_2$ building sets for $L(\A)$. Then $\G_1\cap \G_2$ is a building set for $L(\A)$.
        In particular, for all $X, Y\in L(\A)$ there is a minimal building set for $L(\A)$ containing both $X$ and $Y$. 
    \end{lemma}
    \begin{proof}
        It is clear that $\G_1\cap \G_2$ contains the set of irreducible elements of $L(\A)$. Take $X, Y\in \G_1\cap\G_2$ with $X+Y\not\in \G_1\cap \G_2$. Then $X+Y\not\in \G_i$ for some $i\in \{1,2\}$. Since $\G_i$ is a building set containing $X$ and $Y$, it follows that $X\oplus Y$ is a decomposition. 
    \end{proof}
    \begin{lemma}[{\cite[Lemma 4.3]{De_Concini-Procesi-Wonderful_models_of_subspace_arrangements}}]
        Let $\G$ be a building set for $L(\A)$ and fix $X\in \G$. 
        \begin{enumerate}
            \item The set
            \[
                \G/X=\{ (Y+X)/X\subseteq V^*/X \mid Y\in \G\}\subseteq L(\A/X)
            \]
            is a building set for $L(\A/X)$.
            \item The set
            \[
                \G\vert_X=\{ Y\in\G \mid Y\subseteq X \}\subseteq L(\A\vert_X)
            \]
            is a building set for $L(\A\vert_X)$. Also, if $\G=\F(\A)$, then $\G\vert_X=\F(\A\vert_X)$.
        \end{enumerate}
    \end{lemma}
    \begin{example}
        \label{example: irreducibles do not pass to quotients}
        If $\G=\F(\A)$, it is not true in general that $\G/X=\F(\A/X)$. For example, consider $V^*=\Q x\oplus \Q y\oplus \Q z$ with arrangement 
        \[
            \A=\{ \langle x\rangle, \langle y\rangle, \langle z\rangle, \langle x+y\rangle, \langle x+z\rangle\}.
        \]
        Irreducible elements of $L(\A)$ of dimension $2$ are planes which contain at least three distinct lines of $\A$. These are therefore $\langle x, y\rangle$ and $\langle x,z\rangle$. In dimension three, we can check that $V^*$ is irreducible. It is clear that $V^*$ does not admit a decomposition as a sum of three lines, because $\#\A> 3$. Moreover, the direct sums $\langle x,y\rangle\oplus \langle z\rangle$ and $\langle x,y\rangle\oplus \langle x+z\rangle$ are not decompositions, as can be seen by intersecting with $\langle x+ z\rangle $ and $\langle z\rangle$ respectively. Similarly, $\langle x,z\rangle\oplus \langle y\rangle$ and $\langle x,z\rangle\oplus \langle x+y\rangle$ are not decompositions either.\\
        Consider the arrangement $\A/\langle x\rangle$. It is apparent that $\F(\A)/\langle x\rangle$ equals the whole $L(\A/\langle x\rangle)$. However, $\A/\langle x\rangle$ consists of the two lines $\langle x,y \rangle/\langle x\rangle$ and $\langle x,z\rangle/\langle x\rangle$ only, so $V^*/\langle x\rangle$ is not irreducible in $L(\A/\langle x\rangle)$.
    \end{example}
    \noindent
    A building set for $L(\A)$ can be extended to the cone over $(V,\A)$:
    \begin{lemma}
        \label{lemma: building sets for cones}
        Let $(V,\A)$ be an arrangement, $\G$ a building set for $L(\A)$. Let $(\widetilde{V}, \widetilde{\A})$ be the cone over $(V,\A)$ with $\widetilde{V}=V\oplus L$. Then $\widetilde{\G}=\G\cup\{X+L^*\mid X\in \G\}$ is a building set for $L(\widetilde{\A})$.
        \begin{proof}
            It is straightforward to see that for all $X\in L(\A)$ the sum $X\oplus L^*$ is a decomposition. The family of irreducible elements of $L(\widetilde{\A})$ is, therefore, $\F(\A)\cup \{L^*\}$. The claim follows at once. 
        \end{proof}
    \end{lemma}

    \subsubsection{De Concini-Procesi compactifications}
    We may now introduce the main geometric construction of this section; we shall follow the exposition of \cite{De_Concini-Procesi-Wonderful_models_of_subspace_arrangements}. Let $(V,\A)$ be an essential hyperplane arrangement and fix a building set $\G$ for $L(\A)$. For all $X\in L(\A)$, the quotient map $V\to V/X^\perp$ induces a rational morphism $V\dashrightarrow \mathbb{P}(V/X^\perp)$, which restricts to a regular morphism on $V\setminus X^\perp$. Thus, for all $X\in L(\A)$ we have a regular map $\phi_X \colon X(V,\A)\to \mathbb{P}(V/X^\perp)$; these maps induce a morphism 
    \[
        \phi: X(V,\A)\to \prod_{X\in \G} \mathbb{P}(V/X^\perp).
    \]
    The graph of $\phi$ is a closed subvariety of $X(V,\A)\times \prod_{X\in \G} \mathbb{P}(V/X^\perp)$, which is an open subset of $V\times \prod_{X\in \G} \mathbb{P}(V/X^\perp)$. We then define $\overline X(V,\A,\G)$ to be the closure of $X(V,\A)$ in $V\times \prod_{X\in \G} \mathbb{P}(V/X^\perp)$ via this embedding through the graph of $\phi$.
    \begin{theorem}[{\cite[Proposition 1.5]{De_Concini-Procesi-Wonderful_models_of_subspace_arrangements}}]
        The variety $\overline{X}(V,\A,\G)$ is smooth and the complement of $X(V,\A)$ in $\overline{X}(V,\A,\G)$ is a simple normal crossings divisor. The irreducible components of this complement are in bijection with the elements of $\G$. Moreover, every intersection of these irreducible components is irreducible.
    \end{theorem}
    \noindent
    Given $X\in \G$, we denote by $D'_X$ the irreducible boundary divisor of $\overline{X}(V,\A,\G)$ corresponding to $X$.
    \begin{remark}
        The variety $\overline X(V,\A,\G)$ can also be realized as a sequence of blow-ups of $\mathbb{A}^{l+1}$ by blowing up the successive strict transforms of the closed subschemes $X^\perp$ for $X\in \G$. There is some freedom in the order of the blow-up, and it is possible to proceed by increasing dimension. This description is useful in cohomological computations. 
    \end{remark}
    \noindent
    A similar construction can be done for the projective complement $Y(V,\A)$. In this case, we let $\overline Y(V,\A,\G)$ be the closure inside $\mathbb{P}(V) \times \prod_{X\in \G} \mathbb{P}(V/X^\perp)$ of the graph of the map
    \[
        Y(V,\A)\to \prod_{X\in \G} \mathbb{P}(V/X^\perp)
    \]
    defined similarly to $\phi$ above.\\
    In this case, since $(V^*)^\perp=0$, the first factor $\mathbb{P}(V)$ in $\mathbb{P}(V) \times \prod_{X\in \G} \mathbb{P}(V/X^\perp)$ appears repeated also in $\prod_{X\in \G}\mathbb{P}(V/X^\perp)$ if $V^*\in \G$. Hence, $\overline{Y}(V,\A,\G)$ does not depend on whether $V^*$ lies in $\G$ or not. For convenience, we will always assume that $V^*\in \G$.
    \begin{theorem}[{\cite[Theorem 4.1]{De_Concini-Procesi-Wonderful_models_of_subspace_arrangements}}]
        $ $
        \begin{enumerate}
            \item The variety $\overline{Y}(V,\A,\G)$ is smooth and projective. Moreover, the complement of $Y(V,\A)$ in $\overline{Y}(V,\A,\G)$ is a simple normal crossings divisor.
            \item The irreducible components of $\overline{Y}(V,\A,\G)\setminus Y(V,\A)$ are in bijection with the elements of $\G\setminus \{V^*\}$. The irreducible component corresponding to $X\in \G\setminus\{V^*\}$ will be denoted by $D_X$.
            \item The variety $\overline{Y}(V,\A,\G)$ is isomorphic to the irreducible boundary divisor $D_{V^*}'$ of $\overline{X}(V,\A,\G)$. Under this isomorphism, for all $X\in \G\setminus \{V^*\}$ the divisor $D_X$ corresponds to $D_{V^*}'\cap D_{X}'$. In particular, intersections of irreducible boundary divisors of $\overline{Y}(V,\A,\G)$ are irreducible. 
        \end{enumerate}
    \end{theorem}
    \begin{remark}
        It is possible to embed $X(V,\A)$ in a projective variety of the above kind by considering the cone over $(V,\A)$.
    \end{remark}
    \noindent
    It is particularly important to observe that the irreducible boundary divisors $D_X$ for $X\in \G$ are themselves compactifications of complements of hyperplane arrangements, as the following result shows.
    \begin{proposition}[{\cite[Theorem 4.3]{De_Concini-Procesi-Wonderful_models_of_subspace_arrangements}}]
        \label{prop: buondary divisor as product of compactifications}
        Let $\G$ be a building set for $(V,\A)$ and fix $X\in \G$. Then there is an isomorphism
        \[
            D_X \cong \overline{Y}(V/X^\perp, \A\vert_X, \G\vert_X) \times \overline{Y}(X^\perp, \A/X, \G/X).
        \]
    \end{proposition}
    \noindent
    A similar description is also there for the irreducible boundary divisor $D_X'$ of $\overline X(V,\A,\G)$, although it involves compactifications of both affine and projective complements of hyperplane arrangements. For these reasons, we prefer to work only with $\overline Y(V,\A,\G)$.\\
    Next, we focus on a combinatorial description of the intersection theory of the boundary divisor $\overline Y(V,\A,\G)\setminus Y(V,\A)$. To describe when some of its irreducible components have non-empty intersections, we recall the notion of nested set from \cite{De_Concini-Procesi-Wonderful_models_of_subspace_arrangements}.
    \noindent
    \begin{definition}
        Let $\G$ be a building set for $L(\A)$. A subset $\S$ of $\G$ is said to be \emph{$\G$-nested} if for all $X_1,\dots, X_r\in \S$ pairwise non-comparable with $r\ge 2$ we have $\sum_{i=1}^r X_i\notin \G$.
    \end{definition}
    \noindent
    We summarize the main properties of $\G$-nested sets. All of these can be found in \cite{De_Concini-Procesi-Wonderful_models_of_subspace_arrangements}.
    \begin{lemma}
        Let $\S$ be a $\G$-nested set.
        \begin{enumerate}
            \item Let $X_1,\dots, X_r\in \S$ be pairwise incomparable. Then $X=X_1\oplus \dots \oplus X_r$ is the $\G$-decomposition of $X$. 
            \item For all $v\in V^*$ the set $\{Y\in \S \mid v\in Y\}\cup \{V^*\}$ is linearly ordered; we denote by $p_\S(v)$ its minimal element.\\
            In particular, for all $X\in \S$ the set $\{Y\in \S \mid X\subsetneq Y\}$ is linearly ordered with respect to inclusion. If $X\neq V^*$, we denote by $X^+$ its minimal element.
            \item $\S$ is maximal with respect to set-theoretic inclusion of $\G$-nested sets if and only if $\#\S=\dim V$. Moreover, every $\G$-nested set can be completed to a maximal one.
        \end{enumerate}
        Furthermore, given $X_1,\dots, X_r\in \G$, the divisors $D_{X_1},\dots,D_{X_r}$ have non-empty intersection if and only if $\{X_1,\dots,X_r\}$ is $\G$-nested.
    \end{lemma}
    \noindent
    In view of this lemma, maximal $\G$-nested sets correspond to maximal intersections of irreducible boundary divisors. Since $\overline Y(V,\A,\G)\setminus Y(V,\A)$ is stratified by locally closed subsets corresponding to intersections of its irreducible components, its zero-dimensional stratum is in bijection with the set of maximal $\G$-nested sets. 
    \begin{lemma}
        \label{lemma: restrictions and quotients of maximal nested sets}
        Let $\S$ be a maximal $\G$-nested set and fix $X\in \S$.
        \begin{enumerate}
            \item $\S\vert_X=\{Y\in \S\mid Y\subseteq X\}$ is a maximal $\G\vert_X$-nested set.
            \item $\S/X=\{ (Y+X)/X\mid Y\in \S, Y\not\subseteq X \}$ is a maximal $\G/X$-nested set.
        \end{enumerate}
    \end{lemma}
    \begin{proof}
        For the first point, since a $\G$-decomposition of an element in $\G\vert_X$ is also a $\G\vert_X$-decomposition, it follows that $\S\vert_X$ is $\G\vert_X$-nested. For maximality, take $Y\in \G\vert_X$ such that $\S\vert_X\cup\{Y\}$ is $\G\vert_X$-nested. Let us show that $\S\cup\{Y\}$ is $\G$-nested.\\
        Given $Z_1,\dots, Z_n\in \S$ pairwise incomparable together with $Y$, we need to show that $Y+Z_1+\dots +Z_n\notin \G$. If $Z_1,\dots, Z_n$ are contained in $X$, this follows from $\S\vert_X\cup\{Y\}$ being $\G\vert_X$-nested. Otherwise, consider $X+Z_1+\dots+Z_n$. Since $\S$ is $\G$-nested, the latter is not in $\G$ and we may write $G_1\oplus \dots \oplus G_r$ with $r\ge 2$ for its $\G$-decomposition. If we assume by contradiction that $Y+Z_1+\dots + Z_n\in \G$, then it is contained in one among $G_1,\dots, G_r$, say $G_1$. Since $X\in \G$ and $Y\subseteq X\cap G_1$, we infer that $X\subseteq G_1$. This would imply that $X+Z_1+\dots+Z_n\subseteq G_1$, hence $r=1$. As $\S\cup \{Y\}$ is $\G$-nested, by maximality of $\S$ we conclude that $Y\in \S\vert_X$.
        For the second point, let $Y_1,\dots, Y_n$ be elements of $\S$ not contained in $X$ and such that $Y_1+X,\dots, Y_n+X$ are pairwise incomparable. Since $\S$ is $\G$-nested, it follows that $Y_1+\dots +Y_n+X\notin \G$, as this can still be written as a sum of pairwise incomparable elements of $\S$. This implies that $(Y_1+\dots+ Y_n+X)/X\not\in \G/X$. Indeed, suppose that $Y_1+\dots+Y_n+X=Y+X$ for some $Y\in \G$. Since $Y+X\not\in \G$, we deduce that $Y\oplus X$ is a $\G$-decomposition. Each $Y_i\in \G$ is not contained in $X$, hence $Y_i\subseteq Y$. It follows that $Y_1+\dots+Y_n=Y\in \G$ and no $Y_i$ contains $X$. As a result, $Y_1,\dots, Y_n$ are pairwise incomparable, so the fact that $\S$ is $\G$-nested implies that $Y\not\in \G$, a contradiction.\\
        For maximality, take $Y\in \G$ such that $\S/X\cup \{(Y+X)/X\}$ is $\G/X$-nested. If $Y+X\in \G$, set $Z=Y+X$, otherwise $Z=Y$. By maximality of $\S$, we only need to show that $\S\cup \{Z\}$ is $\G$-nested. Choose $Z_1,\dots, Z_n\in \S$ pairwise incomparable together with $Z$. If $Z_1,\dots, Z_n\subseteq X$, then $Y+X\not\in \G$, otherwise $Z_i\subseteq Z$ for all $i=1,\dots, n$ by the definition of $Z$. It follows that $Y\oplus X$ is a decomposition, hence $Y\oplus (Z_1 + \dots + Z_n)$ is a decomposition as well, so $Z+Z_1+\dots + Z_n\not\in \G$.\\
        Thus, we may assume that $Z_1\subseteq X$, which implies that $Z+X$ and $Z_1+X$ are incomparable. Indeed, suppose that $Z+X\subseteq Z_1+X$. Since $Z_1$ and $X$ belong to $\S$, either $X\subseteq Z_1$ or $Z_1\oplus X$ is a decomposition. In both cases we have $Z\subseteq Z_1$, which contradicts the fact that $Z$ and $Z_1$ are incomparable. The case $Z_1+X\subseteq Z+X$ can be dealt with in an analogous manner using the definition of $Z$. As a result, $Z+Z_1+\dots + Z_n+X$ is the sum of at least two incomparable elements of $\G$ whose quotient modulo $X$ lies in $\S/X\cup \{(Y+X)/X\}$. This implies that $(Z+Z_1+\dots +Z_n+X)/X\not\in \G/X$, so in particular $Z+Z_1+\dots + Z_n\not\in \G$.
    \end{proof}
    \noindent
    We now give an explicit description of $\overline Y(V,\A,\G)$ by local charts. This variety is covered by complements of hypersurfaces in the affine space centered at the points that are maximal intersections of irreducible boundary divisors. To explain this thoroughly, we need one more construction.
    \begin{definition}
        Let $\S$ be a maximal $\G$-nested set. An \emph{adapted basis} for $\S$ is a function $\beta \colon  \S\to V^*$ such that for all $X\in \S$ the set $\{ \beta(Y) \mid Y\in \S, Y\subseteq X\}$ is a basis for $X$.
    \end{definition}
    \begin{lemma}[{\cite[Lemma 1.3]{De_Concini-Procesi-Wonderful_models_of_subspace_arrangements}}]
        Every maximal $\G$-nested set admits an adapted basis $\beta$. Moreover, for all $X\in \S$ we have $p_\S(\beta(X))=X$.
    \end{lemma}
    \noindent
    We will sometimes fix for every $H\in \A$ a non-zero vector $x_H\in H$. In this case, we will tacitly assume that adapted bases take values among these vectors. Thus, an adapted basis will be regarded as a function $\beta \colon \S\to \A$ and the vector corresponding to $X\in \S$ will be denoted by $x_{\beta(X)}$.  \\
    Notice that, if $\beta$ is an adapted basis for a $\G$-nested set $\S$, for all $X\in \S$ there are canonically induced adapted bases $\beta\vert_X$ and $\beta/X$ for $\S\vert_X$ and $\S/X$ respectively.\\
    To find some local charts for $\overline Y(V,\A,\G)$, we start by fixing a maximal $\G$-nested set $\S$ with adapted basis $\beta$. Write $l=\dim V-1$ and consider the morphism
    \[
        \rho \colon  \mathbb{A}^{l+1}=\text{Spec}\,k[u_X\mid X\in \S]\longrightarrow \mathbb{A}^{l+1}=\text{Spec}\, k[\beta(X)\mid X\in \S]
    \]
    which, at the level of regular functions, is defined by
    \[
        \beta(X)\longmapsto \prod_{Y\in \S, X\subseteq Y} u_Y.
    \]
    The map $\rho$ restricts to an isomorphism between the open subsets defined by removing the hyperplanes $u_X=0$ in the source and $\beta(X)=0$ in the target for all $X\in \S$. Its inverse is given by
    \[
        u_X\longmapsto \frac{\beta(X)}{\beta(X^+)},
    \]
    formally setting $\beta((V^*)^+)=1$.
    \begin{lemma}[{\cite[Lemma 3.1]{De_Concini-Procesi-Wonderful_models_of_subspace_arrangements}}]
        Let $H\in \A$ and $x\in H$, $x\neq 0$. Then $x=\beta(p_\S(x)) P^\S_x$, where the image of $P^\S_x$ under $\rho$ is a polynomial that depends only on the coordinates $u_Y$ for $Y\in \S$, $Y\subsetneq p_\S(x)$ and does not vanish at $0$.
    \end{lemma}
    \noindent
    For all $H\in \A$ fix a non-zero vector $x_H\in H$ and write $P^\S_H$ for $P^\S_{x_H}$. Consider the open subset of $\mathbb{A}^{l+1}$ given by
    \[
        U_{\S}'=\text{Spec}\, k[u_X\mid X\in \S]\left[ \left(P^\S_H\right)^{-1} \,\middle\vert H\in \A \right].
    \]
    The morphism $\rho$ induces an isomorphism between the complement in $U_\S'$ of the coordinate hyperplanes $u_X=0$ for $X\in \S$ and $X(V,\A)$.\\
    Given any $X\in \G$, the morphism $\phi_X \colon  X(V,\A)\to \mathbb{P}(V/X^\perp)$ defined above extends to a map $\phi_X \colon  U_\S'\to \mathbb{P}(V/X^\perp)$. In this way, one obtains an embedding $U_\S'\hookrightarrow \overline X(V,\A,\G)$. Notice that, although the definition of $U_\S'$ depends on the choice of an adapted basis for $\S$, its image in $\overline X(V,\A,\G)$ does not.
    \begin{proposition}[{\cite[Proposition 1.5]{De_Concini-Procesi-Wonderful_models_of_subspace_arrangements}}]
        The varieties $U_\S'$ for $\S$ running through all maximal $\G$-nested sets cover $\overline{X}(V,\A,\G)$. For all $X\in\S$ the closure of the subscheme of equation $u_X=0$ in $\overline X(V,\A,\G)$ coincides with $D_X'$.
    \end{proposition}
    \noindent
    A similar description applies to $\overline Y(V,\A,\G)$ by setting to zero the coordinate $u_{V^*}$ in $U_\S'$. Notice that the polynomials $P^\S_H$ do not depend on $u_{V^*}$. We then have a local chart
    \[ 
        U_{\S}=\text{Spec}\, k[u_X\mid X\in \S, X\neq V^*]\left[ \left(P^\S_H\right)^{-1} \,\middle\vert H\in \A \right].
    \]
    \begin{proposition}[{\cite[Section 4.1]{De_Concini-Procesi-Wonderful_models_of_subspace_arrangements}}]
        The varieties $U_\S$ for $\S$ running through all maximal $\G$-nested sets cover $\overline{Y}(V,\A,\G)$. For all $X\in\S\setminus \{V^*\}$ the closure of the subscheme of equation $u_X=0$ in $\overline Y(V,\A,\G)$ coincides with $D_X$.
    \end{proposition}
    \noindent
    If the situation requires a more precise notation, we will denote $U_\S$ by $U_{\S,\beta}$ or $U_\S(V,\A,\G)$. Whenever we speak of the divisor $D_X$ for $X\in \S\setminus \{V^*\}$, we will implicitly assume that $X\neq V^*$ when needed. The same applies when we list the coordinates $u_X$ of $U_\S$. When necessary, these coordinates will be denoted also by $u_X^\S$. \\
    The isomorphism of Proposition~\ref{prop: buondary divisor as product of compactifications} can be already described at the level of the local charts $U_\S$:
    \begin{proposition}
        \label{prop: local decomposition of boundary divisors}
        Let $\S$ be a maximal $\G$-nested set and fix $X\in \S$. Then there is an isomorphism
        \[
        \{u_X^\S=0\} \cong U_{\S\vert_X}(V/X^\perp, \A\vert_X, \G\vert_X)\times U_{ \S/X}(X^\perp, \A/X,\G/X).
        \]
    \end{proposition}
    \begin{proof}
        We consider the morphism 
        \[
            f \colon U_{\S\vert_X}(V/X^\perp, \A\vert_X, \G\vert_X)\times U_{ \S/X}(X^\perp, \A/X,\G/X) \longrightarrow  \{u_X^\S=0\}
        \]
        which, at the level of regular functions, is defined by
        \[
            u^\S_Y \longmapsto 
            \begin{cases}
                u_Y^{\S\vert_X} & \text{if $Y\subsetneq X$}; \\
                u_{(Y+X)/X}^{\S/X} & \text{if $Y\not\subseteq X$}.
            \end{cases}
        \]
        For all $H\in \A$ with $H\subseteq X$ the polynomial $P^\S_H$ depends only on the variables $u_Y^\S$ with $Y\subseteq X$, since $p_\S(H)\subseteq X$. It follows that the image of $P^\S_H$ under $f$ equals $P^{\S\vert_X}_H$.\\
        Take now $H\in \A$, $H\notin X$, so $p_\S(H)\not\subseteq X$. Write 
        \[
            x_H=\sum_{Y\in \S} c_Y x_{\beta(Y)}
        \]
        for some uniquely determined $c_Y\in k$. We then have
        \[
            x_{(H+X)/X}=\sum_{Y\in \S,Y\not\subseteq X} c_Y x_{\beta((Y+X)/X)},
        \]
        because $\beta(Y)\in Y\subseteq X$ for all $Y\in \S$, $Y\subseteq X$. On the other hand, since $x_\beta(Y)=\prod_{Z\in \S, Y\subseteq Z} u_Y^\S$, setting to zero $u_X^\S$ in $x_H$ has the effect of setting to zero all $x_{\beta(Y)}$ with $Y\subseteq X$. This implies that the image of $P^\S_H\vert_{u_X^\S=0}$ via $f$ equals $P^{\S/X}_{(H+X)/X}$. In particular, the restriction of $P^\S_H$ to $u_X^\S=0$ does not depend on the choice of a representative of $(H+X)/X$ in $\A$. Having proven this, it follows immediately that $f$ is well-defined and an isomorphism.
    \end{proof}

    \subsection{Examples}
        
    Let us spend a few words on a few key examples for our exposition. Given $l\ge 1$, let $V$ be an $(l+1)$-dimensional $\Q$-vector space with basis $x_1,\dots, x_{l+1}$.
    \begin{definition}
        The \emph{braid arrangement} is the arrangement $(V,\A_l)$ with
        \[
            \A_l=\{ \langle x_i\rangle, \langle x_i-x_j\rangle \mid i,j=1,\dots, l+1, \; i\neq j\}. 
        \]
    \end{definition}
    \noindent 
    What we define here as braid arrangement is the essentialization of the homonymous arrangement in \cite{Orlik-Terao-Arrangements_of_hyperplanes}. In affine coordinates $t_i=\frac{x_i}{x_{l+1}}$, the variety $Y(V,\A_l)$ is the complement in $\mathbb{A}^l$ of the hyperplanes of equation $t_i=0$, $t_i=1$, $t_i=t_j$. This scheme can be defined over $\Z$ and it is isomorphic to the moduli space $\f{M}_{0,l+3}$ of smooth projective curves of genus zero with $l+3$ marked points. This coincides with the quotient of $(\mathbb{P}^1)^{l+3}$ without the big diagonal modulo the diagonal action of $\text{PSL}_2$. For a more detailed description of this moduli space, we refer to \cite{Brown-Multiple_zeta_values_and_periods_of_moduli_spaces_of_curves}.\\
    Let us describe $L(\A_l)$. Let $P_{l+1}$ be the set of subsets of $\{0,\dots,l+1\}$ with cardinality at least $2$. Consider the map $f \colon P_{l+1}\to L(\A_l)\setminus\{0\}$ which sends $\lambda \in P_{l+1}$ to
    \[
    f(\lambda)=
    \begin{cases}
        \langle x_i\mid i\in \lambda\rangle & \text{if $0\in \lambda$}; \\
        \langle x_i-x_j \mid i,j\in \lambda,\; i\neq j\rangle & \text{if $0\notin\lambda$}.
    \end{cases}
    \]
    The image of $f$ coincides with the set of irreducible elements of $L(\A_l)\setminus\{0\}$. It is not difficult to see that $f$ induces an order-preserving bijection between $P_{l+1}$ and $\F(\A_l)\setminus \{0\}$, where we order $P_{l+1}$ by set-theoretic inclusion. \\
    For a subset $\lambda\subseteq \{1,\dots, l+1\}$, define $X(\lambda)=f(\lambda\cup \{0\})$ and, if $\#\lambda\ge 2$, set $Y(\lambda)=f(\lambda)$. Given $\Lambda\subseteq \{1,\dots, l+1\}$, let $\{\lambda_1,\dots, \lambda_r\}$  be a partition of $\Lambda$ and $\{\lambda_1',\dots, \lambda_s'\}$ be a partition of $\{ 1,\dots, l+1\}\setminus \Lambda$. Then the direct sum
    \[
        X(\lambda_1)\oplus \dots \oplus X(\lambda_r)\oplus Y(\lambda'_1)\oplus \dots \oplus Y(\lambda_s')
    \]
    is a decomposition, and all elements of $L(\A_l)$ may be written in this form for some $\Lambda\subseteq \{1,\dots, l+1\}$ and suitable partitions $\{\lambda_1,\dots, \lambda_r\}$, $\{\lambda_1',\dots, \lambda_s'\}$. \\
    Given $X\in L(\A)$ of dimension $n$, the arrangements $\A_l\vert_X$ and $\A_l/X$ are isomorphic to $\A_{n-1}$ and $\A_{l-n}$ respectively. Moreover, the building set $\F(\A_l)/X$ coincides with the family of irreducible elements $\F(\A_{l-n})$ of $L(\A_{l-n})$.\\
    The compactification $\overline Y(V,\A_l, \F(\A_l))$ of $Y(V,\A_l)$ is isomorphic to the moduli space $\overline{\f{M}}_{0,l+3}$ of stable projective curves of genus $0$ with $l+3$ marked points. The irreducible boundary divisor $D_X$ for $X\in \F(\A_l)$, $\dim X=n$, can be written as
    \[
        D_X\cong \overline{Y}(V/X, \A_{n-1}, \F(\A_{n-1}))\times \overline{Y}(X^\perp, \A_{l-n}, \F(\A_{l-n}))  \cong \overline{\f{M}}_{0,n+2} \times \overline{\f{M}}_{0,l-n+3}.
    \]
    We include here a picture of the real points of $Y(V,\A_2)$.
    \begin{center}
        \includegraphics[width=0.25\textwidth]{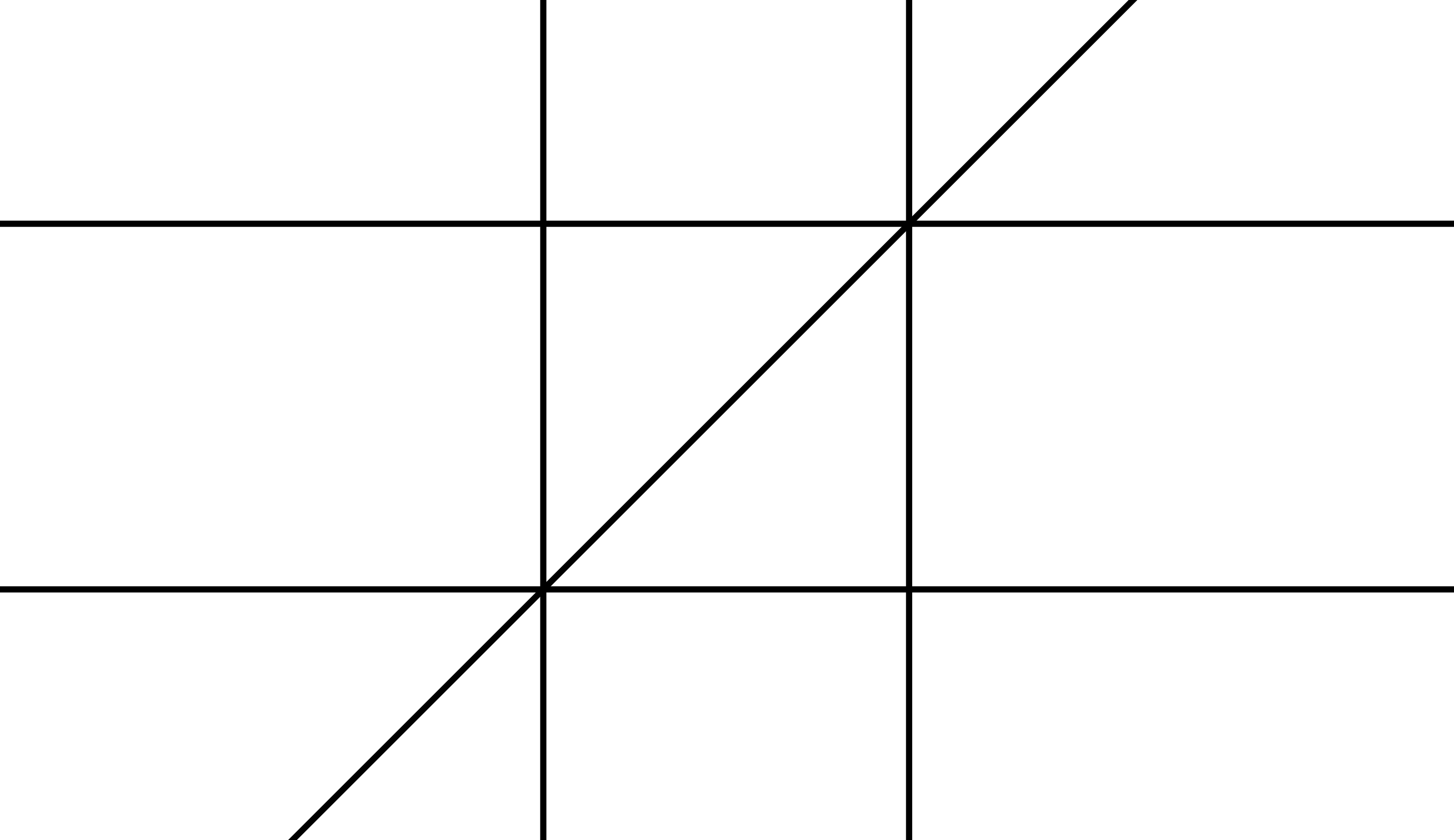}
    \end{center}
    We now illustrate the main motivating example for this work. Let $q$ be a positive integer and $\mu\in \C$ a primitive $q$-th root of unity. We work over the base field $k=\Q(\mu)$. Let $V$ be an $(l+1)$-dimensional $k$-vector space and fix a basis $x_1,\dots, x_{l+1}$ for $V^*$. 
    \begin{definition}
        \label{def: arrangement full monomial group}
        The \emph{reflection arrangement of the full monomial group} is the arrangement $(V,\A_{l,q})$ with
        \[
        \A_{l,q}=\left\{\langle x_1\rangle ,\dots,\langle x_{l+1}\rangle \right\} \cup \{ \langle x_i-\mu^n x_j\rangle \mid n=1,\dots, q,\; i,j=1,\dots,l+1, \;i\neq j \},
        \]
        whose defining polynomial is therefore
        \[
            x_1\dots x_{l+1} \prod_{1\le i <j\le l+1}(x_i^q-x_j^q).
        \]
    \end{definition}
    \noindent
    Putting the hyperplane $x_{l+1}=0$ at infinity, the variety $Y(V,\A)$ is the complement in $\mathbb{A}^l$ of the hyperplanes with equations, in affine coordinates $t_i=\frac{x_i}{x_{l+1}}$, given by
    \[ 
        t_i=0, \quad t_i=\mu^n, \quad t_i=\mu^n t_j
    \]
    for $ n=1,\dots, q$, $i,j=1,\dots,l$, $i\neq j $. \\
    Notice that for $q=1$ we recover the braid arrangement. If $q\ge 2$, the hyperplanes in $\A_{l,q}$ are the fixed points of the reflections of a certain finite reflection group, the so-called \emph{full monomial group} $G(q,1,l+1)$. The latter is isomorphic to the wreath product $C_q\wr \text{Sym}_{l+1}$, where $C_q$ is a cyclic group of order $q$. \\
    For $q=2$, the full monomial group $G(2,1,l)$ is the Coxeter group of type $B_l$. We include here a picture of the real points of $Y(V,\A_{2,2})$.\\
    \begin{center}
        \includegraphics[width=0.25\textwidth]{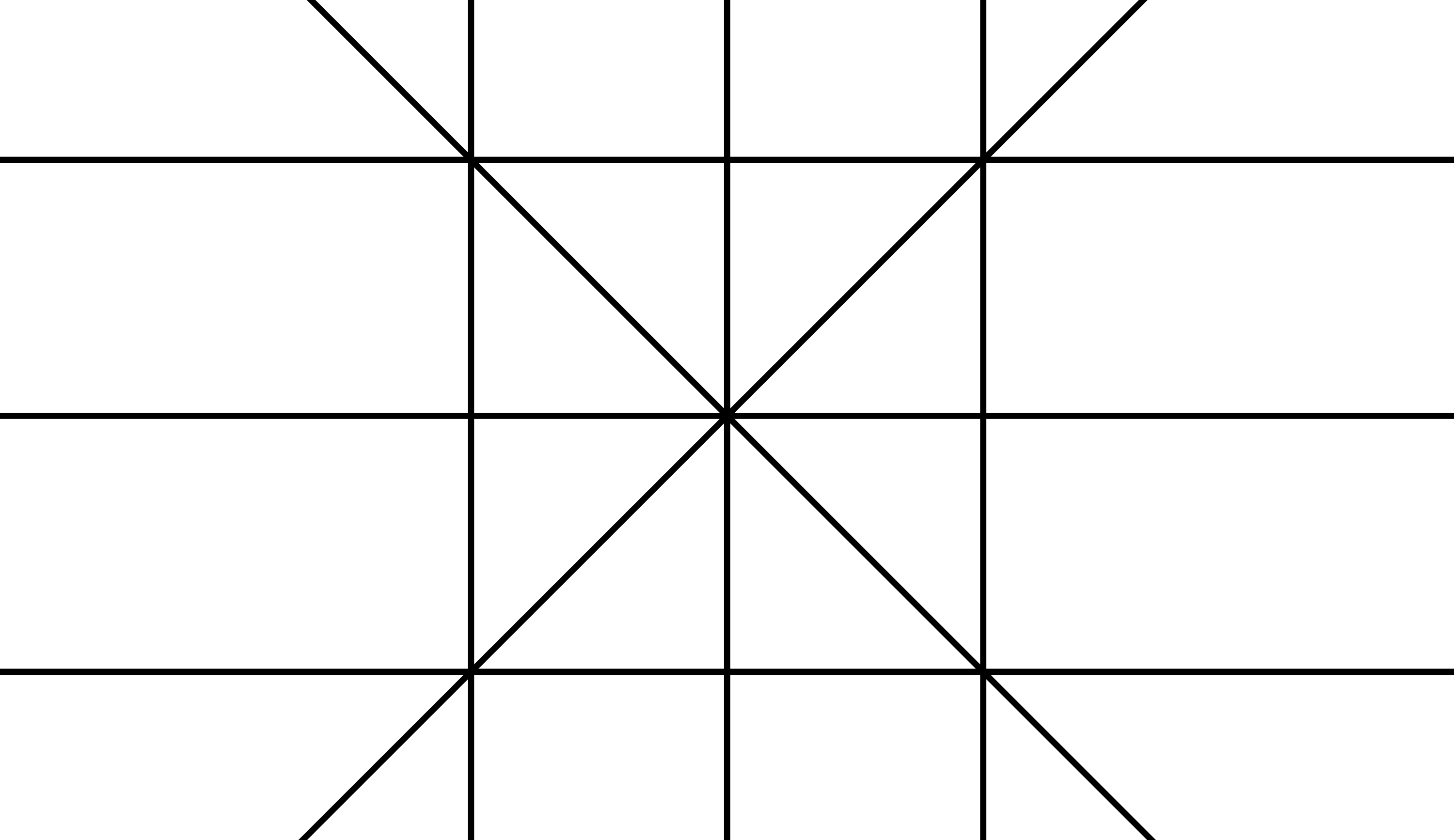}
    \end{center}Let us briefly describe $L(\A_{l,q})$. Let $\delta, \lambda\subseteq \{1,\dots,l+1\}$ with $\#\lambda\ge 2$ and consider any function $\nu \colon  \lambda\to \{1,\mu,\dots, \mu^{q-1}\}$. We define 
    \[
        X(\delta)=\sum_{i\in \delta} \langle x_i\rangle, \quad X(\lambda, \nu)=\sum_{i,j\in \lambda} \langle \nu(j)x_i-\nu(i)x_j\rangle.
    \]
    If $\delta=\emptyset$, we set $X(\delta)=0$. Clearly $X(\delta), X(\lambda,\nu)\in L(\A_{l,q})$ and $\dim X(\delta)=\#\delta$, $\dim X(\lambda,\nu)=\#\lambda-1$. By \cite[Proposition 6.74]{Orlik-Terao-Arrangements_of_hyperplanes}, every element of $L(\A_{l,q})$ is of the form
    \[
        X(\delta)\oplus X(\lambda_1,\nu_1)\oplus \dots \oplus X(\lambda_r,\nu_r)
    \]
    for some disjoint subsets $\delta,\lambda_1,\dots, \lambda_r\subseteq \{1,\dots,l+1\}$ with $\#\lambda_i\ge 2$ and some functions $\nu_i \colon \lambda_i\to \{1,\mu,\dots, \mu^{q-1}\}$. Furthermore, one may check that the above direct sum is an irreducible decomposition, so the irreducible elements of $L(\A_{l,q})$ are precisely those of the form $X(\delta)$ or $X(\lambda,\nu)$. \\
    Every one-dimensional hyperplane arrangement obtained from $\A_{l,q}$ by restriction or quotient is isomorphic to $\A_{l',q}$ or $\A_{l',1}$ for some $l'\le l$. In particular, it follows that for all $X\in \F(\A_{l,q})$ the boundary divisor $D_X$ of $\overline{Y}(V,\A_{l,q},\F(\A_{l,q}))$ is isomorphic to the product of varieties of the form $\overline{Y}(V,\A_{l',q'},\G')$ for some $l'<l$, $q'=q$ or $q'=1$ and some building set $\G'$ for $L(\A_{l',q'})$. 

    \subsection{Supersolvable arrangements}
        
    In this section, we describe a particular class of hyperplane arrangements, the so-called \emph{supersolvable arrangements}, following the exposition of \cite{Orlik-Terao-Arrangements_of_hyperplanes}.\\
    Let $(V,\A)$ be an essential hyperplane arrangement over $k$ of dimension $l+1$. Given $X,Y\in L(\A)$, in general it is not true that $X\cap Y\in L(\A)$. 
    \begin{definition}
        An element $X\in L(\A)$ is said to be \emph{modular} if for all $Y\in L(\A)$ we have $X\cap Y\in L(\A)$.     
    \end{definition}
    \noindent
    The elements $0$ and $V^*$ of $L(\A)$ are always modular, as well as the lines of $\A$. 
    \begin{example}
        In the braid arrangement $\A_l$ all irreducible elements of $L(\A_l)$ are modular. This is not true for the reflection arrangement of the full monomial group $\A_{l,q}$ with $q\ge 2$. For example, the only modular elements of codimension $1$ in $V^*$ are given by $\langle x_i \mid i\neq j\rangle$ for $i=1,\dots, l$.
    \end{example}
    \begin{lemma}
        \label{lemma: intersections with modular elts}
        Let $(V,\A)$ be a hyperplane arrangement and $M\in L(\A)$ a modular element.
        \begin{enumerate}
            \item If $N\in L(\A)$ is modular, then $M\cap N$ is modular in $L(\A)$.
            \item For all $X\in L(\A)$, the element $M\cap X \in L(\A)$ is modular in $L(\A\vert_X)$.
        \end{enumerate}
    \end{lemma}
    \begin{proof} 
        Let $M,N\in L(\A)$ be modular. For all $X\in L(\A)$ we have $N\cap X\in L(\A)$ by modularity of $N$. Since also $M$ is modular, $(M\cap N)\cap X=M\cap (N\cap X)\in L(\A)$. \\
        Let $M, X\in L(\A)$ with $M$ modular. For all $Y\in L(\A\vert_X)$ we have $Y\subseteq X$, so $M\cap X\cap Y=M\cap Y$ belongs to $L(\A)$ by modularity of $M$. But $M\cap X\cap Y\subseteq X$, hence it also lies in $L(\A\vert_X)$.
    \end{proof}
    \noindent
    The following result gives a geometric interpretation of the concept of modular elements.
    \begin{proposition}[{\cite[Theorem 5.111]{Orlik-Terao-Arrangements_of_hyperplanes}}]
    \label{prop: geometric meaning of modular element}
        Let $M\in L(\A)$. The following are equivalent:
        \begin{enumerate}
            \item $M$ is modular;
            \item The morphism $X(V,\A)\to X(V/M^\perp,\A\vert_M)$ induced by $V\to V/M^\perp$ is a fiber bundle projection.
        \end{enumerate}
    \end{proposition}
    \noindent
    If $M\in L(\A)$ is a modular element of codimension $1$, the map $X(V,\A)\to X(V/M^\perp,\A\vert_M)$ induces a morphism $Y(V,\A)\to Y(V/M^\perp,\A\vert_M)$ which is also a fiber bundle projection. Write $x_1,\dots, x_{l+1}$ for a basis of $V^*$ and let $t_i=\frac{x_i}{x_{l+1}}$ be affine coordinates for $Y(V,\A)$. Up to a change of coordinates, we may suppose that $M=\langle x_2,\dots, x_{l+1}\rangle$, so $M^\perp$ is the line given by the equations $x_2=\dots =x_{l+1}=0$. By identifying $V/M^\perp$ with the linear subspace of equation $x_1=0$, the map $Y(V,\A)\to Y(V/M^\perp,\A\vert_M)$ is induced by the linear projection $\mathbb{A}^{l}\to \mathbb{A}^{l-1}$, $(t_1,\dots, t_l)\mapsto (t_2,\dots, t_l)$. Thus, modular elements yield fiber bundle linear projections also at the level of projective complements of arrangements.   
    \begin{definition}
        An essential hyperplane arrangement is called \emph{supersolvable} if there is a maximal chain  $0=X_0\subsetneq X_1 \subsetneq\dots \subsetneq X_{l+1}=V^*$ of modular elements in $L(\A)$.
    \end{definition}
    \noindent
    In the literature, supersolvable arrangements are usually assumed to be central, but we prefer to stick with this slightly more general definition.
    \begin{example}
        The braid arrangement $\A_l$ is supersolvable. A maximal chain of modular elements is
        \[
        0 \subsetneq \langle x_1 \rangle \subsetneq \langle x_1, x_2\rangle \subsetneq \dots \subsetneq \langle x_1,\dots, x_{l}\rangle \subsetneq V^*.
        \]
        An example of an arrangement which is not supersolvable is $(V,\A)$ with $V^*$ of dimension $3$ with basis $x,y,z$ and $\A=\{\langle x\rangle, \langle y\rangle, \langle z\rangle ,\langle x+y-z\rangle \} $. It is straightforward to see that $(V,\A)$ has no modular element of dimension $2$. 
    \end{example}
    \begin{lemma}[{\cite[Lemma 2.62]{Orlik-Terao-Arrangements_of_hyperplanes}}]
        Let $(V,\A)$ be a supersolvable arrangement. Then for all $X\in L(\A)$ the arrangements $(V/X^\perp, \A\vert_X)$ and $(X^\perp, \A/X)$ are supersolvable.
    \end{lemma}
    \noindent
    Geometrically, central supersolvable arrangements fiber over $1$-dimensional arrangements. To draw this connection in a more neat way, we introduce the following definition.
    \begin{definition} $ $
        \begin{enumerate}
            \item A hyperplane arrangement $(V,\A)$ of dimension $l+1$ is \emph{strictly linearly fibered} if there is a hyperplane arrangement $(V',\A')$ of dimension $l$ together with a fiber bundle linear projection $Y(V,\A)\to Y(V',\A')$ whose fiber is given by $\mathbb{P}^1$ with finitely many points removed.
            \item \emph{Fiber type hyperplane arrangements} are defined inductively as follows. Every arrangement of dimension $2$ is fiber type. For $l\ge 2$ and $(V,\A)$ of dimension $l+1$, the arrangement $(V,\A)$ is fiber type if it is strictly linearly fibered over an arrangement $(V',\A')$ of dimension $l$ which is fiber type.
        \end{enumerate}
    \end{definition}
    \noindent
    These definitions are adapted to projective arrangements, but they are still equivalent to those of \cite{Orlik-Terao-Arrangements_of_hyperplanes}. If $(V,\A)$ is fiber type, one may find hyperplane arrangements $(V,\A)=(V_1,\A_1),(V_2,\A_2),\dots, (V_l,\A_l)$ together with maps
    \[
        Y(V,\A)=Y(V_1,\A_1) \overset{\pi_1}{\longrightarrow} Y(V_2,\A_2) \overset{\pi_2}{\longrightarrow} \dots \overset{\pi_{l-1}}{\longrightarrow}Y(V_l,\A_l)
    \]
    such that $\dim Y(V_i,\A_i)=l+1-i$ and $\pi_i$ is a fiber bundle projection whose fiber is isomorphic to $\mathbb{P}^1$ with finitely many points removed. \\
    If $(V,\A)$ is central, it is possible to construct a sequence of projections $\pi_1,\dots, \pi_{l-1}$ of this kind starting from a maximal chain of modular elements and using Proposition~\ref{prop: geometric meaning of modular element}. When making this idea precise, one obtains the following result, which gives a geometric interpretation to the concept of supersolvable arrangement. 
    \begin{proposition}[{\cite[Theorem 5.113]{Orlik-Terao-Arrangements_of_hyperplanes}}]
        \label{prop: fiber type equals supersolvable}
        Let $(V,\A)$ be a central hyperplane arrangement. Then $(V,\A)$ is supersolvable if and only if it is fiber type.
    \end{proposition}
    \noindent
    With our conventions, supersolvable arrangements are not necessarily central. From the discussion above, modular elements give fiber bundle linear projections also on projective complements of arrangements. Thus, in our context, supersolvable hyperplane arrangements are fiber type.

    \subsection{Arrangements with enough modular elements}
        
    Let $(V,\A)$ be a hyperplane arrangement. For $X\in L(\A)$ the quotient map $V\to V/X^\perp$ always induces a morphism $Y(V,\A)\to Y(V/X^\perp, \A\vert_X)$. However, in general there need not be a morphism $Y(V,\A)\to Y(X^\perp, \A/X)$ which is induced by a linear projection $V\to X^\perp$, as the following example shows.\\
    Suppose $V$ has dimension $3$ with a dual basis $x_1,x_2,x_3$ and let $\A$ consist of the lines $\langle x_i\rangle$ for $i=1,2,3$ together with the line $\langle x_1+x_2-x_3 \rangle$. The projection $V\to \langle x_2\rangle^\perp$ with kernel $\langle x_1, x_3 \rangle^\perp$ does not induce a morphism between the complements of the arrangements $\A$ and $\A/\langle x_2\rangle$. Intuitively, the obstruction to this is the the fact that the line $\langle x_1-x_3\rangle$ does not belong to $\A$. Notice that this line is the intersection of $\langle x_1, x_3\rangle$, which is the kernel of the above projection, together with $\langle x_2, x_1+x_2-x_3\rangle\in L(\A)$. Thus, the fact that such line is missing from $\A$ implies that $\langle x_1, x_3\rangle$ is not a modular element of $L(\A)$. \\
    In this section, we wish to isolate a class of hyperplane arrangements which admit linear projections onto the restriction arrangements $Y(X^\perp, \A/X^\perp)$ for every element $X\in L(\A)$.
    \begin{definition}
        Let $(V,\A)$ be a hyperplane arrangement. We say that $(V,\A)$ \emph{has enough modular elements} if for every $H\in\A$ there is $M\in L(\A)$ of codimension $1$ which is modular and does not contain $H$.
    \end{definition}
    \begin{example}
        The braid arrangement $\A_l$ has enough modular elements. Indeed, for every $j=1,\dots, l+1$, the element $\langle x_i \mid i\neq j\rangle$ is modular. The condition of having enough modular elements is then readily verified. \\
        In a completely analogous way, one may show that the reflection arrangement of the full monomial group from Definition~\ref{def: arrangement full monomial group} has enough modular elements.
    \end{example}
    \noindent
    Our goal is to show that having enough modular elements ensures the existence of linear projections onto restriction arrangements. Before proceeding, we observe that having enough modular elements is a stronger property than being supersolvable.
    \begin{lemma}
        Hyperplane arrangements with enough modular elements are supersolvable.
    \end{lemma}
    \begin{proof}
        Let $M\neq 0$ be any modular element in $L(\A)$. Given $H\in \A$, $H\subseteq M$, there is $N\in L(\A)$ modular such that $N$ has codimension $1$ in $V^*$ and does not contain $H$. Since $H\subseteq M$, for dimension reasons we have $M+N=V^*$, thus $M\cap N$ has codimension $1$ in $M$. Moreover, $M\cap N$ is modular in $L(\A)$ by Lemma~\ref{lemma: intersections with modular elts}. Applying this argument inductively provides a maximal chain of modular elements in $L(\A)$, which proves supersolvability.
    \end{proof}
    \begin{example}
        \label{example: enough modular elts not supersolvable}
        An example of a supersolvable arrangement which does not have enough modular elements is the following. Consider a $3$-dimensional $k$-vector space $V$ with a dual basis $x_1,x_2,x_3$ and the hyperplane arrangement 
        \[
            \A=\{\langle x_1\rangle, \langle x_2\rangle, \langle x_3 \rangle, \langle x_1+x_2-x_3\rangle, \langle x_2-x_3\rangle\}.
        \]
        An affine picture of $\A$ in the chart $x_3\neq 0$ is the following:
        \smallskip
        \begin{center}
            \includegraphics[width=0.25\textwidth]{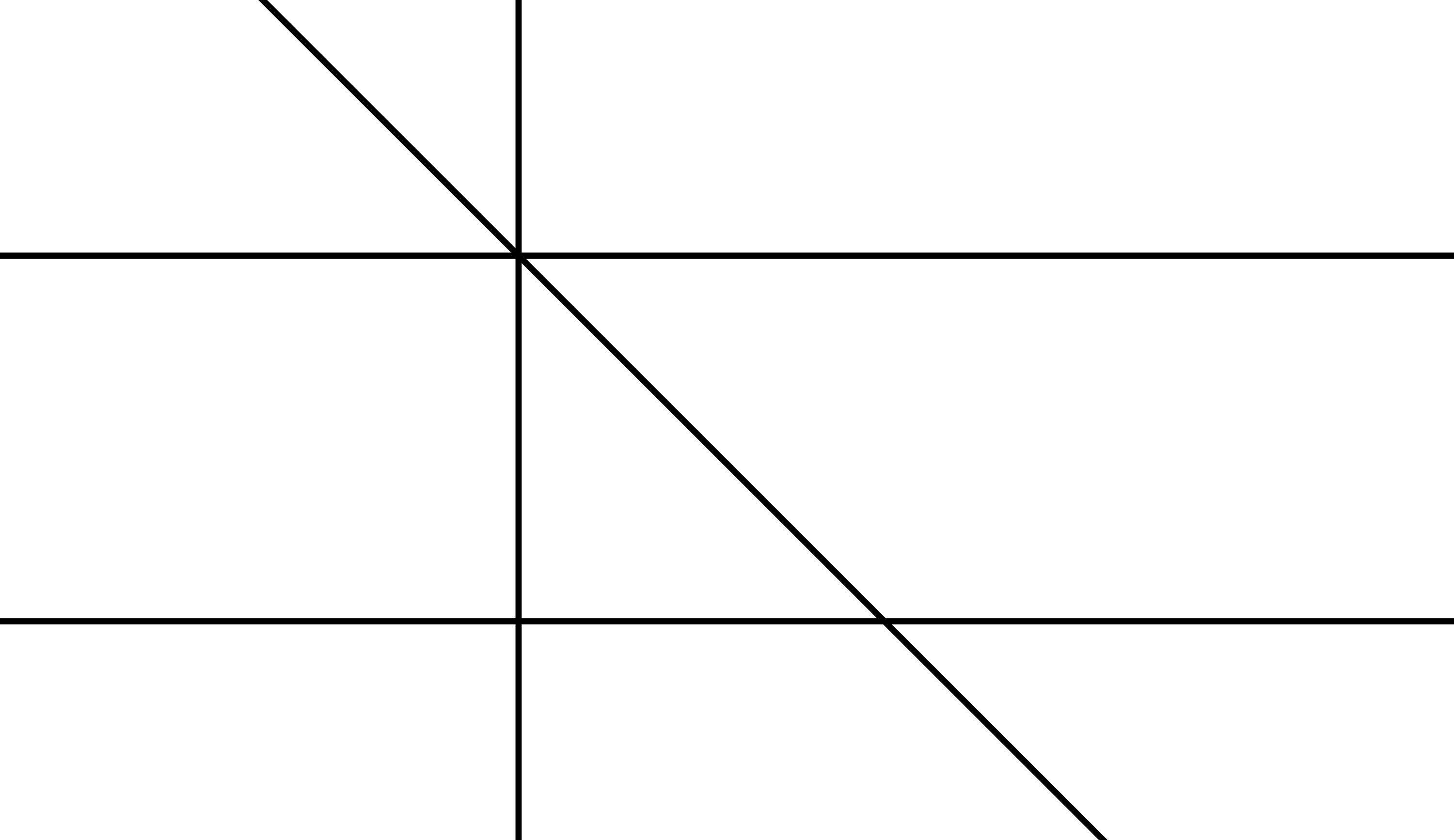}
        \end{center}

        \smallskip
        \noindent
        Since $\dim V=3$, to prove supersolvability it is enough to exhibit a modular element in $L(\A)$ of codimension $1$. By taking intersections with the other lattice elements of codimension $1$, one can check that both $\langle x_2, x_3\rangle $ and $\langle x_1,x_1+x_2-x_3\rangle$ are modular.\\
        However, these are the only modular elements of $L(\A)$ of codimension $1$. This is due to the fact that the intersections
        \begin{align*}
            \langle x_1, x_2\rangle \cap \langle x_3, x_1+x_2-x_3\rangle =\langle x_1+x_2\rangle, \\
            \langle x_1,x_3\rangle \cap \langle x_2,x_1+x_2-x_3\rangle = \langle x_1-x_3\rangle
        \end{align*}
        do not belong to $\A$.\\
        Since $\langle x_2,x_3\rangle\cap \langle x_1,x_1+x_2-x_3\rangle=\langle x_2-x_3\rangle$, there is no modular element of codimension $1$ that does not contain $\langle x_2-x_3\rangle$. This shows that $(V,\A)$ does not have enough modular elements. 
    \end{example}
    \noindent
    Having enough modular elements ensures the existence of fiber bundle linear projections onto all restriction arrangements $(H^\perp,\A/H)$ for all $H\in \A$ and actually, more generally, onto all restrictions to lattice elements $(X^\perp, \A/X)$ for $X\in L(\A)$. We start by dealing with the case of hyperplanes.
    \begin{lemma}
        \label{lemma: restrictions into quotients of modular elts for H}
        Let $(V,\A)$ be a hyperplane arrangement and let $H\in \A$. Suppose that there is a modular element $M\in L(\A)$ of codimension $1$ which does not contain $H$. Then the restriction of the quotient map $V\to V/M^\perp$ to $H^\perp$ induces an isomorphism of hyperplane arrangements between $(H^\perp, \A/H)$ and $(V/M^\perp, \A\vert_M)$.
    \end{lemma}
    \begin{proof}
        At the level of arrangements, the map in the statement is given by $\A\vert_{M}\to \A/H$, $X\mapsto (X+H)/H$. To prove injectivity, suppose that there are $X_1,X_2\in L(\A)$ such that $X_1, X_2\subseteq M$ and $X_1+H=X_2+H$. Since $H\not\subseteq M$, we have $H\not\subseteq X_1, X_2$, hence $\dim (X_1+H)=\dim X_1 +1$ and $\dim (X_2+H)=\dim X_2 +1$. If we assume $X_1\neq X_2$, for dimension reasons we have $X_1+H=X_1+X_2\subseteq M$, which leads to a contradiction since $H\not\subseteq M$. \\
        For surjectivity, let $X\in L(\A)$ containing $H$. Since $M$ is modular, the intersection $X\cap M$ still belongs to $L(\A)$ and is of course contained in $M$. Moreover, since $X$ contains $H$, we have $X=(X\cap M) \oplus H$, hence $(X\cap M + H)/H=X/H$.
    \end{proof}
    \begin{corollary}
        \label{cor: projections to restrictions}
        Let $(V,\A)$ be a hyperplane arrangement and let $H\in \A$. Suppose that there is a modular element $M\in L(\A)$ of codimension $1$ which does not contain $H$. Then there is a linear projection $Y(V,\A)\to Y(H^\perp, \A/H)$ which is also a fibration. 
    \end{corollary}
    \begin{proof}
        For all $M\in L(\A)$ a linear projection $Y(V,\A)\to Y(V/M^\perp, \A\vert_M)$ always exists and it is a fibration when $M$ is modular. By the previous Lemma, we have a linear isomorphism $Y(V/M^\perp, \A\vert_M)\cong Y(H^\perp, \A/H)$.
    \end{proof}
    \begin{proposition}
    \label{prop: restrictions into quotients of modular elts}
        Let $(V,\A)$ be a hyperplane arrangement with enough modular elements. Then for all $X\in L(\A)$ there is a modular element $M\in L(\A)$ such that $V^*=X\oplus M$. Moreover, the restriction of the quotient map $V\to V/M^\perp$ to $X^\perp$ induces an isomorphism of hyperplane arrangements $(X^\perp, \A/X)\cong (V/M^\perp, \A\vert_M)$.
    \end{proposition}
    \begin{proof}
        We argue by induction on $\dim X$. If $\dim X=1$, we may find some $M\in L(\A)$ modular of codimension $1$ with $X\not\subseteq M$. For dimension reasons, this means that $V^*=X\oplus M$. The rest of the claim is precisely the content of Lemma~\ref{lemma: restrictions into quotients of modular elts for H}.\\
        Assume that $\dim X=n\ge 2$ and the statements holds for all $(n-1)$-dimensional subspaces in $L(\A)$. Let us write $X=H_1+\dots+  H_n$ with $H_1,\dots, H_n\in \A$. Since $Y=H_1+\dots + H_{n-1}$ has dimension $n-1$, by induction hypothesis we can find $M\in L(\A)$ modular such that $V^*=Y\oplus M$. Moreover, there is a natural isomorphism $(Y^\perp, \A/Y)\cong (V/M^\perp, \A\vert_M)$. Since $X=Y+H_n$, under this isomorphism $X/Y\in \A/Y$ corresponds to the line $L=X\cap M\in\A\vert_M$. Thus, we have $X=Y+L$ with $L\subseteq M$. Since $\A$ has enough modular elements, there is $N\in L(\A)$ modular of codimension $1$ in $V^*$ which does not contain $L$. \\
        Let us now consider $M\cap N$. This is an intersection of modular elements, so it is itself modular. Since $M$ and $N$ have codimension $1$ in $V^*$ and $L\subseteq M$, but $L\not\subseteq N$, it follows that $M=L\oplus (M\cap N)$. Let us then show that $V^*=X\oplus M\cap N$. First, observe that $X\cap M\cap N=0$. Indeed, as $X=Y+L$, an element in this intersection can be written as $y+l=z$ with $y\in Y, l\in L$ and $z\in M\cap N$. Thus, $y=z-l\in M\cap Y=0$, hence $l=z\in L\cap N=0$. Secondly, we have $V^*=Y+M=Y+L+(M\cap N)=X+(M\cap N)$.\\
        In the arrangement $(Y^\perp, \A/Y)\cong (V^*/M^\perp, \A\vert_M)$ we have $L\in \A\vert_M$ and a modular element $M\cap N$ of codimension $1$ not containing $L$. By Lemma~\ref{lemma: restrictions into quotients of modular elts for H}, we have a chain of isomorphisms
        \begin{align*}
        (X^\perp, \A/X) &= (Y^\perp\cap L^\perp, (\A/Y)/L) \\
                        &\cong  \left( \frac{V^*/M^\perp}{(M\cap N)^\perp / M^\perp} , \left(\A\vert_M \right)\vert_{M\cap N} \right) \\
                        &\cong  (V^*/(M\cap N)^\perp, \A\vert_{M\cap N}). 
        \end{align*}
    \end{proof}
    \begin{corollary}
        Let $(V,\A)$ be a hyperplane arrangement with enough modular elements. Then for all $X\in L(\A)$ there is a fiber bundle linear projection $Y(V,\A)\to Y(X^\perp, \A/X)$.
    \end{corollary}
    \begin{proof}
        For all $X\in L(\A)$ there is a modular element $M\in L(\A)$ for which there is a linear isomorphism $(X^\perp, \A/X)\cong (V/M^\perp, \A\vert_M)$ by the previous Proposition. For $M$ is modular, the quotient map $Y(V,\A)\to Y(V/M^\perp, \A\vert_M)$ is a fibration and the claim follows by composing with said isomorphism.
    \end{proof}
    \noindent
    In the rest of the exposition, we will also need a few other elementary properties of arrangements with enough modular elements which we now collect.
    \begin{lemma}
        \label{lemma: enough modular elts preserved by restrictions and quotients}
        Let $(V,\A)$ be a hyperplane arrangement with enough modular elements and let $X\in L(\A)$. Then
        \begin{enumerate}
            \item the arrangement $(V/X^\perp, \A\vert_X)$ has enough modular elements;
            \item the arrangement $(X^\perp, \A/X)$ has enough modular elements.
        \end{enumerate}
    \end{lemma}
    \begin{proof} 
    For the first claim, let $H\in \A$, $H\subseteq X$. Since $(V,\A)$ has enough modular elements, there is $M\in L(\A)$ modular of codimension $1$ in $V^*$ which does not contain $H$. Then $X\cap M$ has codimension $1$ in $X$, as $H\subseteq X$, and does not contain $H$, as $H\not\subseteq M$. Moreover, $M\cap X$ is modular in $L(\A\vert_X)$ by Lemma~\ref{lemma: intersections with modular elts}. \\
    For the second point, by Proposition~\ref{prop: restrictions into quotients of modular elts}, there is a modular element $M\in L(\A)$ such that $(X^\perp, \A/X)\cong (V/M^\perp, \A\vert_M)$. It is therefore enough to prove the claim for the latter arrangement, which follows from the previous point.
    \end{proof}
    \begin{lemma}
        \label{lemma: enough modular elements preserved by sums}
        Let $(V_1,\A_1)$ and $(V_2,\A_2)$ be hyperplane arrangements with enough modular elements. Then the arrangement $(V_1\oplus V_2, \A_1\sqcup \A_2)$ has enough modular elements.
    \end{lemma}
    \begin{proof}
        Let $H\in \A_1\sqcup \A_2$, which we may assume to lie in $\A_1$ without loss of generality. Since $(V_1,\A_1)$ has enough modular elements, there is $M\in L(\A_1)$ modular of codimension $1$ in $V_1^*$ which does not contain $H$. It follows that $M\oplus V_2^*$ has codimension $1$ in $V_1^*\oplus V_2^*$ and does not contain $H$; let us show that it is also modular. \\
        Every $X\in L(\A_1\sqcup \A_2)$ has a decomposition $X=X_1\oplus X_2$ with $X_1\in L(\A_1)$ and $X_2\in L(\A_2)$. Since $M\oplus V_2^*$ is a decomposition as well, we have
        \begin{align*}
            (M\oplus V_2^*)\cap X & = (M\cap X)\oplus (V_2^*\cap X)\\
                                  & = (M\cap X_1)\oplus (M\cap X_2)\oplus (V_2^*\cap X_1)\oplus (V_2^*\cap X_2)\\
                                  & = (M\cap X_1)\oplus X_2,
        \end{align*}
        which belongs to $L(\A_1\sqcup \A_2)$ by modularity of $M$.
    \end{proof}
    \noindent
    The following is a straightforward criterion for checking whether a hyperplane arrangement has enough modular elements.
    \begin{lemma}
        \label{lemma: criterion for enough modular elements}
        Let $(V,\A)$ be a hyperplane arrangement. Suppose that there exist modular elements $X_1,\dots, X_r\in L(\A)$ of codimension $1$ such that $\bigcap_{i=1}^rX_i=0$. Then $(V,\A)$ has enough modular elements. 
    \end{lemma}
    \begin{proof}
        If $(V,\A)$ does not have enough modular elements, there is $H\in \A$ contained in all modular elements of $L(\A)$ of codimension $1$. This contradicts the assumption that the intersection of $X_1,\dots, X_r$ is zero.
    \end{proof}
    \noindent
    Let us now analyze Example~\ref{example: enough modular elts not supersolvable} in more detail. If we call $(V,\A)$ the arrangement appearing in this example, it is possible to add some hyperplanes to the arrangement to ensure the existence of enough modular elements. For instance, this can be achieved by considering the arrangements $\A\cup \{\langle x_1-x_3\rangle\}$ or $\A\cup \{\langle x_1+x_2\rangle \}$. It is therefore natural to ask whether every supersolvable arrangement can be embedded into one with enough modular elements.\\
    This question has a certain relevance when it comes to computing period integrals associated with hyperplane arrangements. If a hyperplane arrangement $(V,\A)$ embeds into another arrangement $(V,\A')$, there is an open immersion $Y(V,\A')\to Y(V,\A)$. If $\G'$ is a building set for $L(\A')$ and we set $\G=\G'\cap L(\A)$, there is an induced morphism $\overline{Y}(V,\A',\G')\to \overline{Y}(V,\A,\G)$ which is given by a chain of blow-ups. At this point, one can compute period integrals on $\overline{Y}(V,\A,\G)$ by pulling them back to $\overline{Y}(V,\A',\G')$. Thus, any result concerning period integrals associated with $(V,\A')$ would apply to all subarrangements $(V,\A)$ as well. It would therefore be interesting to find a criterion to understand whether a supersolvable arrangement can be embedded into one with enough modular elements.\\
    Returning to the case of $(V,\A)$ as in Example~\ref{example: enough modular elts not supersolvable}, it appears difficult to determine an algorithm to add certain hyperplanes to obtain enough modular elements. For instance, the arrangements $\A\cup \{\langle x_1-x_3\rangle\}$ and $\A\cup \{\langle x_1+x_2\rangle \}$ proposed above are chosen so to make $\langle x_1,x_3\rangle$ and $\langle x_3, x_1+x_2-x_3\rangle$ modular, respectively. In this case, this procedure suffices to get enough modular elements. However, if we were to make both $\langle x_1,x_3\rangle$ and $\langle x_3, x_1+x_2-x_3\rangle$ modular, we would get the arrangement $\A\cup \{\langle x_1-x_3\rangle,\langle x_1+x_2\rangle\}$, which can be checked not to have enough modular elements.\\
    \begin{center}
        \includegraphics[width=0.25\textwidth]{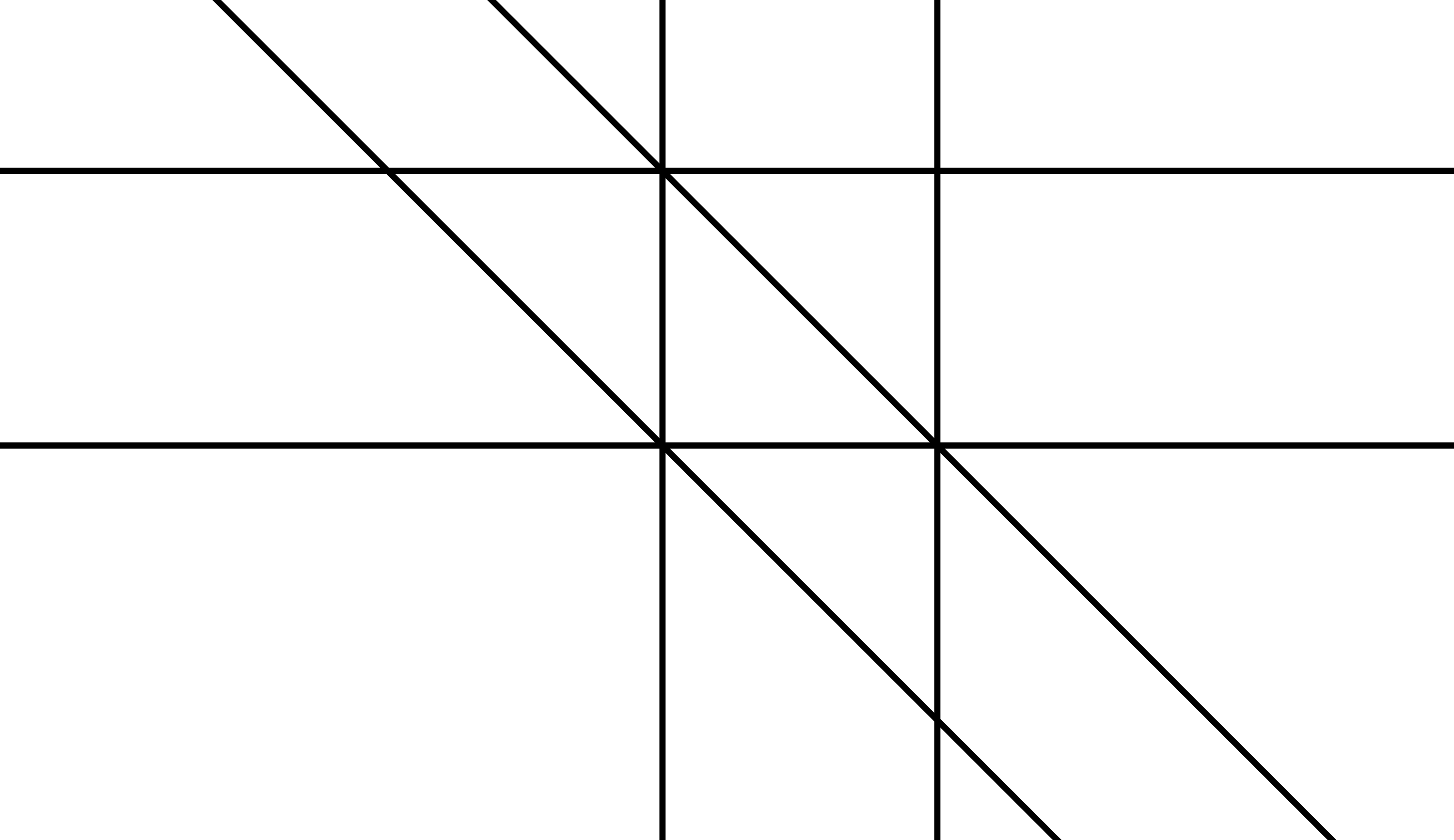}
    \end{center}
    \begin{example}
        \label{example: failure of recovering enough modular elements}
        Consider the hyperplane arrangement $\A$ over $V^*=\bigoplus_{i=1}^3 \Q x_i$ given by
    \begin{gather*}
        \langle x_1\rangle, \; \langle x_2\rangle, \; \langle x_3\rangle, \; \langle x_1+x_3\rangle, \;
        \langle x_2-x_3\rangle, \; \langle 2x_2-x_3\rangle, \\\langle x_2-2x_3\rangle, \;\langle x_1+x_2-x_3\rangle,\; \langle x_1-2x_2+2x_3\rangle.
    \end{gather*}
    The affine picture with respect to $x_3\neq 0$ is the following:
    \medskip
    \begin{center}
        \includegraphics[width=0.25\textwidth]{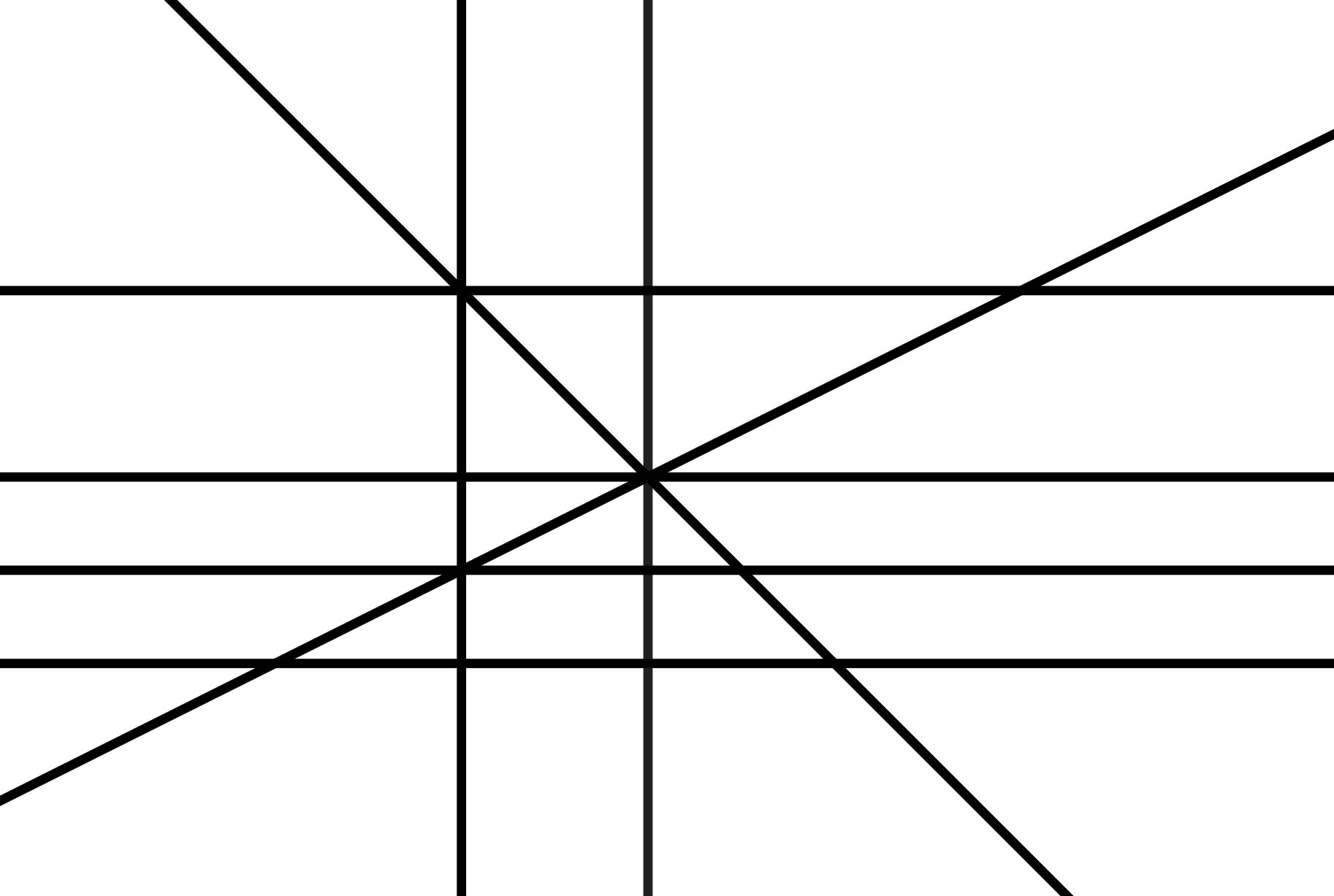}
    \end{center}
    \medskip
    The set $\F$ irreducible elements in $L(\A)$ consists of the lines in $\A$, $V^*$ and the following elements of codimension $1$ in $V^*$:
    \begin{enumerate}
        \item $\langle x_2,x_3,x_2-x_3,2x_2-x_3,x_2-2x_3\rangle$;
        \item $\langle x_1,x_3,x_1+x_3\rangle$;
        \item $\langle x_1,x_2-x_3,x_1-2x_2+2x_3,x_1+x_2-x_3\rangle$;
        \item $\langle x_1+x_3,2x_2-x_3, x_1-2x_2+2x_3\rangle$;
        \item $\langle x_1+x_3, x_2-2x_3, x_1+x_2-x_3\rangle$;
    \end{enumerate}
    Among these,
    \begin{enumerate}
        \item $\langle x_2,x_3,x_2-x_3,2x_2-x_3,x_2-2x_3\rangle$ is modular;
        \item $\langle x_1,x_3,x_1+x_3\rangle$ is not modular, because 
        \[
            \langle x_1,x_3\rangle\cap \langle x_2,x_1+x_2-x_3\rangle= \langle x_1-x_3\rangle \notin L(\A);
        \]
        \item $\langle x_1,x_2-x_3,x_1-2x_2+2x_3,x_1+x_2-x_3\rangle$ is not modular, because
        \[
            \langle x_1,x_2-x_3\rangle\cap \langle x_2,x_1+x_3\rangle=\langle x_1-x_2+x_3\rangle\notin L(\A);
        \]
        \item $\langle x_1+x_3, 2x_2-x_3, x_1-2x_2+2x_3\rangle$ is not modular, because
        \[
            \langle x_1+x_3,2x_2-x_3\rangle\cap \langle x_1,x_2\rangle= \langle x_1+2x_2\rangle \notin L(\A); 
        \]
        \item $\langle x_1+x_3, x_2-2x_3, x_1+x_2-x_3\rangle$ is not modular, because
        \[
            \langle x_1+x_3, x_2-2x_3 \rangle \cap \langle x_1,x_2\rangle=\langle 2x_1+x_2\rangle\notin L(\A).
        \]
    \end{enumerate}
    It is straightforward to conclude that $\langle x_2,x_3\rangle$ is the only modular element of $L(\A)$ of codimension $1$ in $V^*$. In particular, $(V,\A)$ is supersolvable, but it does not have enough modular elements. Also in this example it does not seem possible to embed $\A$ in an arrangement with enough modular elements.
    \end{example}
    \begin{remark}
        It would be interesting to study further properties of arrangements with enough modular elements, for example their characteristic polynomials and combinatorial or topological invariants.
    \end{remark}

    \subsection{Retractions to boundary divisors}
        
    \subsubsection{Modular elements and building sets}
    Let $(V,\A)$ be an essential arrangement over $k$ and fix a building set $\G$ for $L(\A)$. In this section, we study the geometric consequences of enough modular elements in $(V,\A)$ on the compactification $\overline{Y}(V,\A,\G)$. \\
    As we have seen, if $(V,\A)$ has enough modular elements, there exist linear projections $Y(V,\A)\to Y(X^\perp,\A/X)$ for every $X\in L(\A)$. Under a refinement of this assumption, we can construct retractions of $\overline{Y}(V,\A,\G)$ onto the boundary divisor $D_X$ for all $X\in \G$.
    \begin{definition}
        Let $(V,\A)$ be a hyperplane arrangement and $\G$ a building set for $L(\A)$.
        \begin{enumerate}
            \item An element $M\in \G$ is said to be $\G$-modular if it is modular and for all $X\in \G$ we have $X\cap M\in \G$.
            \item We say that $(V,\A)$ \emph{has enough $\G$-modular elements} if for all $H\in \A$ there is a $\G$-modular element $M\in \G$ of codimension $1$ in $V^*$ that does not contain $H$.
        \end{enumerate}
    \end{definition}
    \begin{remark}
        In the notation of the definition, $M\in L(\A)$ is modular if and only if it is $L(\A)$-modular. Thus, the definitions given in the previous section are a particular case of the present ones, namely the case $\G=L(\A)$.
    \end{remark}
    \begin{example}
        The braid arrangement $A_l$ has enough $\F(\A_l)$-modular elements. More generally, the reflection arrangement of the full monomial group from Definition~\ref{def: arrangement full monomial group} has enough $\F(\A_{l,q})$-modular elements.
    \end{example}
    \begin{lemma}
        \label{lemma: cone has enough modular elements}
        Let $(V,\A)$ be a hyperplane arrangement, $\G$ a building set for $L(\A)$. Let $(\widetilde{V}, \widetilde{\A})$ be the cone over $(V,\A)$ and write $\widetilde{V}=V\oplus L$. Let $\widetilde{\G}$ be the building set defined in Lemma~\ref{lemma: building sets for cones}. If $(V,\A)$ has enough $\G$-modular elements, then $(\widetilde{V},\widetilde{\A})$ has enough $\widetilde{\G}$-modular elements.
    \end{lemma}
    \begin{proof}
        If $M\in \G $ is $\G$-modular, then $M+L^*$ is $\widetilde{\G}$-modular. Indeed, for $X\in \G$  we have that $X\oplus L^*$ is a decomposition, as well as $M\oplus L^*$. Hence
        \[
            (M\oplus L^*)\cap (X\oplus L^*)=(M\cap X)\oplus (M\cap L^*)\oplus (L^* \cap X)\oplus (L^*\cap L^*)= M\cap X\oplus L^*.
        \]
        Since $M$ is $\G$-modular $M\cap X\in \G$, so $M\cap X\oplus L^*\in \widetilde{\G}$. Similarly, one checks that $(M+L^*)\cap X\in \widetilde{\G} $.\\ 
            For all $H\in \A$ there is a $\G$-modular element $M\in \G$ such that $V^*=H\oplus M$. Thus, $M+L^*$ is a $\G$-modular element of $L(\widetilde{\A})$ of codimension $1$ that does not contain $H$. Finally, $V^*$ is clearly a $\G$-modular element of $L(\widetilde{\A})$ of codimension $1$ and does not contain $L^*$.
    \end{proof}
    \noindent
    All properties that have been previously exposed for arrangements with enough modular elements can be swiftly adapted to the $\G$-modular case. In particular, these arrangements are supersolvable, or, more precisely, $\G$-supersolvable \cite{Coron-Supersolvability_of_built_lattices_and_Koszulness_of_generalized_Chow_rings}.
    \begin{proposition}
        \label{prop: restrictions into quotients of G-modular elts}
        Let $(V,\A)$ be a hyperplane arrangement and $\G$ a building set for $L(\A)$. Suppose that $(V,\A)$ has enough $\G$-modular elements. Fix $X\in \G$.
        \begin{enumerate}
            \item There is a $\G$-modular element $M\in \G$ such that $V^*=X\oplus M$.
            \item The restriction of the quotient map $V\to V/M^\perp$ to $X^\perp$ induces an isomorphism of arrangements $(X^\perp, \A/X) \cong (V/M^\perp, \A\vert_M)$.
            \item  Via the induced bijection $L(\A/X)\to L(\A\vert_M)$, the building set $\G\vert_M$ is the image of the building set $\G/X$.
        \end{enumerate}
    \end{proposition}
    \begin{proof}
        For the first two points it is possible to apply the strategy of the proof of Proposition~\ref{prop: restrictions into quotients of modular elts} with some minor modifications. In the notation of this proof, since $X\in \G$, by induction hypothesis we may assume $M$ to be $\G$-modular. On the other hand, the fact that $(V,\A)$ has enough $\G$-modular elements allows to assume $N$ to be $\G$-modular. As the intersection of $\G$-modular elements is itself $\G$-modular, the first two points follow at once.\\
        Regarding the third claim, for all $Y\in L(\A)$ the bijection $L(\A/X)\to L(\A\vert_M)$ maps $(Y+X)/X$ to $(Y+X)\cap M$. If $Y\in \G$, since $X\in \G$ as well, either $Y\oplus X$ is a decomposition or $Y+X\in \G$. In the former case, we have $(Y+X)\cap M = (Y\cap M)\oplus(X\cap M)=Y\cap M$, as $X\cap M=0$. Since $Y\in \G$, by $\G$-modularity of $M$ we conclude that $Y\cap M\in \G$. In the latter case, we have $X+Y\in \G$, hence directly $(X+Y)\cap M\in \G$ by $\G$-modularity. This shows that the image of $\G/X$ is contained in $\G\vert_M$.\\
        For the opposite inclusion, it suffices to observe that for all $Y\in \G$, $Y\subseteq M$ we have $(Y+X)\cap M=Y+M\cap X=Y$. 
    \end{proof}
    \begin{remark}
        In the above situation, suppose that $\G=\F(V,\A)$ is the family of irreducible elements of $L(\A)$. Since a $\G\vert_M$-decomposition is also a $\G$-decomposition, it follows that $\G\vert_M$ is itself the family of irreducible elements of $L(\A\vert_M)$. Thus, $\G/X$ coincides with the set of irreducible elements of $L(\A/X)$, as well. Note that this is not always the case, in general, as we have seen in Example~\ref{example: irreducibles do not pass to quotients}.\\
        Geometrically, we can interpret this fact as follows. Recall that we have an isomorphism
        \[
            D_X\cong \overline{Y}(V/X^\perp, \A\vert_X, \G\vert_X) \times \overline{Y}(X^\perp, \A/X, \G/X),
        \]
        where $D_X$ is the irreducible component of the boundary divisor of $\overline{Y}(V,\A,\G)$ corresponding to $X\in \G$. If $\overline{Y}(V,\A,\G)$ is the minimal De Concini-Procesi compactification of $Y(V,\A)$, then $D_X$ is the minimal De Concini-Procesi compactification of 
        \[
            Y(V/X^\perp \oplus X^\perp, \A\vert_X \sqcup \A/X),
        \]
        provided that $(V,\A)$ has enough $\G$-modular elements. 
    \end{remark}
    \noindent
    As far as restrictions, quotients and sums of arrangements with enough $\G$-modular elements are concerned, we observe that Lemma~\ref{lemma: intersections with modular elts} still holds when replacing modular elements for $L(\A)$ by $\G$-modular ones, and modular elements for $L(\A\vert_X)$ by $\G\vert_X$-modular ones. The following result then follows immediately from Lemma~\ref{lemma: enough modular elts preserved by restrictions and quotients} and Lemma~\ref{lemma: enough modular elements preserved by sums}.
    \begin{lemma}
        Let $(V,\A)$ be a hyperplane arrangement and $\G$ a building set for $L(\A)$. Suppose that $(V,\A)$ has enough $\G$-modular elements. Fix $X\in L(\A)$. Then
        \begin{enumerate}
            \item the arrangement $(V/X^\perp, \A\vert_X)$ has enough $\G\vert_X$-modular elements;
            \item the arrangement $(X^\perp, \A/X)$ has enough $\G/X$-modular elements.
        \end{enumerate}
        Moreover, if $(V',\A')$ is another arrangement with enough $\G'$-modular elements for some building set $\G'$, the arrangement $(V\oplus V', \A\sqcup \A')$ has enough $\G\sqcup\G'$-modular elements.
    \end{lemma}

    \subsubsection{Construction of retractions}
    Let us now turn to the main geometric property of arrangements with enough $\G$-modular elements that we are interested in. Let $(V,\A)$ be a hyperplane arrangement, $\G$ a building set for $L(\A)$ and $X\in \G$. The irreducible boundary divisor $D_X$ of $\overline{Y}(V,\A,\G)$ corresponding to $X$ is isomorphic to 
    \[
        D_X\cong \overline{Y}(V/X^\perp, \A\vert_X, \G\vert_X) \times \overline{Y}(X^\perp, \A/X, \G/X).
    \]
    We wish to construct a retraction $\overline{Y}(V,\A,\G)\to D_X$ by defining a projection onto each of the above factors which extends a linear map on the complements of the arrangements. \\
    For the first factor no assumption on the presence of enough $\G$-modular elements is needed. On the other hand, for the second factor we use Proposition~\ref{prop: restrictions into quotients of G-modular elts} to reduce to the case of quotient arrangements.
    \begin{lemma}
        \label{lemma: quotient maps extend to compactifications}
        Let $(V,\A)$ a hyperplane arrangement, $\G$ a building set for $L(\A)$ and take $X\in \G$. Then the natural quotient map $Y(V,\A)\to Y(V/X^\perp,\A\vert_X)$ extends to a morphism $\overline{Y}(V,\A,\G)\to \overline{Y}(V/X^\perp, \A\vert_X, \G\vert_X)$.
    \end{lemma}
    \begin{proof}
        Recall that $\overline{Y}(V,\A,\G)$ is the closure of $Y(V,\A)$ inside
        \[
            \prod_{Y\in \G} \mathbb{P}(V/Y^\perp).
        \]
        Similarly, $\overline{Y}(V/X^\perp, \A\vert_X, \G\vert_X)$ is the closure of $Y(V/X^\perp, \A\vert_X)$ inside
        \[
            \prod_{Y\in \G\vert_X}\mathbb{P}\left( \frac{(V/X^\perp)}{(Y^\perp/X^\perp)} \right) \cong \prod_{Y\in \G\vert_X}\mathbb{P}\left(V/Y^\perp\right).
        \]
        Consider the diagram
        \[
        \begin{tikzcd}
            Y(V,\A) \ar[r, hook] \ar[d] & \prod_{Y\in \G} \mathbb{P}(V/Y^\perp) \ar[d, "\pi"] \\
            Y(V/X^\perp, \A\vert_X) \ar[r, hook] & \prod_{Y\in \G\vert_X}\mathbb{P}\left(V/Y^\perp\right)
        \end{tikzcd}
        \]
        where the left vertical arrow is the quotient map and $\pi$ is the projection onto the factors of the target. This diagram commutes. Since $\pi$ is continuous, we conclude that 
        \begin{align*}
        \pi\left(\overline{Y}(V,\A,\G)\right) &=\pi \left(\overline{Y(V,\A)}\right)\subseteq \overline{\pi(Y(V,\A))}\\
                                   &\subseteq \overline{Y(V/X^\perp, \A\vert_X)}=\overline{Y}(V/X^\perp, \A\vert_X, \G\vert_X).
        \end{align*}
        This proves the statement.
    \end{proof}
    \begin{proposition}
    \label{prop: retractions to boundary divisors}
        Let $(V,\A)$ be a hyperplane arrangement and $\G$ a building set for $L(\A)$. Assume that $(V,\A)$ has enough $\G$-modular elements. Then for every $X\in \G$ there is a retraction $\overline{Y}(V,\A,\G)\to D_X$.\\
        More precisely, consider the isomorphism
        \[
            D_X\cong \overline{Y}(V/X^\perp, \A\vert_X, \G\vert_X) \times \overline{Y}(X^\perp, \A/X, \G/X).
        \]
        Then there is a retraction $\overline{Y}(V,\A,\G)\to D_X$ such that 
        \begin{enumerate}
            \item its composition with the projection onto the first factor of $D_X$ restricts to the natural quotient map $Y(V,\A)\to Y(V/X^\perp,\A\vert_X)$;
            \item its composition with the projection onto the second factor of $D_X$ restricts to a linear projection $Y(V,\A)\to Y(X^\perp, \A/X)$.
        \end{enumerate}
    \end{proposition}
    \begin{proof}
        By the previous lemma, the quotient map $Y(V,\A)\to Y(V/X^\perp,\A\vert_X)$ extends to a morphism $\pi_1 \colon \overline{Y}(V,\A,\G)\to \overline{Y}(V/X^\perp, \A\vert_X, \G\vert_X)$.\\
        By Proposition~\ref{prop: restrictions into quotients of G-modular elts}, there is a $\G$-modular element $M\in \G$ such that $(X^\perp,\A/X)\cong (V/M^\perp, \A\vert_M)$ and the induced bijection $L(\A/X)\to L(\A\vert_M)$ sends $\G/X$ to $\G\vert_M$. Hence, there is an isomorphism $\overline{Y}(X^\perp, \A/X, \G/X)\cong \overline{Y}(V/M^\perp, \A\vert_M, \G\vert_M)$ induced by a linear isomorphism $X^\perp\to V/M^\perp$. Another application of the previous lemma then provides us with a morphism 
        \[
            \pi_2 \colon  \overline{Y}(V,\A,\G)\to \overline{Y}(V/M^\perp, \A\vert_M,\G\vert_M)\cong \overline{Y}(X^\perp, \A/X, \G/X)
        \]
        which extends a linear projection $Y(V,\A)\to Y(X^\perp, \A/X)$.\\
        It remains to show that $\pi = \pi_1\times \pi_2$ is a retraction. Let us set for brevity $\overline{Y}_1=\overline{Y}(V/X^\perp, \A\vert_X, \G\vert_X)$ and $\overline{Y}_2=  \overline{Y}(X^\perp, \A/X, \G/X)$. We need to check that the two compositions
        \begin{gather*}
            \overline{Y}_1\times \overline{Y}_2 \cong D_X \hookrightarrow \overline{Y}(V,\A,\G) \overset{\pi_1}{\longrightarrow} \overline{Y}_1, \\
            \overline{Y}_1 \times \overline{Y}_2 \cong D_X \hookrightarrow \overline{Y}(V,\A,\G) \overset{\pi_2}{\longrightarrow} \overline{Y}_2
        \end{gather*}
        are the projections onto the first and second component, respectively.\\
        Recall that $\overline{Y}(V,\A,\G)$ is the closure of $Y(V,\A)$ inside
        \[
        \prod_{Y\in \G} \mathbb{P}(V/Y^\perp)=\prod_{Y\in \G\vert_X} \mathbb{P}(V/Y^\perp)\times \prod_{Y\in \G, Y\not \subseteq X} \mathbb{P}(V/Y^\perp).
        \]
        Let $p_1$ and $p_2$ denote the projections from $\overline{Y}(V,\A,\G)$ to the first and second factor respectively. Moreover, $D_X$ is the closure of the preimage of $\overline{Y}(X^\perp, \A/X)$ under the projection $\overline{Y}(V,\A,\G)\to \mathbb{P}(V)$.\\
        Notice that $\pi_1$ factors through $p_1$, as it is straightforward, and $\pi_2$ factors through $p_2$. Indeed, by construction, $\pi_2$ is the restriction to $\overline{Y}(V,\A,\G)$ of the composition
        \[
        \phi \colon  \prod_{Y\in \G}\mathbb{P}(V/Y^\perp)\to \prod_{Y\in \G\vert_M} \mathbb{P}(V/Y^\perp)\cong \prod_{Y\in \G/X}\mathbb{P}(X^\perp/Y^\perp),
        \]
        where the first arrow is the projection onto the corresponding factors. Since $X$ and $M$ make up a direct sum, it is clear that this first map factors through $p_2$. As a result, it suffices to show that the compositions
        \begin{gather*}
            \overline{Y}_1\overset{i_1}{\hookrightarrow} p_1(\overline{Y}(V,\A,\G)) \overset{\pi_1}{\longrightarrow} \overline{Y}_1, \\
            \overline{Y}_2\cong p_2(D_X) \overset{i_2}{\hookrightarrow} p_2(\overline{Y}(V,\A,\G)) \overset{\pi_2}{\longrightarrow} \overline{Y}_2,
        \end{gather*}
        are the identities.\\
        For $\pi_1\circ i_1$, this fact follows immediately from the construction of $\pi_1$ in the previous lemma. For the other map, the embedding $Y(V,\A)\hookrightarrow p_2(\overline{Y}(V,\A,\G))$ factors through the complement $Y(V,\A')$ of $\A'=\{H\in \A\mid H\not\subseteq X\}$. The map $i_2$ therefore restricts to $Y(X^\perp,\A/X)\hookrightarrow Y(V,\A')$ induced by the inclusion $X^\perp \to V$. Thus, we have a commutative diagram
        \[
            \begin{tikzcd}
                                &  \prod_{Y\in \G, Y\not\subseteq X}\mathbb{P}(V/Y^\perp) \ar[r,"\phi"] & \prod_{Y\in \G/X}\mathbb{P}(X^\perp/Y^\perp) \\
            p_2(D_X) \ar[r,hook]& p_2(\overline{Y}(V,\A,\G)) \ar[u,hook]\ar[r,"\pi_2"] & \overline{Y}_2\ar[u,hook] \\
            Y(X^\perp,\A/X) \ar[r] \ar[u,hook] & Y(V,\A')\ar[r] \ar[u,hook] & Y(X^\perp, \A/X) \ar[u,hook]
            \end{tikzcd}
        \]
        Given the definition of $\phi$, the restriction of $\pi_2$ to $Y(V,\A')$ is the projection to $X^\perp$ with kernel $M^\perp$. It follows that the composition of the morphisms in the bottom row is the identity, hence $\pi_2\circ i_2$ is the identity as well by density.\\
        We then have $p_1(D_X)=\overline{Y}_1$ and thus, by the construction of $\pi_1$ in the previous lemma, it follows that $\pi_1\circ i_1$ is the projection onto the first coordinate.
    \end{proof}
    \noindent
    Let us now give an explicit description of the retraction just constructed in coordinates on the standard charts of $\overline{Y}(V,\A,\G)$. Although this does not provide any further geometric insight, it will come in handy a few times during our exposition.\\
    Let us fix a hyperplane arrangement $(V,\A)$ with enough $\G$-modular elements for some building set $\G$. Let $X\in \G$ and $\S$ be a maximal $\G$-nested set containing $X$ together with an adapted basis $\beta \colon  \S\to \A$. Consider the associated affine open subset $U_\S \subseteq \overline{Y}(V,\A,\G)$, on which we put coordinates $u^\S_Y$ for all $Y\in \S$, $Y\neq V^*$. The closure of $\{u^\S_Y=0\}$ in $\overline{Y}(V,\A,\G)$ is precisely $D_Y$. Keep in mind that these coordinates depend on the choice of $\beta$, although it has been suppressed in the notation.\\
    Consider the map 
    \[
        \pi_1 \colon \overline{Y}(V,\A,\G) \to \overline{Y}(V/X^\perp, \A\vert_X, \G\vert_X)
    \]
    constructed in Proposition~\ref{prop: retractions to boundary divisors}. The restriction of $\pi_1$ to $U_\S$ coincides with
    \[
        U_\S\to U_{\S\vert_X}, \quad (u^\S_Y)_{Y\in \S, Y\neq V^* }\mapsto (u^\S_Y)_{Y\in \S, Y\subsetneq X}.
    \]
    The other component of the retraction of Proposition~\ref{prop: retractions to boundary divisors}, namely the map
    \[
        \pi_2: \overline{Y}(V,\A,\G)\to \overline{Y}(X^\perp, \A/X, \G/X),
    \]
    can also be given an elementary description in terms of coordinates, although the situation is slightly more delicate. \\
    Let $M\in \G$ be a $\G$-modular element such that $V^*=X\oplus M$, whose existence is ensured by Proposition~\ref{prop: restrictions into quotients of G-modular elts}. We first need to understand the restriction of the natural map $\overline{Y}(V,\A,\G)\to \overline{Y}(V/M^\perp, \A\vert_M, \G\vert_M)$ to the affine open subset $U_\S$. In contrast to the previous case, the fact that $M\not\in \S$ prevents us from mirroring the description of $\pi_1$.\\
    Via the isomorphism $L(\A/X)\to L(\A\vert_M)$, the image of $\S/X$ is 
    \[
        \S_M=\{ (Y+X)\cap M\mid Y\in \S, Y\not\subseteq X\}.
    \]
    There is an adapted basis $\beta_M \colon  \S_M\to \A\vert_M$ which sends $(\beta(Y)+X)\cap M$ to $Y\in \S, Y\not\subseteq X$.
    \begin{definition}
        We say that $\beta$ is \emph{$M$-adapted} if for all $Y\in \S$ not contained in $X$ we have $\beta(Y)\subseteq M$.
    \end{definition}
    \noindent
    If $\beta$ is $M$-adapted, it follows that $\beta_M(Y)=\beta(Y)$ for all $Y\in \S, Y\not\subseteq X$. Moreover, we also have $\S_M=\S\vert_M$, since for all $Y\in\G, Y\not\subseteq X$ we have $(X+Y)\cap M=Y\cap M$, as we have seen in the proof of Proposition~\ref{prop: restrictions into quotients of G-modular elts}.  \\
    In this case, the projection $\overline{Y}(V,\A,\G)\to \overline{Y}(V/M^\perp, \A\vert_M, \G\vert_M)$ restricts to
    \[
        U_\S\to U_{\S\vert_M}, \quad (u^\S_Y)_{Y\in \S, Y\neq V^* }\mapsto (u^\S_Y)_{Y\in \S, Y\subsetneq M}.
    \]
    The isomorphism $\overline{Y}(V/M^\perp, \A\vert_M, \G\vert_M)\cong \overline{Y}(X^\perp, \A/X, \G/X)$ identifies $U_{\S\vert_M}$ with $U_{\S/X}$ by sending the coordinate $u^{\S\vert_M}_Y$ to $u^{\S/X}_{(Y+X)/X}$ for $Y\in \S$, $Y\subsetneq M$. As a result, the map $\pi_2$ restricts to the morphism
    \[
        U_\S\to U_{\S/X}, \quad (u^\S_Y)_{Y\in \S, Y\neq V^* } \mapsto (u^\S_Y)_{Y\in\S, Y\not\subseteq X, Y\neq V^*}.
    \]
    At the level of regular functions, the retraction $\pi \colon  \overline{Y}(V,\A,\G)\to D_X$ therefore restricts to a map $U_\S\to U_{\S\vert_X}\times U_{\S/X}$ such that 
    \begin{align*}
        u^{\S\vert_X}_Y \mapsto u^\S_Y \quad &   \text{for $Y\in \S\vert_X$, $Y\neq X$,} \\
        u^{\S/X}_{(Y+X)/X} \mapsto u^\S_Y \quad &  \text{for $Y\in \S$, $Y\not\subseteq X$, $Y\neq V^*$}.
    \end{align*}
    Notice that changing the adapted basis $\beta$ for $\S$ yields an automorphism of $U_S$, but the open subset $U_\S$ itself does not depend on the choice of said basis. Thus, it is possible to find a suitable adapted basis $\beta$ for $\S$, namely an $M$-adapted one, so that the retraction constructed in Proposition~\ref{prop: retractions to boundary divisors} locally assumes the simple shape written above.
    The necessity of choosing an $M$-adapted basis mirrors the fact that a different choice of a $\G$-modular element $M$ such that $V^*=X\oplus M$ leads to different projections $V^*\to X^\perp$. \\
    We summarize this discussion in the following
    \begin{corollary}
        \label{cor: embedding regualr functions with enough modular elements}
        Let $(V,\A)$ be a hyperplane arrangement and $\G$ a building set for $L(\A)$. Assume that $(V,\A)$ has enough $\G$-modular elements. Then for every $X\in \G$ and every maximal $\G$-nested set $\S$ containing $X$ the retractions constructed in Proposition~\ref{prop: retractions to boundary divisors} induce injective $k$-algebra homomorphisms
        \[
            \mathcal{O}_{U_{\S\vert_X}}\otimes_k \mathcal{O}_{U_{\S/X}} \to \mathcal{O}_{U_\S}.
        \]
        In particular, the coefficients of the Laurent series of a function $f\in \mathcal{O}_{Y(V,\A)}$ around $\{u^\S_X=0\}$ can be extended to rational functions on $Y(V,\A)$. 
    \end{corollary}
    \begin{proof}
        It remains to clarify only the claim on the coefficients of the Laurent series of a regular function $f\in \mathcal{O}_{Y(V,\A)}$. In a neighborhood of the origin of $U_S$ we may write 
        \[
            f=\sum_{n\ge -N}^\infty a_n(u^\S_Y\mid Y\neq X) (u^\S_X)^n
        \]
        for some $N\ge 0$. We then have
        \[
            a_n=\frac{1}{(n+N)!}\frac{\partial^{n+N}}{\partial (u^\S_X)^{n+N}} \left( (u^\S_X)^N f\right) \big\vert_{u^\S_X=0}\in \mathcal{O}_{U_{\S\vert_X}}\otimes_k \mathcal{O}_{U_{\S/X}},
        \]
        which extends to a rational function on $\mathcal{O}_{Y(V,\A)}$ via the embedding of regular functions provided by a retraction as in Proposition~\ref{prop: retractions to boundary divisors}.
    \end{proof}
    
    \noindent
    We conclude this section by exposing a purely combinatorial property of arrangements with enough modular elements. This result will not be necessary for the rest of the discussion.
    \begin{lemma}
        \label{lemma: predecessors in nested sets}
        Let $(V,\A)$ be a essential hyperplane arrangement, $\G$ a building family for $L(\A)$.
        Let $\S$ be a maximal $\G$-nested set and $X\in \S$. Let $P(X)$ be the set of $Y\in \S$ such that $Y^+=X$. If $(V,\A)$ has enough $\G$-modular elements, then $\#P(X)\le 2$.
    \end{lemma}
    \begin{proof}
        Write $n=\#P(X)$ and let $Y_1,\dots Y_n$ be the elements of $P(X)$. We may assume $X=V^*$. Consider the arrangement $(Y_1^\perp,\A/Y_1,\G/Y_1)$ which has $\S/Y_1$ as a maximal $\G/Y_1$-nested set. The maximal elements of $\S/Y_1\setminus\{V^*/Y_1\}$ are $n-1$ in total, namely $(Y_2+Y_1)/Y_1,\dots, (Y_n+Y_1)/Y_1$. This arrangement still has enough $\G/Y_1$-modular elements, so by induction on the dimension of $V^*$ we deduce that $n\le 3$. Assume by contradiction that $n=3$.\\
        Since $\S\vert_{Y_1}$ is a maximal $\G\vert_{Y_1}$-nested set, there is $H\in \S\vert_{Y_1}$ of dimension $1$. If we have $H\neq Y_1$, the arrangement $(H^\perp,\A/H,\G/H)$ has enough $\G/H$-modular elements, while the maximal $\G/H$-nested set $\S/H$ has $Y_1/H,(Y_2+H)/H$ and $(Y_3+H)/H$ as predecessors of $V^*/H$. Again by induction on the dimension of $Y^*$, these predecessors should be at most two, hence a contradiction. It follows that $Y_1=H$, and thus we may assume that $Y_1,Y_2,Y_3$ have dimension $1$. \\
        By the existence of an adapted basis for $\S$, we have $\dim V^*=1+\sum_{i=1}^3\dim Y_i=4$. Let $Y=Y_1\oplus Y_2 \oplus Y_3$, which is a decomposition by the fact that $\S$ is nested. In particular, $Y_1$, $Y_2$ and $Y_3$ are the unique lines of $\A$ contained in $Y$. We claim that for all $i=1,2,3$ there is a unique $\G$-modular element $M_i\in \G$ in $V^*$ of codimension $1$ which does not contain $Y_i$. Its existence is ensured by the presence of enough $\G$-modular elements. For uniqueness, suppose that $M$ and $N$ are two such elements. Without loss of generality, assume $i=1$. Since $M,N\in \G$, we have $M\neq Y$, $N\neq Y$, thus $\dim M\cap Y=\dim N\cap Y=2$. Since $M$ and $N$ are modular, $M\cap Y\in L(\A)$ and $N\cap Y\in L(\A)$. But $M$ and $N$ do not contain $Y_1$, hence $M\cap Y=N\cap Y=Y_2+Y_3$. If $M\neq N$, then $Y_2+Y_3\subseteq M\cap N\neq M$, thus for dimension reasons $M\cap N=Y_2+Y_3$. Since $M$ is $\G$-modular and $N\in \G$, also $N\cap M=Y_1+Y_2\in \G$, against the fact that $Y_1\oplus Y_2$ is a decomposition.\\
        Let $M_1$, $M_2$ and $M_3$ be as above. These are the unique $\G$-modular element of $\G$ of codimension $1$, for a further one would contain all $Y_1$, $Y_2$, $Y_3$, but $Y\not\in \G$. As above, $M_1\cap Y=Y_2+Y_3$ and $M_2\cap Y=Y_1+Y_3$, so $Y_3\subseteq M_1\cap M_2$. But $M_3$ does not contain $Y_3$, so for dimension reasons $\dim M_1\cap M_2\cap M_3=1$. By the existence of enough $\G$-modular elements, there is a $\G$-modular element $M\in \G$ of codimension $1$ not containing $M_1\cap M_2\cap M_3$. This means that $M$ is different from $M_1$, $M_2$, $M_3$, against uniqueness of the latter as $\G$-modular elements of codimension $1$. 
    \end{proof}
    \begin{example}
        The assumption of $(V,\A)$ with enough $\G$-modular elements in Lemma~\ref{lemma: predecessors in nested sets} is necessary. For instance, let $V^*=\bigoplus_{i=1}^4 k x_i$ with arrangement
        \[
            \A=\left\{ \langle x_1\rangle, \langle x_2\rangle, \langle x_3\rangle,\langle x_4\rangle, \langle x_1-x_4\rangle, \langle x_2-x_4\rangle, \langle x_3-x_4\rangle \right\}.
        \]
        Let $\F$ be the building set of irreducible elements of $L(\A)$.
        The elements of $\F$ of dimension $2$ are $\langle x_1, x_4 \rangle$, $\langle x_2, x_4 \rangle$, $\langle x_3, x_4 \rangle$, while the ones of dimension $3$ are $\langle x_1, x_2, x_4 \rangle$, $\langle x_1, x_3, x_4 \rangle$, $\langle x_2, x_3, x_4 \rangle$. It is straightforward to check that $\S=\{ \langle x_1\rangle, \langle x_2\rangle,\langle x_3\rangle , V^*\}$ is a maximal $\F$-nested set. Notice that $\#P(V^*)=3$, but $(V,\A)$ does not have enough $\F$-modular elements. Indeed, every element of $\F$ of codimension $1$ is modular, but contains $\langle x_4\rangle$.
    \end{example}

    \section{Periods of hyperplane arrangements}
        
    We will now turn to the proof of Theorem~\ref{theo: intro periods of arrangements with enough modular elements} of the introduction or, with a more precise statement, Theorem~\ref{theo: periods of hyperplane arrangements with enough modular elements}. Our strategy generalizes that of Brown \cite{Brown-Multiple_zeta_values_and_periods_of_moduli_spaces_of_curves} for moduli spaces of curves, i.e. for braid arrangements. \\
    Let $(V,\A)$ be a hyperplane arrangement, $\G$ a building set for $L(\A)$; suppose that $(V,\A)$ has enough $\G$-modular elements. In the first section, we introduce the reduced bar complex of $Y(V,\A)$, which provides an algebraic model of an algebra of generalized multiple polylogarithms. This is a differential algebra which a theorem of Brown \cite{Brown-Multiple_zeta_values_and_periods_of_moduli_spaces_of_curves} ensures to be closed under primitives, in a suitable sense. \\
    Starting from the second section, we realize the reduced bar complex as an algebra of multi-valued functions on $Y(V,\A)$ by studying the solutions of a well-known differential equation, the so-called \emph{holonomy equation}. We recall the results of existence and uniqueness of the solutions to this equation in the second section following \cite{De_Concini-Procesi-Hyperplane_arrangements_and_holonomy_equations} and prove some special properties when the arrangement has enough modular elements from section 2.3 to to section 2.6. In the seventh section, we conclude the computation of the periods, while in section 2.8 we make the result more explicit in some specific cases.

    \subsection{The reduced bar complex}
        
    \subsubsection{Definition and first properties}
     Consider an $(l+1)$-arrangement $(V,\A)$ over $k$ and a building set $\G$ for $L(\A)$. To compute the period integrals of $\overline Y(V,\A,\G)$, we need to construct an algebra of multi-valued functions on $Y(V,\A)(\C)$ in which primitives of algebraic differential $l$-forms exist. The reduced bar complex of $\A$ offers a purely algebraic model of this algebra.\\
    We start by recalling some classical facts about the cohomology of $Y(V,\A)$. For all $H\in \A$ fix a non-zero vector $y_H\in H$ and a hyperplane $H_0\in\A$. We get a regular function $x_H=\frac{y_H}{y_{H_0}}$ on $Y(V,\A)$ by regarding $H_0$ as a hyperplane at infinity. Consider the algebraic differential $1$-form
    \[
        \omega_H=d\log x_H = \frac{dx_H}{x_H}=\frac{dy_H}{y_H}-\frac{dy_{H_0}}{y_{H_0}}\in \Omega^1_{Y(V,\A)/k}.
    \]
    Since $\omega_H$ is closed, it induces a cohomology class in $H^1_{\text{dR}}(Y(V,\A),k)$. Notice that $\omega_{H_0}=0$. Brieskorn showed that the de Rham cohomology of $Y(V,\A)$ is completely determined by the forms $\omega_H$ for $H\in \A$.
    \begin{proposition}[{\cite[Lemme 5]{Brieskorn-Sur_les_groupes_de_tresses}}]
        The $k$-algebra $H^*_{\text{dR}}(Y(V,\A),k)$ is naturally isomorphic to the $k$-subalgebra of $\Omega^*_{Y(V,\A)/k}$ generated by the forms $\omega_H$ for all $H\in \A$. 
    \end{proposition}
    \noindent
    In other words, the subcomplex of $\Omega_{Y(V,\A)/k}^*$ generated by the forms $\omega_H$ for $H\in \A$ is quasi-isomorphic to $\Omega_{Y(V,\A)/k}^*$ itself. There is a purely combinatorial description of $H^*_{\text{dR}}(Y(V,\A),k)$ by the so-called Orlik-Solomon algebra. Since this will not be necessary for the rest of the exposition, we refer the reader to \cite{Orlik-Terao-Arrangements_of_hyperplanes} for further details. \\
    Set for short $Y=Y(V,\A)$. For all $n\ge 1$ and $j=1,\dots, n-1$, let $N_j$ be the kernel of
    \begin{align*}
        H^1_{\text{dR}}(Y,k)^{\otimes n} & \longrightarrow H^1_{\text{dR}}(Y,k)^{\otimes (j-1)} \otimes_k H^2_{\text{dR}}(Y,k) \otimes_k H^1_{\text{dR}}(Y,k)^{\otimes( n -j-1)} \\
        \omega_{1}\otimes \dots \otimes \omega_{n} & \longmapsto 
        \omega_{1}\otimes \dots \otimes \omega_{j-1} \otimes (\omega_j\wedge \omega_{j+1})\otimes \omega_{j+2}\otimes \dots\otimes \omega_{n}.
    \end{align*}
    Define $B_n(V,\A)=\bigcap_{j=1}^{n-1} N_j$. Thus, $B_n(V,\A)$ consists of all the elements 
    \[
        \sum_{I=(i_1,\dots, i_n)} c_I \omega_{i_1}\otimes \dots \otimes \omega_{i_n} \in H^1_{\text{dR}}(Y,k)^{\otimes n}
    \]
    such that for all $j=1,\dots, n-1$
    \[
        \sum_{I=(i_1,\dots, i_n)} c_I \omega_{i_1}\otimes \dots \otimes \omega_{i_{j-1}} \otimes (\omega_{i_j}\wedge \omega_{i_{j+1}})\otimes \omega_{i_{j+2}}\otimes \dots\otimes \omega_{i_n}=0.
    \]
    Also set $B_0(V,\A)=k$. For the element $\omega_1\otimes \dots \otimes \omega_n$ of $B_n(V,\A)$ we will use the bar notation and denote it by $[\omega_1\vert \dots \vert \omega_n]$.
    \begin{definition}
        The \emph{reduced bar complex} of $(V,\A)$ is the $k$-vector space
        \[
            B(V,\A)=\bigoplus_{n\ge 0} B_n(V,\A).
        \]
        We also introduce the notation
        \[
            B'(V,\A)=B(V,\A)\otimes_k \mathcal{O}_{Y}.
        \]
    \end{definition}
    \noindent
    Notice that this convention differs slightly from the one of \cite{Brown-Multiple_zeta_values_and_periods_of_moduli_spaces_of_curves}. We will sometimes refer to $B'(V,\A)$ as the reduced bar complex, as well. The elements of $B_n(V,\A)$ are said to have \emph{weight} $n$. The $k$-vector space $B'(V,\A)$ can be endowed with the structure of a differential $k$-algebra, while $B(V,\A)$ with that of a graded differential Hopf-algebra. We briefly summarize all structures involved.\\
    $B(V,\A)$ is a $k$-algebra when equipped with the shuffle product, which we now define. Let $k\langle\omega_H\mid H\in \A\setminus\{H_0\}\rangle $ be the set of non-commutative polynomials with coefficients in $k$ in the indeterminates $\omega_H$ for $H\in \A\setminus\{H_0\}$. Consider two monomials $x,y\in k\langle\omega_H\mid H\in \A\setminus\{H_0\}\rangle $ of degree $m-1$ and $n-1$ respectively; take also $H, K\in \A\setminus\{H_0\}$. The shuffle product of $\omega_Hx$ and $\omega_K y$ is defined inductively on $n+m$ by
    \[
        \omega_Hx \,\sh\,\omega_Ky= \omega_H(x \,\sh\, \omega_K y)+\omega_K(\omega_Hx \,\sh\, y),
    \]
    together with $x \,\sh\, 1=1 \,\sh\, x = x$ and $1 \,\sh\, 1 =1$. This definition can be extended to the whole $k\langle\omega_H\mid H\in \A\setminus\{H_0\}\rangle$ by $k$-linearity. Once equipped with the shuffle product, $k\langle \omega_H\mid H\in \A\setminus\{H_0\}\rangle$ turns into a commutative $k$-algebra, the so-called free shuffle $k$-algebra on the set $\{\omega_H \mid H\in\A\setminus\{H_0\}\} $. \\
    One can check that $B(V,\A)$ naturally embeds in $k\langle\omega_H\mid H\in \A\setminus\{H_0\}\rangle $ with image closed under the shuffle product. Thus, we may endow $B(V,\A)$ with the structure of a commutative $k$-algebra, which is also graded by the weight.\\
    For convenience, we also introduce the notation
    \[
        W_n B(V,\A)=\bigoplus_{m=0}^n B_n(V,\A)
    \]
    for all $n\ge 0$. Thus, $W_\bullet B(V,\A)$ is an increasing filtration of $B(V,\A)$.
    \begin{remark}
        To have a description of the reduced bar complex which is independent of the choice of a hyperplane $H_0$ at infinity, consider formal variables $\omega_H$ for all $H\in \A$. Let $\beta$ be any basis over $k$ of the subspace of the vector space $\bigoplus_{H\in \A} k\,\omega_H$ given by those vectors whose coordinates sum to zero. The reduced bar complex then embeds in the free shuffle $k$-algebra over $\beta$. The explicit description given so far coincides with the choice $\beta=\{\omega_H-\omega_{H_0}\mid H\in \A\setminus\{H_0\}\}$.
    \end{remark}
    \begin{remark}
        If $Y(V,\A)$ is one-dimensional, then $B(V,\A)$ coincides with the whole free shuffle algebra $k\langle\omega_H\mid H\in \A\setminus\{H_0\} \rangle  $.
    \end{remark}
    \noindent
    We can also put a structure of bialgebra on $B(V,\A)$. As counit we consider the linear projection onto $B_0(V,\A)=k$. The coproduct is given by the deconcatenation coproduct: for $I=(i_1,\dots,i_n)$ and $\omega_I=[\omega_{i_1}\vert\dots\vert \omega_{i_n}]$
    \[
        \Delta \left(\sum_{I}c_I\omega_I\right)= \sum_{I}\sum_{[\phi_I\vert\psi_I]=\omega_I} c_I \phi_I\otimes \psi_I.
    \]
    This actually turns $B(V,\A)$ into a Hopf algebra, the antipode being given by
    \[
        S\left(\sum_{I=(H_1,\dots, H_n)} c_I \left[ \omega_{H_1}\middle\vert \dots \middle\vert \omega_{H_n} \right] \right) = \sum_{I=(H_1,\dots, H_n)} (-1)^nc_I \left[ \omega_{H_n}\middle\vert \dots \middle\vert \omega_{H_1} \right]
    \]
    for all $n\ge 0$. The Hopf algebra structure is compatible with the grading induced by the weight.\\
    There is a differential
    \[
        d \colon  \Omega^h_{Y/k}\otimes_{\mathcal{O}_Y} B'(V,\A)\to \Omega^{h+1}_{Y/k} \otimes_{\mathcal{O}_Y}  B'(V,\A)
    \]
    which maps
    \[
        \sum_{I=(i_1,\dots, i_n)} \phi_I \otimes [\omega_{i_1}\vert\dots\vert \omega_{i_n}]
    \]
    to the element
    \[
        \sum_I (-1)^{\deg \phi_I}\phi_I\wedge \omega_{i_1} \otimes [\omega_{i_2}\vert\dots\vert \omega_{i_n}] + \sum_{I=(i_1,\dots, i_n)} d\phi_I \otimes [\omega_{i_1}\vert\dots\vert \omega_{i_n}].
    \]
    This induces $l$ commuting derivations $\partial_1,\dots,\partial_l$ on $B'(V,\A)$, which becomes a differential $k$-algebra. It is also graded by the weight.
    \begin{remark}
        The reduced bar complex can be defined in much greater generality and can be given a deeper geometric meaning. For further details, we refer to \cite{Burgos-Gil-Fresan-Multiple_zeta_values_from_numbers_to_motives}, with the caveat that the algebra just introduced is the cohomology in degree $0$ of the reduced bar complex defined there.
    \end{remark}
    \noindent
    The de Rham cohomology groups $H^i(B'(V,\A),k)$ of $B(V,\A)$ are defined to be the cohomology of the complex 
    \[
        B'(V,\A)\to \Omega^1_{Y/k}\otimes_{\mathcal{O}_Y} B'(V,\A) \to \Omega^2_{Y/k}\otimes_{\mathcal{O}_Y} B'(V,\A) \to \dots
    \]
    A similar definition of cohomology groups can be given for an arbitrary differential $k$-algebra $A$, that is, a commutative $k$-algebra $A$ equipped with $l$ commuting derivations $\partial_1, \dots, \partial_l$. We now recall a few basic definitions regarding differential algebras following \cite{Brown-Multiple_zeta_values_and_periods_of_moduli_spaces_of_curves}.\\
    A differential $k$-algebra is said to be \emph{differentially simple} if $H^0(A,k)=k$ and $A$ is a simple module over the ring $A[\partial_1,\dots, \partial_l]$. It is not difficult to see that $\mathcal{O}_Y$ is a differentially simple $k$-algebra when endowed with its standard structure of differential algebra.
    \begin{definition}
        Let $A$ be a differentially simple $k$-algebra. A \emph{unipotent extension} of $A$ is a differential $k$-algebra $B$ such that 
        \begin{enumerate}
            \item $A$ is a differential subalgebra of $B$;
            \item There exists an increasing filtration $W^\bullet B$ of $B$ by differential subalgebras
            \[
                A=W^0B\subseteq W^1B\subseteq \dots \subseteq B
            \]
            such that $B=\bigcup_i W^iB$ and $W^{i+1}B$ is a $W^iB$-algebra generated by finitely many elements whose derivatives lie in $W^iB$.
        \end{enumerate} 
    \end{definition}
    \noindent
    Every unipotent extension of a differentially simple algebra is a polynomial algebra. The main reason for our interest in the reduced bar complex is the following result.
    \begin{theorem}[{\cite[Theorem 3.26]{Brown-Multiple_zeta_values_and_periods_of_moduli_spaces_of_curves}}]
        \label{theo: reduced bar complex is unipotent closure}
        $B'(V,\A)$ is a unipotent closure of $\mathcal{O}_Y$, that is,
        \begin{enumerate}
            \item $B'(V,\A)$ is a unipotent extension of $\mathcal{O}_Y$;
            \item $H^0(B'(V,\A),k)=k$;
            \item $H^1(B'(V,\A),k)=0$.
        \end{enumerate}
    \end{theorem}
    \noindent
    In particular, the last property ensures that in the reduced bar complex it is always possible to find $1$-primitives: given $f_1,\dots,f_l\in B'(V,\A)$ such that $\partial_i f_j=\partial_j f_i$ there exists $f\in B'(V,\A)$ such that $\partial_if=f_i$ for all $i=1,\dots, l$. This operation increases the weight by at most $1$.
    
    \subsubsection{The reduced bar complex of fiber type arrangements}
    If $(V,\A)$ is fiber type, the higher cohomology groups of $B'(V,\A)$ also vanish, as we now see.\\
    Let us fix $H\in \A$ and set for brevity $Y=Y(V,\A)$ and $Y'=Y(H^\perp, \A/H)$. Assume that there exists a modular element $M\in L(\A)$ such that $V^*=M\oplus H$. By Corollary~\ref{cor: projections to restrictions}, it follows that the quotient map $V\to V/M^\perp$ induces a projection $\pi \colon Y(V,\A)\to Y(H^\perp, \A/H)$ which is also a fibration. It is not difficult to see that, as an $\mathcal{O}_{Y'}$-algebra via $\pi$, $\mathcal{O}_Y$ takes the form
    \[
        \mathcal{O}_Y=\mathcal{O}_{Y'}[x]\left[ \frac{1}{x-f_1}, \dots, \frac{1}{x-f_m} \right]
    \]
    for some $f_1,\dots,f_n\in \mathcal{O}_{Y'}$ such that $f_i-f_j$ is invertible in $\mathcal{O}_{Y'}$ for all $i\neq j$.\\
    For $i=1,\dots, n$ set $\alpha_i=d\log(x-f_i)$ and consider the free shuffle $\mathcal{O}_{Y'}$-algebra
    \[
        B_{Y'}(Y)= \mathcal{O}_{Y'}\otimes_k k\left\langle \alpha_1,\dots, \alpha_m \right\rangle.
    \]
    We may endow $B_{Y'}(Y)$ with a derivation $\partial_x$ as follows:
    \[
        \partial_x(f\otimes [ \alpha_{i_1}\vert \dots \vert \alpha_{i_n}])=\partial_xf \otimes [ \alpha_{i_1}\vert \dots \vert \alpha_{i_n}] + \frac{f}{x-f_{i_1}}\otimes [\alpha_{i_2}\vert\dots \vert \alpha_{i_n}].
    \]
    The standard derivations of $\mathcal{O}_{Y'}$ may also be extended to derivations of $B_{Y'}(Y)$, although in a more subtle way. Consider the ring of Laurent series $\mathcal{O}_{Y'}[[\epsilon]][\epsilon^{-1}]$ with its standard derivation $\partial_\epsilon$. We adjoin one formal variable $\log \epsilon$ to this ring and we extend the derivation $\partial_\epsilon$ uniquely by the rule $\partial_\epsilon \log \epsilon = \epsilon^{-1}$. \\
    Fix an $\mathcal{O}_{Y'}$-linear map $p \colon  \mathcal{O}_Y\to \mathcal{O}_{Y'}[[\epsilon]][\epsilon^{-1}]$ such that $p\circ \partial_x=\partial_\epsilon\circ p$. Then $p$ extends uniquely to a map
    \[
        \phi \colon  B_{Y'}(Y) \to \mathcal{O}_{Y'}[[\epsilon]][\epsilon^{-1}][\log\epsilon]
    \]
    such that $\phi([ \alpha_{i_1}\vert \dots \vert \alpha_{i_n}])$ has zero constant coefficient for all choices of $n\ge 1$ and $i_1,\dots, i_n$. By \cite[Proposition~3.23]{Brown-Multiple_zeta_values_and_periods_of_moduli_spaces_of_curves}, the standard derivations on $\mathcal{O}_{Y'}$ extend uniquely to derivations on $B_{Y'}(Y)$ that commute with $\phi$ and with each other. Note that the definition of these derivations depends on the choice of $p$.\\
    The differential algebra $B_{Y'}(Y)$ is a relative unipotent closure of $\mathcal{O}_{Y'}$, in the sense that 
    \[
        H^0(B_{Y}(Y),k)=\mathcal{O}_{Y'} \qquad \text{and} \qquad H^1(B_{Y'}(Y),k)=0.
    \]
    \begin{proposition}[{\cite[Corollary 3.37]{Brown-Multiple_zeta_values_and_periods_of_moduli_spaces_of_curves}}]
        \label{prop: reduced bar complex for supersolvable arrangements}
        In the setting just introduced, there is an isomorphism of differential $k$-algebras
        \[
            B'(V,\A) \cong B'(H^\perp,\A/H)\otimes_{\mathcal{O}_{Y'}} B_{Y'}(Y).
        \]
    \end{proposition}
    \noindent
    Thus, the fact that $(V,\A)$ fibers over $(H^\perp, \A/H)$ implies that the reduced bar complex decomposes as the tensor product of $B'(H^\perp,\A/H)$ and a one-dimensional reduced bar complex, corresponding to the fiber of the projection $\pi \colon  Y\to Y'$. It is important to note that the isomorphism in this proposition depends on several choices, such as, for example, $\pi$ and $p$.\\
    If $(V,\A)$ is fiber type, one may find a sequence of linear fibrations
    \[
        Y=Y_1\overset{\pi_1}{\longrightarrow} Y_2 \overset{\pi_2}{\longrightarrow} \dots \overset{\pi_{l-1}}{\longrightarrow} Y_{l-1},
    \]
    where each $Y_i$ is the complement of a hyperplane arrangement and $\pi_i$ is a fiber bundle projection with one-dimensional fibers. In this situation it follows that 
    \[
        B'(V,\A)=B_{Y_2}(Y_1)\otimes_{\mathcal{O}_{Y_2}} B_{Y_3}(Y_2)\otimes_{\mathcal{O}_{Y_3}} \dots \otimes_{\mathcal{O}_{Y_{l-2}}}  B_{Y_{l-1}}(Y_{l-2})\otimes_{\mathcal{O}_{Y_{l-1}}}B(Y_{l-1}).
    \]
    Since all factors in the right-hand side are relative unipotent closures, we deduce the following
    \begin{corollary}[{\cite[Corollary 3.39]{Brown-Multiple_zeta_values_and_periods_of_moduli_spaces_of_curves}}]
        \label{cor: vanishing top degree cohomology reduced bar complex}
        If $(V,\A)$ is a fiber type hyperplane arrangement, then $H^{i}(B'(V,\A))=0$ for all $i=1,\dots, l$. A primitive of an $h$-form in $\Omega^h_{Y/k}\otimes_{\mathcal{O}_Y} B'(V,\A)$ of weight $n$ has weight at most $n+1$.
    \end{corollary} 
    \noindent
    For our arguments, a crucial point will be the vanishing of the cohomology in top degree for the reduced bar complex of a fiber type arrangement. Recall that supersolvable arrangements are fiber type.
    
    \subsubsection{Analytic manifolds with corners}
    In this section we also fix some conventions about manifolds with corners which will be used later in the exposition. Our definitions are the same as those of \cite{Brown-Multiple_zeta_values_and_periods_of_moduli_spaces_of_curves}. \\
    For $p,q\ge 1$ let $U_{p,q}=\R^p\times \R_+^q$, where $\R_+=\{x\in \R \mid x\ge 0\}$. The boundary of $U_{p,q}$ is 
    \[
        \partial U_{p,q}=\bigcup_{i+j=q-1} \R^p\times \R_+^i\times \{0\}\times\R_+^j.
    \]
    An analytic manifold with corners is a topological space that is locally homeomorphic to $U_{p,q}$ for some $p,q\ge 1$ and whose transition maps are appropriately defined analytic functions $U_{p,q}\to U_{p,q}$. We will now make precise how these latter functions are defined. Notice that, topologically, $\R_+^2$ is homeomorphic to $\R\times \R_+$, while these two spaces should not be isomorphic as manifolds with corners.\\
    Write $n=p+q$ and let $x_1,\dots, x_n$ be coordinates on $\R^n$ such that $U_{p,q}$ is defined by $x_i\ge 0$ for $i=1,\dots,q$. An analytic isomorphism $\phi=(\phi_1,\dots, \phi_n) \colon U_{p,q}\to U_{p,q}$ is a homeomorphism that enjoys the following properties:
    \begin{enumerate}
        \item $\phi$ is an analytic isomorphism, that is, it is an analytic function on an open neighborhood of $U_{p,q}$ in $\R^n$ with nowhere vanishing Jacobian on $U_{p,q}$;
        \item $\phi$ permutes the components of the boundary of $\partial U_{p,q}$. Thus, there is a permutation $\sigma$ of $\{1,\dots, q\}$ such that $\phi_i$ vanishes identically on $x_{\sigma(i)}=0$.
        \item for all $i=1,\dots, q$ we have 
        \[
            \frac{\partial\phi_i}{\partial x_{\sigma(i)}}\bigg\vert_{x_{\sigma(i)}=0}=1.
        \]
    \end{enumerate}
    The last condition is needed to have well-defined regularizations along the boundary, as we will explain later. The second condition implies that an analytic isomorphism $\phi \colon U_{p,q}\to U_{p,q}$ preserves the boundary $\partial U_{p,q}$. An analytic manifold with corners $X$ admits a stratification $X=X_1\supseteq X_2\supseteq \dots \supseteq X_{n+1}$, where $X_i$ is a manifold with corners and $\partial X_{i}=X_{i+1}$.
    \begin{remark}
        This definition of a manifold with corners, which agrees with the one of \cite{Brown-Multiple_zeta_values_and_periods_of_moduli_spaces_of_curves}, may seem somewhat restrictive, due to the first order behavior of the transition functions. A more general framework is offered by \cite{Dupont-Panzer-Prym-Regularized_integrals_and_manifold_with_log_corners}, but Brown's setup will be sufficient for our applications.
    \end{remark}
    For $p,q\ge 1$ let $\F(U_{p,q})^{\text{an}}$ be the ring of analytic functions on $U_{p,q}$ and 
    \[
        \F(U_{p,q})^{\log}=\F(U_{p,q})^{\text{an}}[\log x_1,\dots \log x_q],
    \] 
    where $\log x_i$ is the principal branch of the logarithm. For $i=1,\dots, q$ and $f\in \F(U_{p,q})^{\text{log}}$, we define the regularization of $f$ along $x_i=0$ to be the function in $\F(\{x_i=0\})^{\log}$ obtained by formally setting to zero $\log x_i$ and then restricting to $x_i=0$. This shall be denoted by $\text{Reg}(f,\{x_i=0\})$. \\
    Let $X$ be a manifold with corners. The rings $\F(U_{p,q})^{\text{an}}$ and $\F(U_{p,q})^{\log}$ glue to sheaves of rings $\F^{\text{an}}$ and $\F^{\log}$ on $X$. Let $D$ be a component of the boundary of $X$. The regularization maps just described yield a homomorphism $\text{Reg}(-,D): \F^{\log}(X)\to \F^{\log}(D)$. For this to be well defined, it is necessary to impose condition (3) in the definition of analytic isomorphism given before. For more details on these constructions, we refer to \cite[Section 4]{Brown-Multiple_zeta_values_and_periods_of_moduli_spaces_of_curves}.\\
    Analytic manifolds with corners provide a natural setup for Stokes' theorem. The reason why we allow logarithms is that we will need to integrate functions which have logarithmic singularities along the boundary of a manifold with corners. To be precise, the main technical tool that we will exploit is the following version of Stokes' theorem.
    \begin{theorem}[{\cite[Theorem 4.11]{Brown-Multiple_zeta_values_and_periods_of_moduli_spaces_of_curves}}]
        \label{theo: Stokes}
        Let $X$ be an analytic manifold with corners of dimension $n$. Let $\omega$ be a differential $n$-form on $X$ with coefficients in $\F^{\log}(X)$. Suppose that there is a differential $(n-1)$-form $\Omega$ on $X$ with coefficients in $\F^{\log}(X)$ such that $d\Omega=\omega$. Then
        \begin{enumerate}
            \item $\Omega$ extends continuously to the interior of $\partial X$;
            \item $\omega$ is integrable on $X$ and 
            \[
                \int_X \omega = \int_{\partial X} \Omega.
            \]
        \end{enumerate}
    \end{theorem}
    \noindent
    We point out that the condition for $\Omega$ to have coefficients in $\F^{\log}(X)$ implies, in particular, that $\Omega$ cannot have poles along $\partial X$. In the one-dimensional case, the fact that $\Omega$ extends continuously to the interior of $\partial X$ is due to the fact that the differential form $\log x\, dx$ has a logarithmic singularity at $0$, but any primitive thereof, such as $x(\log x-1)$, extends continuously at $0$.

    \subsection{The holonomy equation}
        
    The goal of this section is to introduce the holonomy equation and discuss some classical results regarding its solutions. We mainly follow the exposition of \cite{De_Concini-Procesi-Hyperplane_arrangements_and_holonomy_equations}.
    
    \subsubsection{Existence and uniqueness of solutions}
    Let $k$ be a subfield of $\C$. Let $(V,\A)$ be a hyperplane arrangement over $k$, which we assume as usual essential. 
    \begin{definition}
        Let $A(V,\A)$ be the free non-commutative $k$-algebra on the set of symbols $t_H$ for all $H\in \A$ modulo the following relations: for all $H\in \A$, $X\in L(\A)$ with $\dim X= 2$ and $H\subseteq X$ 
        \[
        \left[ t_H,  \sum_{K\in \A, K\subseteq X} t_K\right]=0, \qquad \sum_{K\in \A}t_K=0.
        \]
        The \emph{holonomy algebra of $(V,\A)$} is the completion $\widehat{A}(V,\A)$ of $A(V,\A)$ at the two-sided ideal generated by all $t_H$ for $H\in \A$.   
    \end{definition}
    \noindent
    To simplify notation, for every $X\in L(\A)$ we shall write
    \[
        t_X=\sum_{H\in \A, H\subseteq X} t_H.
    \]
    Notice that $\widehat{A}(V,\A)$ is a graded possibly non-commutative $k$-algebra. Moreover, consider the free Lie algebra over $k$ on the symbols $t_H$ for $H\in \A$ and let $\f{g}$ be its quotient modulo the Lie subalgebra generated by $[t_H, t_X]$ for all $X\in L(\A), H\in \A$ with $\dim X=2$ and $H\subseteq X$. It is clear that $A(V,\A)$ is the universal enveloping algebra of $\f{g}$. As such, $A(V,\A)$ is a Hopf algebra when endowed with the coproduct $\Delta$ and the counit $\epsilon$ defined by
    \[
        \Delta(t_H)=t_H\otimes 1 + 1\otimes t_H, \qquad \epsilon(t_H)=0
    \]
    for all $H\in \A$. The antipode then maps $t_H$ to $-t_H$. It is then readily checked that $\widehat{A}(V,\A)$ carries an induced structure of completed Hopf algebra.
    \begin{remark}
        The holonomy algebra considered here is that of \cite{De_Concini-Procesi-Hyperplane_arrangements_and_holonomy_equations}, but with the extra relation $t_{V^*}=0$. The latter is needed to work with the projectivization $Y(V,\A)$ of the affine complement $X(V,\A)$, as in \cite{De_Concini-Procesi-Hyperplane_arrangements_and_holonomy_equations} only the case of affine complements of central arrangements is considered.
    \end{remark}
    \noindent
    The following lemma concerning the relations between the $t_H$'s will be used several times throughout our exposition.
    \begin{lemma}[ {\cite[Theorem 2.1]{De_Concini-Procesi-Hyperplane_arrangements_and_holonomy_equations}} ]
        \label{lemma: commutativity relations in holonomy algebra}
        Let $X,Y\in L(\A)$. Suppose that $X$ and $Y$ are $\G$-nested for some building set $\G$ for $L(\A)$ containing both $X$ and $Y$. Then $[t_X, t_Y]=0$.
    \end{lemma}
    \noindent
    We now introduce the holonomy equation. First, for every $H\in \A$ fix a non-zero vector $y_H\in H$. Given $H_0\in \A$, let us set $x_H=\frac{y_H}{y_{H_0}}$ and consider the differential form $\omega_H=d\log x_H$ on $Y(V,\A)$. Let us formally define $\widetilde{\omega}_H=d\log y_H$ and consider the vector space $W=\bigoplus_{H\in \A} k\,\widetilde{\omega}_H$. Notice that $H^1_{\text{dR}}(Y(V,\A),k)$ is the subspace of $W$ consisting of those vectors whose coordinates sum to zero. If $t_H$ denotes the element of $W^*$ dual to $\widetilde{\omega}_H$ with respect to the basis $\{\widetilde{\omega}_H\mid H\in \A\}$, then $H^1_{\text{dR}}(Y(V,\A),k)$ is the kernel of $T=\sum_{H\in \A} t_H$. The map $W^*\otimes_k W\to W^*/(T)\otimes_k W$ induced by the quotient $W^*\to W ^*/(T)$ sends the element 
    \[
        \sum_{H\in \A} t_H \otimes \widetilde{\omega}_H
    \]
    to an element in the smaller subspace $W^*/(T)\otimes_k \Omega^1_{Y(V,\A)/k} $, namely
    \[
        \sum_{H\in \A\setminus\{H_0\}} t_H\otimes (\widetilde{\omega}_H-\widetilde{\omega}_{H_0})=
        \sum_{H\in \A\setminus\{H_0\}} t_H\otimes \omega_H.
    \]
    By abuse of notation, we will still denote this by
    \[
        \Omega= \sum_{H\in \A} t_H\otimes \omega_H \in A(V, \A)\otimes_k \Omega^1_{Y(V,\A)/k}.
    \]
    This argument shows that the definition of $\Omega$ is independent of $H_0$. we regard $\Omega$ as a formal differential form over $Y(V,\A)$. The relations in the definition of the holonomy algebra ensure that $\Omega$ is integrable, thus $\Omega\wedge \Omega=0$. The \emph{holonomy equation} associated with $(V,\A)$ is the formal differential equation 
    \[
        dF=\Omega F,
    \]
    where $F$ is a multi-valued function on the complex points of $Y(V,\A)$ with values in $\widehat{A}(V,\A)\widehat{\otimes}_k \C$.\\
    To spell this out in more detail, denote by $\widehat{Y}(V,\A)$ the universal cover of the complex points of $Y(V,\A)$. Fix a graded basis $\beta$ of $A(V,\A)$ as a $k$-vector space. Every element of $f\in\widehat{A}(V,\A)$ can be written uniquely as $f=\sum_{b\in \beta} \alpha_b\, b$
    for some $\alpha_b\in k$. A function $F \colon \widehat{Y}(V,\A)\to \widehat{A}(V,\A)\widehat{\otimes}_k\C$ is holomorphic if, when writing $F(z)=\sum_{b\in\beta} \alpha_b(z) \, b$, the functions $\alpha_b \colon  \widehat{Y}(V,\A)\to \C$ are holomorphic. Such $F$ is then called a multi-valued function on $Y(V,\A)$ with values in $\widehat{A}(V,\A)\widehat{\otimes}_k \C$. Notice that this definition does not depend on the choice of a graded basis for $A(V,\A)$.\\
    To ensure uniqueness of the solutions of the holonomy equation, we fix the asymptotic behaviour in a neighbourhood of a point in $\overline{Y}(V,\A,\G)$ which is a maximal intersection of irreducible boundary divisors. More precisely, fix a building set $\G$ for $L(\A)$ and let $\S$ be a maximal $\G$-nested set. Write $D_\S$ for the point which is the intersection of all irreducible boundary divisors $D_X$ for $X\in\S$. Moreover, let $u^\S_X$ for $X\in \S, X\neq V^*$ be the standard coordinates for the local affine chart $U_\S$ of $\overline{Y}(V,\A)$ with respect to the choice of an adapted basis for $\S$, which we fix once and for all. Given a multi-valued function $F$ on the complex points of $Y(V,\A)$ with values in $\widehat{A}(V,\A)\widehat{\otimes}_k \C$, we write
    \[
        F\sim \prod_{X\in \S, X\neq V^*} \left(u_X^\S\right)^{t_X}
    \]
    if there is a holomorphic function $f$ defined on a neighborhood of $D_\S$ in $\overline{Y}(V,\A,\G)$ in the analytic topology and with values in $\widehat{A}(V,\A)\widehat{\otimes}_k \C$ which takes the value $1$ at $D_\S$ and satisfies 
    \[
        F=f \prod_{X\in \S, X\neq V^*} \left(u_X^\S\right)^{t_X}.
    \]
    Here, for $X\in \S$, $X\neq V^*$ the notation $\left(u_X^\S\right)^{t_X}$ refers to the power series
    \[
        \left(u_X^\S\right)^{t_X}=e^{t_X\log u_X^\S}=\sum_{n=0}^\infty \left(\log u_X^\S\right)^n \frac{t_X^n}{n!},
    \]
    where the logarithm is viewed as a multi-valued function on $Y(V,\A)$. Notice that the product appearing above does not depend on the order of the factors; indeed, since $\S$ is $\G$-nested, we have $[t_X,t_Y]=0$ for all $X,Y\in \S$ by Lemma~\ref{lemma: commutativity relations in holonomy algebra}. We also remark that the function $f$ is required to be single-valued, so that it makes sense to ask for it to extend to the irreducible boundary divisors $D_X$ for $X\in \S$, $X\neq V^*$ in a neighborhood of $D_\S$.\\
    With this notation, the holonomy equation \emph{localized at $\S$}, or \emph{with boundary conditions at $D_\S$}, is the following differential equation with boundary conditions:
    \[
    \begin{cases}
        dF=\Omega F, \\
        F\sim \prod_{X\in S, X\neq V^*} \left(u_X^\S\right)^{t_X}.
    \end{cases}
    \]
    \begin{theorem}[{\cite{De_Concini-Procesi-Hyperplane_arrangements_and_holonomy_equations}}]
        \label{theo: solutions of holonomy equation}
        The holonomy equation localized at $\S$ admits a unique solution $L_\S$, which restricts to a holomorphic function on every simply connected open subset of $Y(V,\A)$.\\
        Moreover, the holomorphic function 
        \[
            f_\S= L_\S\prod_{X\in \S, X\neq V^*} \left(u_X^\S\right)^{-t_X}
        \]
        is \emph{unipotent}, i.e. it is a group-like element of $\widehat{A}(V,\A)$, so $\Delta f_\S=f_\S\otimes f_\S$.
    \end{theorem}
    \begin{proof}
        In \cite{De_Concini-Procesi-Hyperplane_arrangements_and_holonomy_equations} a solution $\tilde{L}_\S$ is given for the holonomy equation for $\overline{X}(V,\A,\G)$ with boundary condition at a maximal $\G$-nested set $\S$. In this case, the holonomy algebra is the same as the one proposed here, leaving out the relation $t_{V^*}=0$. The function $\tilde{L}_\S (u_{V^*}^\S)^{-t_{V^*}}$ extends to the open part of the divisor $D_{V*}\cong \overline{Y}(V,\A,\G)$ and yields the sought solution $L_\S$.
    \end{proof}

    \subsubsection{Restriction to boundary divisors}
    We now turn to the study of the restriction of a solution to the holonomy equation on boundary divisors. Given a building set $\G$ for $L(\A)$ and $X\in \G$, $X\neq V^*$, recall that by Proposition~\ref{prop: buondary divisor as product of compactifications} there is an isomorphism
    \[
        D_X\cong \overline{Y}(V/X^\perp, \A\vert_X, \G\vert_X)\times \overline{Y}(X^\perp, \A/X, \G/X).
    \]
    Let $\S$ be a maximal $\G$-nested set containing $X$ and fix an adapted basis $\beta$ for $\S$. We then have two maximal nested sets $\S\vert_X$ and $\S/X$ for $\G\vert_X$ and $\G/X$ respectively, as given in Lemma~\ref{lemma: restrictions and quotients of maximal nested sets}. These also come with canonically induced adapted bases and by Proposition~\ref{prop: local decomposition of boundary divisors} the above isomorphism restricts to an isomorphism at the level of affine charts
    \[
        U_\S\cap D_X\cong U_{\S\vert_X}\times U_{\S/X}.
    \]
    We wish to relate the solution $L_\S$ of the holonomy equation localized at $\S$ to the functions $L_{\S\vert_X}$ and $L_{\S/X}$ arising as solutions to the holonomy equations for $Y(V/X^\perp, \A\vert_X)$ and $Y(X^\perp, \A/X)$ localized at $\S\vert_X$ and $\S/X$ respectively. The function $L_\S \left(u_X^\S\right)^{-t_X}$ extends to a meromorphic function on $D_X$ and we will show that its restriction to $D_X$ coincides with $L_{\S\vert_X}L_{\S/X}$. Before doing so, we need to take care of a few details, since, for instance, these functions take values in different algebras.\\
    We start by showing that there is a natural embedding of graded $k$-algebras
    \[
        \widehat{A}(V/X^\perp, \A\vert_X)\widehat{\otimes}_k\widehat{A}(X^\perp, \A/X) \hookrightarrow \widehat{A}(V,\A).
    \]
    First, consider the $k$-algebra homomorphism
    \begin{gather*}
        i_1 \colon  A(V/X^\perp,\A\vert_X)\to A(V,\A),\\
        t_H\mapsto 
        \begin{cases}
            t_H & \text{if $H\neq \beta(X)$};\\
            t_{\beta(X)}+\sum_{K\in \A, K\not\subseteq X} t_K & \text{if $H=\beta(X)$}.
        \end{cases}
    \end{gather*}
    for all $H\in \A$, $H\subseteq X$.
    \begin{lemma}
        The map $i_1$ is well-defined, injective and it extends to a homomorphism $\widehat{A}(V/X^\perp,\A\vert_X)\to \widehat{A}(V,\A)$.
    \end{lemma}
    \begin{proof}
        To prove that $i_1$ is well-defined, first observe that $i_1(t_X)=t_{V^*}=0$. Moreover, we need to check that for all $H\in\A$ and $Y\in L(\A)$ of dimension $2$ such that $H\subseteq Y\subseteq X$ we have $[i_1(t_H),i_1(t_Y)]=0$.\\
        If $Y$ does not contain $\beta(X)$, then $i_1(t_H)=t_H$ and $i_1(t_Y)=t_Y$, so the claim follows from the relations in $A(V,\A)$. If on the other hand $Y$ contains $\beta(X)$, then $i_1(t_Y)=t_Y+t_{V^*}-t_X=t_Y-t_X$. Thus, by Lemma~\ref{lemma: commutativity relations in holonomy algebra}
        \[
            [i_1(t_{\beta(X)}), i_1(t_Y)]=[t_{\beta(X)}-t_X, t_Y-t_X]=0,
        \]
        since $\{ \beta(X), X\}$, $\{\beta(X), Y\}$ and $\{Y, X\}$ are $\G$-nested. Similarly for the case $H\neq \beta(X)$. The rest of the claim is immediate.
    \end{proof}
    \noindent
    For $A(X^\perp, \A/X)$ we consider instead the $k$-algebra homomorphism given by
    \begin{gather*}
        i_2 \colon  A(X^\perp, \A/X) \to A(V,\A), \\
        t_{(H+X)/X} \mapsto 
        \begin{cases}
            \sum_{K\in \A, K+X=H+X} t_K & \text{if $H+X\neq \beta(X^+)+X$}; \\
            t_X+\sum_{K\in \A, K+X=H+X} t_K & \text{if $H+X=\beta(X^+)+X$,}
        \end{cases}
    \end{gather*}
    where $X^+$ is the smallest element of $\S$ which strictly contains $X$. Recall that $X^+$ is well-defined, as the set of elements of $\S$ containing $X$ is linearly ordered by inclusion.
    \begin{lemma}
        The map $i_2$ is well-defined and injective. Moreover, together with the embedding $A(V/X^\perp, \A\vert_X)\hookrightarrow A(V,\A)$ it induces an injective $k$-algebra homomorphism
        \[
            \widehat{A}(V/X^\perp, \A\vert_X)\widehat{\otimes}_k\widehat{A}(X^\perp, \A/X) \hookrightarrow \widehat{A}(V,\A).
        \]
    \end{lemma}
    \begin{proof}
        It is apparent that $i_2(t_{V^*/X})=t_{V^*}=0$. To prove that $i_2$ is well-defined, we are left to check that for all $Y\in L(\A)$ with $\dim (Y+X)/X=2$ and $H\in \A$ with $H\subseteq Y+X$ we have 
        \[
        [i_2(t_{(Y+X)/X}),i_2(t_{(H+X)/X})]=0.
        \] 
        If we write $(Y+X)/X=\bigoplus_{i=1}^r (Y_i+X)/X$ for the $\G/X$-decomposition of the subspace $(Y+X)/X$, then $t_{(Y+X)/X}=\sum_{i=1}^r t_{(Y_i+X)/X}$. Since $\dim (Y+X)/X=2$, we have $r\le 2$. \\
        If $r=2$, without loss of generality we may assume $Y_1=H$. Since $(Y+X)/X\not\in \G$, for all $K_1,K_2\in \A$ such that $K_1+X=H+X$ and $K_2+X=Y_2+X$ we have $(K_1+K_2+X)/X=(Y+X)/X$, thus $K_1+K_2\not\in \G$. As a result, $\{K_1,K_2\}$ is a $\G$-nested set and $[t_{K_1},t_{K_2}]=0$ by Lemma~\ref{lemma: commutativity relations in holonomy algebra}. This shows that 
        \[
            [i_2(t_{(Y+X)/X}),i_2(t_{(H+X)/X})]=[i_2(t_{(Y_2+X)/X}), i_2(t_{(Y_1+X)/X})]=0
        \]
        if none between $Y_1+X$ and $Y_2+X$ equals $\beta(X^+)+X$. If this is not the case and, say, $Y_1+X=\beta(X^+)+X$, then either $Y_2\oplus X$ is a decomposition or $Y_2+X\in \G$. In the first case $Y_2$ is the only $K\in \A$ such that $K+X=Y_2+X$; since$\{Y_2,X\}$ is $\G$-nested, we have $[t_X,t_{Y_2}]=0$. Otherwise, $i_2(t_{(Y_2+X)/X})=t_{Y_2+X}-t_X$ and $t_X$ commutes with both itself and $t_{Y_2+X}$, given that $\{Y_2+X,X\}$ is $\G$-nested. \\
        If $r=1$, then we may assume $Y\in \G$. We observe that
        \[
            i_2(t_{(Y+X)/X})= 
            \begin{cases}
                t_Y & \text{if $Y\oplus X$ decomposition, $\beta(X^+)+X \not\subseteq Y+X$}; \\
                t_{Y+X}-t_X & \text{if $Y+ X\in \G$, $\beta(X^+)+X \not\subseteq Y+X$}; \\
                t_Y+t_X & \text{if $Y\oplus X$ decomposition, $\beta(X^+)+X \subseteq Y+X$};\\
                t_{Y+X} & \text{if $Y+ X\in \G$, $\beta(X^+)+X \subseteq Y+X$}.
            \end{cases}
        \]
        In order to check that $i_2(t_{(Y+X)/X})$ commutes with $i_2(t_{(H+X)/X})$ one can perform a case by case analysis. \\
        For example, suppose that $Y\oplus X$ is a decomposition and $Y+X$ contains $\beta(X^+)$. If $H\oplus X$ is a decomposition and $H\neq \beta(X^*)$, then $[t_H,t_X]=0$ because $\{H, X\}$ is $\G$-nested. Hence
        \[
            [i_2(t_{(Y+X)/X}), i_2(t_{(H+X)/X})]= [t_Y+t_X,t_H]=[t_Y,t_H]+[t_X,t_H]=0.
        \]
        For the remaining cases, we remark that, if $Y\oplus X$ is a decomposition, then $H\oplus X$ is one, too.\\
        For injectivity, choose a set of representatives $H_1,\dots, H_n\in \A$ for $\A/X$ and consider the $k$-algebra homomorphism $j \colon A(V,\A)\to A(X^\perp,\A/X)$ which sends $t_{H_i}$ to $t_{(H_i+X)/X}$ and $t_K$ to $0$ if $K\in \A$, $K\neq H_i$ for all $i=1,\dots, n$. To prove that $j$ is well-defined, let $Y\in L(\A)$ have dimension $2$ and take $H\in \A$ contained in $Y$. We want to see that $[j(t_Y),j(t_H)]=0$.\\
        If $H\neq H_i$ for all $i=1,\dots, n$, then $t_H=0$ and this claim trivially holds. Otherwise, without loss of generality we may assume $H=H_1$. Notice that
        \begin{align*}
        j(t_Y)&=\sum_{K\in \A, K\subseteq Y} j(t_K) =\sum_{i=1}^n \sum_{\substack{K\in \A,K\subseteq Y \\ K+X=H_i+X}} j(t_K) \\
              &= \sum_{i=1}^n \sum_{H_i+X\subseteq Y+X} t_{(H_i+X)/X}=t_{(Y+X)/X}.
        \end{align*}
        If $\dim (Y+X)/X=2$, then 
        \[
            [j(t_Y),j(t_{H_1})]=[t_{(Y+X)/X}, t_{(H_1+X)/X}]=0
        \]
        by the relations in $A(X^\perp, \A/X)$. If $\dim Y=1$, then $Y+X=H_1+X$ and the vanishing of $[j(t_Y),j(t_{H_1})]$ follows immediately. Thus, $j$ is well-defined. It is immediate to see that $j$ is a left-inverse of $i_2$, hence $i_2$ is injective.\\
        In order to get a homomorphism $A(V/X^\perp, \A\vert_X)\otimes_k A(X^\perp, \A/X) \to A(V,\A)$, we prove that the images of $A(V/X^\perp,\A\vert_X)$ and $A(X^\perp, \A/X)$ in $A(V,\A)$ commute with one another. It suffices to see that for all $H,K\in \A $ with $H\not \subseteq X$ and $K\subseteq X$ the commutator $[i_1(t_K),i_2(t_{(H+X)/X})]$ vanishes. Since $H\in \G$, we have observed above that $i_2(t_{(H+X)/X})$ is a linear combination of $t_{H+X}$, $t_H$ and $t_X$, while $i(t_K)$ equals either $t_K$ or $t_K-t_X$. Since $K\subseteq X$, the sets $\{K, X\}$, $\{K, H+X\}$ and $\{X,H+X\}$ are $\G$-nested. Moreover, $t_H$ appears only when $H\oplus X$ is a decomposition, in which case $\{H, K\}$  and $\{H, X\}$ are $\G$-nested. As a result, $[i_1(t_K),i_2(t_{(H+X)/X})]=0$, as desired.\\
        Given that all $k$-algebra homomorphisms considered so far respect the grading, the map $A(V/X^\perp, \A\vert_X)\otimes_k A(X^\perp, \A/X) \to A(V,\A)$ extends to the corresponding completions.
    \end{proof}
    \begin{remark}
        Notice that the embedding 
        \[
            \widehat{A}(V/X^\perp, \A\vert_X)\widehat{\otimes}_k\widehat{A}(X^\perp, \A/X) \hookrightarrow \widehat{A}(V,\A)
        \] 
        depends on the choice of a nested set $\S$ with an adapted basis. 
    \end{remark}
    \noindent
    Going back to the solutions of the holonomy equation, in view of the previous lemma we may regard $L_{\S\vert_X}$ and $L_{\S/X}$ as multi-valued functions with values in $\widehat{A}(V,\A)\widehat{\otimes}_k \C$. Also note that the asymptotic behavior of $L_\S$ around $D_\S$ ensures that the function $L_\S\left( u_X^\S\right)^{-t_X}$ extends to a meromorphic function on $U_\S\cap D_X$ in a neighborhood of $D_S$.\\
    We anticipate that, when computing explicitly with $\Omega$ in coordinates on the local chart $U_\S$ associated with a maximal nested set $\S$ and adapted basis $\beta$, we always fix as hyperplane at infinity the one given by $\langle \beta(V^*)\rangle$.
    \begin{proposition}
        \label{prop: restriction of holonomy equation to the boundary}
        The restriction of $L_\S\left( u_X^\S\right)^{-t_X}$ to $D_X$ equals the product $L_{\S\vert_X}L_{\S/X}$.
    \end{proposition}
    \begin{proof}
        In local coordinates $u_Y^\S$ on $U_\S$ we have
        \begin{align*}
            \Omega &=\sum_{H\in \A } t_H \otimes d\log x_H= \sum_{H\in \A} t_H \otimes d\log x_{p_\S(H)} + \sum_{H\in \A} t_H\otimes d\log P_{H}^\S\\
            &=\sum_{H\in \A} \sum_{Y\in \S, p_\S(H)\subseteq Y} t_H\otimes d\log u_Y^\S + \sum_{H\in \A} t_H\otimes d\log P_{H}^\S\\
            &= \sum_{Y\in \S} \left(\sum_{H\in\A, p_\S(H)\subseteq Y}t_H\right)\otimes d\log u_Y^\S + \sum_{H\in \A} t_H\otimes d\log P_{H}^\S\\
            &= \sum_{Y\in\S} t_Y\otimes \frac{du_Y^\S}{u_Y^\S} + \sum_{H\in \A} t_H\otimes d\log P_{H}^\S.
        \end{align*}
        Notice that the differential form $\sum_{H\in \A} t_H\otimes d\log P_{H}^\S$ does not have poles along the divisors $\{u_Y^\S=0\}$ for all $Y\in \S$.
        Analogous expressions hold also for the formal differential forms $\Omega_1$ and $\Omega_2$ associated with the holonomy equations on $Y(V/X^\perp, \A\vert_X)$ and $Y(X^\perp, \A/X)$ respectively.\\
        Under the embedding $i_1 \colon A(V/X^\perp, \A\vert_X)\hookrightarrow A(V,\A)$ it is clear that $\Omega_1$ becomes
        \[
            \Omega_1 = \sum_{Y\in\S\vert_X, Y\neq X} t_Y\otimes \frac{du_Y^\S}{u_Y^\S} + \sum_{H\in \A\vert_X} t_H\otimes d\log P_{H}^\S,
        \]
        where we are using the fact that the isomorphism $U_\S\cap D_X\cong U_{\S\vert_X}\times U_{\S/X}$ sends the coordinate $u_Y^{\S\vert_X}$ to $u_Y^\S$ for all $Y\in \S$, $Y\subsetneq X$. Notice that $P^\S_{\beta(X)}=1$ since $X\in \S$, hence $d\log P^\S_{\beta(X)}=0$.\\
        For $\Omega_2$ we first observe that for all $Y\in \S$, $Y\not\subseteq X$ we have $i_2(t_{(Y+X)/X})=t_Y$. Indeed, if $Y\oplus X$ is a decomposition, then $\beta(X^+)\not\subseteq Y$, or else the properties of adapted bases would imply $X\subseteq X^+=p_\S(\beta(X^+))\subseteq Y$; as we have seen in the proof of the previous lemma, this yields $i_2(t_{(Y+X)/X})=t_Y$. If on the other hand $Y+X\in \G$, then $X\subseteq Y$ for $\S$ is $\G$-nested. As $Y\not\subseteq X$, by minimality of $X^+$ we also have $X^+\subseteq Y$, hence $i_2(t_{(Y+X)/X})=t_Y$ follows from the definition of $i_2$.\\ 
        Then, for all $H\in \A$, $H\not\subseteq X$, denote by $\tilde{P}_H^\S$ the restriction of $P_H^\S$ to $\{u_X^\S=0\}$. Notice that $\tilde{P}_H^\S$ is the image of $P_H^\S$ under the isomorphism $U_\S\cap D_X\cong U_{\S\vert_X}\times U_{\S/X}$. The image of $\Omega_2$ via $i_2$ then becomes
        \begin{align*}
            \Omega_2 &= \sum_{Y\in\S, Y\not\subseteq X} i_2(t_{(Y+X)/X})\otimes d\log u_Y^{\S}+ \\
            & \qquad \sum_{(H+X)/X\in\A/X} i_2(t_{(H+X)/X})\otimes d\log P_{(H+X)/X}^{\S/X}\\
            &= \sum_{Y\in\S, Y\not\subseteq X} t_Y\otimes d\log u_Y^{\S} + \sum_{\substack{H\in\A, H\not\subseteq X\\ H+X\neq\beta(X^+)+X}} t_H\otimes d\log \tilde{P}_H^\S \\
             & \qquad + i_2(t_{(\beta(X^+)+X)/X})\otimes d\log \tilde{P}_{\beta(X^+)}^\S \\
            &=\sum_{Y\in\S, Y\not\subseteq X} t_Y\otimes \frac{du_Y^\S}{u_Y^\S} + \sum_{H\in\A, H\not\subseteq X} t_H\otimes d\log \tilde{P}_H^\S,
        \end{align*}
        using that for all $H\in \A$ with $H+X=\beta(X^+)+X$ we have $d\log\tilde{P}_H^\S=d\log P_{\beta(X^+)}^\S=0$. Indeed, $P^\S_{\beta(Y)}=1$ for all $Y\in \S$, so $\tilde{P}_H^\S=P_{\beta(X^+)}^\S=1$. \\
        Overall, this shows that 
        \begin{align*}
            \Omega_1+\Omega_2 &= \left(\Omega-t_X\otimes d\log u_X^\S\right)\Big\vert_{\left\{u_X^\S=0\right\}}.
        \end{align*}
        We now observe that
        \begin{align*}
        d\left( L_\S(u_X^\S)^{-t_X}\right) &= \Omega L_\S(u_X^\S)^{-t_X}-L_\S (u_X^\S)^{-t_X}t_Xd\log u_X^\S \\
        &= (\Omega-t_Xd\log u_X^\S) L_S(u_X^\S)^{-t_X} - \frac{1}{u_X^\S}\left[L_\S(u_X^\S)^{-t_X}, t_X\right]du_X^\S.
        \end{align*}
        The differential form $\Omega-t_Xd\log u_X^\S$ is holomorphic at $\{u_X^\S=0\}$ and so are  $L_\S(u_X^\S)^{-t_X}$ and its differential. This implies that $(u_X^\S)^{-1}\left[L_\S(u_X^\S)^{-t_X},t_X\right]$ is holomorphic as well at $\{u_X^\S=0\}$ in a neighborhood of the origin. As a result, 
        \begin{align*}
            d\left( L_\S(u_X^\S)^{-t_X}\right)\Big\vert_{\{u_X^\S=0\}}&= (\Omega-t_Xd\log u_X^\S)\big\vert_{\{u_X^\S=0\}} \left(L_S(u_X^\S)^{-t_X}\right)\Big\vert_{\{u_X^\S=0\}} \\
            &= (\Omega_1+\Omega_2)\left(L_S(u_X^\S)^{-t_X}\right)\Big\vert_{\{u_X^\S=0\}}.
        \end{align*} 
        Moreover, we have
        \begin{align*}
            d\left(L_{\S\vert_X} L_{\S/X} \right)&= \Omega_1 L_{\S\vert_X} L_{\S/X} + L_{\S\vert_X} \Omega_2L_{\S/X}\\
            &=(\Omega_1+\Omega_2) L_{\S\vert_X} L_{\S/X},
        \end{align*}
        since $[\Omega_2, L_{\S\vert_X}]=0$. This shows that $L_\S(u_X^\S)^{-t_X}$ and $L_{\S\vert_X} L_{\S/X}$ satisfy the same differential equation. Let us prove that they also share the same asymptotic behavior around the origin of $U_\S\cap D_X$. \\
        For $L_\S$ we know that 
        \[
            L_\S\left(u_X^\S\right)^{-t_X} \sim \prod_{Y\in \S, Y\neq X, V^*}\left(u_Y^\S\right)^{t_Y}.
        \]
        On the other hand,
        \begin{gather*}
            L_{\S\vert_X} L_{\S/X} \prod_{\substack{Y\in \S\vert_X\\ Y\neq X}}\left(u_Y^{\S\vert_X}\right)^{-t_Y}\prod_{\substack{Y\in \S/X \\ Y\neq V^*/X}}\left(u_Y^{\S/X}\right)^{-t_Y} \\ = \left(L_{\S\vert_X} \prod_{\substack{Y\in \S\vert_X\\ Y\neq X}}\left(u_Y^{\S\vert_X}\right)^{-t_Y}\right) \left( L_{\S/X} \prod_{\substack{Y\in \S/X \\ Y\neq V^*/X}}\left(u_Y^{\S/X}\right)^{-t_Y}\right) \sim 1,
        \end{gather*}
        since $[t_Y, L_{\S/X}]=0$ for all $Y\in \S\vert_X$. Hence, $L_\S(u_X^\S)^{-t_X}$ and $L_{\S\vert_X} L_{\S/X}$ are solutions to the same differential problem
        \[
            \begin{cases}
                dF=(\Omega_1+\Omega_2)F, \\
                F\sim \prod_{\substack{Y\in \S\vert_X\\ Y\neq X}}\left(u_Y^{\S\vert_X}\right)^{t_Y}\prod_{\substack{Y\in \S/X \\ Y\neq V^*/X}}\left(u_Y^{\S/X}\right)^{t_Y}.
            \end{cases}
        \]
        This coincides with the holonomy equation of 
        \[
            Y(V/X^\perp\oplus X^\perp, \A\vert_X\sqcup \A/X) \cong Y(V/X^\perp, \A\vert_X)\times Y(X^\perp, \A/X)
        \]
        localized at the $\G'$-nested set $(\S\vert_X\setminus \{X\})\sqcup (\S/X\setminus \{V^*/X\})\sqcup \{V/X^\perp\oplus X^\perp\}$, where 
        \[
            \G'= (\G\vert_X\setminus \{X\})\sqcup (\G/X\setminus \{V^*/X\})\sqcup \{V/X^\perp\oplus X^\perp\}.
        \]
        The equality $L_\S(u_X^\S)^{-t_X}=L_{\S\vert_X} L_{\S/X}$ follows at last from the uniqueness of the solution of the holonomy equation with given boundary conditions.
    \end{proof}
    \begin{remark}
        The study carried out in \cite{De_Concini-Procesi-Hyperplane_arrangements_and_holonomy_equations} deals with the restriction of the solutions of the holonomy equation at the hyperplanes $H^\perp$ with $H\in \A$ and their intersections, but not at the divisors $D_X$ for $X\in \G$ in $\overline{Y}(V,\A,\G)$. To handle this case, it seems necessary to introduce the embeddings $A(V/X^\perp, \A\vert_X)\to A(V,\A)$ and $A(X^\perp, \A/X)\to A(V,\A)$ studied above.
    \end{remark}

    \subsection{Properties of the solutions of the holonomy equation}
        
    As we have seen, the solution $L_\S$ of the holonomy equation localized at $\S$ can be regularized and restricted to the irreducible boundary divisors $D_X$ for all $X\in\S$, where it equals $L_{\S\vert_X}L_{\S/X}$. We specialize now to hyperplane arrangements with enough modular elements. In this case, we show that there is a similar decomposition of $L_\S$ depending on $L_{\S\vert_X}$ and $L_{\S/X}$ defined, however, on the whole $Y(V,\A)$, not only on the open subset $Y(V/X^\perp, \A\vert_X)\times Y(X^\perp, \A/X)\subseteq D_X$. The existence of retractions $\overline{Y}(V,\A,\G)\to D_X$ induced by linear projections at the level of arrangements enables us to argue as in \cite{Brown-Multiple_zeta_values_and_periods_of_moduli_spaces_of_curves} in the case of the braid arrangement.

    \subsubsection{Global decomposition with respect to a boundary divisor}
    Let $(V,\A)$ be a hyperplane arrangement and $\G$ a building family for $L(\A)$. Let $X\in\G$ and assume that there is a $\G$-modular element $M\in \G$ such that $V^*=M\oplus X$. By Proposition~\ref{prop: restrictions into quotients of G-modular elts} this is the case if $(V,\A)$ has enough $\G$-modular elements. For a maximal $\G$-nested set $\S$ with $M$-adapted basis $\beta$ and $X\in\S$, let $L_\S$, $L_{\S\vert_X}$ and $L_{\S/X}$ be the solutions of the holonomy equation on $Y(V,\A)$, $Y(V/X^\perp, \A\vert_X)$ and $Y(X^\perp, \A/X)$ respectively, localized at the corresponding nested sets. \\
    Let $\pi \colon  \overline{Y}(V,\A,\G)\to D_X$ be the retraction described in Proposition~\ref{prop: retractions to boundary divisors} with respect to $M$. Also denote by
    \[
        \pi_1 \colon  \overline{Y}(V,\A,\G)\to \overline{Y}(V/X^\perp, \A\vert_X, \G\vert_X), \quad \pi_2 \colon  \overline{Y}(V,\A,\G)\to \overline{Y}(X^\perp, \A/X, \G/X)
    \]
    the projections of $\pi$ onto the two factors of $D_X$. By abuse of notation, we keep the same notation also for the restrictions of $\pi_1$ and $\pi_2$ to $Y(V,\A)$, which therefore take value in $Y(V/X^\perp, \A\vert_X)$ and $Y(X^\perp, \A/X)$ respectively. Let us write
    \[
        L_1=L_{\S\vert_X}\circ \pi_1, \quad L_2=L_{\S/X}\circ \pi_2,
    \]
    which are multi-valued functions on $Y(V,\A)$. Via the embeddings of holonomy algebras in the previous sections, we regard $L_1$ and $L_2$ as functions taking values in $\widehat{A}(V,\A)$.\\
    Let $\Omega_1$ and $\Omega_2$ be the differential forms associated with the holonomy equation on $Y(V/X^\perp, \A\vert_X)$ and $Y(X^\perp, \A/X)$ respectively. Since the adapted basis for $\S$ is $M$-adapted, the explicit description of $\pi_1$ and $\pi_2$ in local coordinates implies that 
    \begin{align*}
        \pi_1^*\Omega_1 &= \sum_{Y\in\S\vert_X, Y\neq X} t_Y\otimes d\log u_Y^\S + \sum_{H\in \A\vert_X} t_H\otimes d\log P_{H}^\S, \\
        \pi_2^*\Omega_2 & =\sum_{Y\in\S, Y\not\subseteq X} t_Y\otimes d\log u_Y^\S + \sum_{H\in\A, H\not\subseteq X} t_H\otimes d\log \tilde{P}_H^\S,
    \end{align*}
    similarly to what we have seen in the proof of Proposition~\ref{prop: restriction of holonomy equation to the boundary}. Thus, for $i=1,2$ we have $dL_i=\pi_i^*\Omega_i L_i$, so $d(L_i)^{-1}=-L_i^{-1}(dL_i)L_i^{-1}=-L_i^{-1}\pi_i^*\Omega_i$.
    \begin{proposition}
    \label{prop: global decomp. of holonomy solutions}
        There is a decomposition
        \[
            L_\S=hL_1L_2,
        \]
        where $h$ is a multi-valued function with values in $\widehat{A}(V,\A)$ uniquely determined by the differential equation
        \[
            \frac{\partial h}{\partial u_X^\S} = \left( t_X\otimes \frac{1}{u_X^\S} + \sum_{H\in\A} t_H \otimes \frac{1}{P_H^\S}\frac{\partial P_H^\S}{\partial u_X^\S}\right) h
        \]
        with boundary condition that $h \left( u_X^\S\right)^{-t_X}$ restricts on $\{u_X^\S=0\}$ to the constant function $1$.
    \end{proposition}
    \begin{proof}
        Let $h=L_\S L_2^{-1}L_1^{-1}$. Since $[\pi_2^*\Omega_2, L_1]=0$, we have
        \begin{align*}
            dh &= \Omega L_\S L_2^{-1} L_1^{-1} - L_\S L_2^{-1} \pi_2^*\Omega_2 L_1^{-1}-L_\S L_2^{-1} L_1^{-1} \pi_1^*\Omega_1\\
            &=\Omega h - h(\pi_1^*\Omega_1+\pi_2^*\Omega_2).
        \end{align*}
        Since $L_1$ and $L_2$ do not depend on the variable $u_X^\S$, we have
        \[
            \frac{\partial h}{\partial u_X^\S} = \frac{\partial L_\S}{\partial u_X^\S}L_2^{-1} L_1^{-1}= \left( t_X\otimes \frac{1}{u_X^\S} + \sum_{H\in\A} t_H \otimes \frac{1}{P_H^\S}\frac{\partial P_H^\S}{\partial u_X^\S}\right) h.
        \]
        Fix a complex point $p$ of $Y(V/X^\perp, \A\vert_X)\times Y(X^\perp, \A/X)\subseteq D_X$. For $H,K\in \A$ write $H\sim K$ when the restriction of $P_H^\S$ and $P_K^\S$ at the fiber $\pi^{-1}(p)$ coincide. Notice that these restrictions are always polynomials of degree at most $1$ in $u_X^\S$, say $P_H^\S\vert_{\pi^{-1}(p)}= a_H u_X^\S+b_H$. We then have
        \[
            \frac{\partial }{\partial u_X^\S}h\vert_{\pi^{-1}(p)}= \left( t_X\otimes \frac{1}{u_X^\S} + \sum_{[K]\in\faktor{\A}{\sim}}\sum_{H\sim K} t_H \otimes \frac{1}{P_H^\S\vert_{\pi^{-1}(p)}}\frac{\partial P_H^\S\vert_{\pi^{-1}(p)}}{\partial u_X^\S}\right) h\vert_{\pi^{-1}(p)}.
        \]
        For all $[K]\in \A/\sim$, write $t_{[K]}=\sum_{H\sim K} t_H$. It is then apparent that $h\vert_{\pi^{-1}(p)}$ satisfies a one-dimensional holonomy equation in the formal variables $t_{[K]}$ for $K\in \A$ with respect to the hyperplane arrangement consisting of the points $0$ and $-\frac{b_H}{a_H}$ for $H\in \A$ with $a_H\neq 0$. \\
        In particular, at each fiber $\pi^{-1}(p)$ the function $h$ is uniquely determined by its asymptotic behavior at $u_X^\S=0$, which we now compute. Let us write
        \[
            f_\S=L_\S \prod_{Y\in \S} \left(u_Y^\S\right)^{-t_Y}, \quad f_1=L_1 \prod_{Y\in \S, Y\subseteq X} \left(u_Y^\S\right)^{-t_Y},  \quad f_2=L_2 \prod_{Y\in \S, Y\not\subseteq X} \left(u_Y^\S\right)^{-t_Y},
        \]
        which extend holomorphically to a neighborhood of $D_\S$ with value $1$ at $D_\S$. Since $[f_2,t_Y]=0$ for all $Y\in \S$ with $Y\subsetneq X$ and $[f_1,t_X]=[f_2,t_X]=0$, we infer that 
        \[
            h=f_\S f_2^{-1}f_1^{-1} \left( u_X^\S\right)^{t_X}. 
        \]
        Thus, $h \left( u_X^\S\right)^{-t_X}$ extends to $\{u_X^\S=0\}$. Consider the function $h'=L_2^{-1}L_1^{-1}L_\S$, which satisfies
        \[
            dh'= L_2^{-1}L_1^{-1}(\Omega-\pi_1^*\Omega_1-\pi_2^*\Omega_2)L_\S.
        \]
        It follows that 
        \[
            d\left( h'\left( u_X^\S\right)^{-t_X}\right)= L_2^{-1}L_1^{-1}(\Omega-t_X d\log u_X^\S-\pi_1^*\Omega_1-\pi_2^*\Omega_2)L_\S +L_2^{-1}L_1^{-1} [t_X,L_\S]d\log u_X^\S.
        \]
        As in the proof of Proposition~\ref{prop: restriction of holonomy equation to the boundary}, the restriction of the differential of $ h'\left( u_X^\S\right)^{-t_X}$ to $\{u_X^\S=0\}$ vanishes. Given that $f_\S$, $f_1$ and $f_2$ all take the value $1$ at $D_\S$, we conclude that $h'\left( u_X^\S\right)^{-t_X}$ restricts to the constant function $1$ on $\{u_X^\S=0\}$. The same holds for $h\left( u_X^\S\right)^{-t_X}=L_1L_2h'L_2^{-1}L_1^{-1}\left( u_X^\S\right)^{-t_X}$. 
    \end{proof}
    
    \subsubsection{Associators}
    Let us now inspect the relation between the solutions of the holonomy equation with different boundary conditions. Let $\S_1$ and $\S_2$ be maximal $\G$-nested sets and $L_{\S_1}$ and $L_{\S_2}$ be the solutions of the holonomy equation for $Y(V,\A)$ localized at $\S_1$ and $\S_2$ respectively. We are implicitly fixing adapted bases for $\S_1$ and $\S_2$.
    \begin{lemma}
        There exists a constant power series $G(\S_1,\S_2)\in \widehat{A}(V,\A)\widehat{\otimes}_k\C$ such that $L_{\S_2}=L_{\S_1}G(\S_1,\S_2)$. 
    \end{lemma}
    \begin{proof}
        Let $G(\S_1,\S_2)=L_{\S_1}^{-1}L_{\S_2}$. Since $d L_{\S_1}^{-1}=-L_{\S_1}^{-1}\Omega$, we have
        \[
            dG(\S_1,\S_2)=d\left(L_{\S_1}^{-1}L_{\S_2} \right)= -L_{\S_1}^{-1}\Omega L_{\S_2}+ L_{\S_1}^{-1}\Omega L_{\S_2}=0.
        \]
    \end{proof}
    \noindent
    Our next goal is to describe the numbers appearing as coefficients of the constant series $G(\S_1,\S_2)$ when we let the maximal $\G$-nested sets $\S_1$ and $\S_2$ vary. However, notice that $G(\S_1,\S_2)$ depends on the choice of adapted bases for $\S_1$ and $\S_2$. In order to have more control on the coefficients of these series, we will impose some restrictions on the maximal $\G$-nested sets and adapted bases under consideration.\\
    Let $\Sigma$ be a family of maximal $\G$-nested sets equipped with adapted bases. An element of $\Sigma$ will be denoted by a pair $(\S,\beta)$, where $\S$ is a maximal $\G$-nested set and $\beta$ is an adapted basis for $\S$. Assume that the maximal $\G$-nested sets appearing in $\Sigma$ make up a connected subgraph of the graph of maximal $\G$-nested sets. \\
    Let us briefly mention the geometric meaning of $\Sigma$. Suppose that we want to compute a period integral over some relative cohomology class $\Delta$ with boundary contained in a set of irreducible components of $\overline Y(V,\A,\G)\setminus Y(V,\A)$. The canonical stratification of $\overline Y(V,\A,\G)\setminus Y(V,\A)$ by intersection of boundary divisors induces a structure of manifold with corners on $\Delta$. The stratum of dimension zero corresponds to a family $\Sigma$ of maximal $\G$-nested sets. The datum of an adapted basis for each $\S\in \Sigma$ allows us to recover local charts for $\Delta$. \\
    Given $(\S_1,\beta_1),(\S_2,\beta_2)\in \Sigma$, there are corresponding local charts $U_{\S_1,\beta_1}$ and $U_{\S_2,\beta_2}$ of $\overline{Y}(V,\A,\G)$, whose coordinates we denote by $u_X^{\S_1}$ and $u_Y^{\S_2}$ respectively, for all $X\in \S_1$ and $Y\in \S_2$. Notice that these coordinates depend on the choice of the adapted bases $\beta_1$ and $\beta_2$, although these have been suppressed in the notation. Recall that, given $Y\in \S_2$, we have defined a polynomial $P^{\S_1}_{\beta_2(Y)}$ in the coordinates $u_{X}^{\S_1}$ which does not vanish at the origin of $U_{\S_1,\beta_1}$. Denote by $\alpha^{\S_1,\S_2}_Y\in k$ the value of this polynomial at the origin. Once again, this number depends on the choice of $\beta_1$ and $\beta_2$. Since $\alpha_{Y}^{\S_1,\S_2}\neq 0$, we can take $\log \alpha_Y^{\S_1,\S_2}$ for a given determination of the logarithm.
    \begin{definition} $ $
    
        \begin{enumerate}
            \item We denote by $F_\Sigma(V,\A,\G)$ the $k$-subalgebra of $\C$ generated by $2\pi i$ and the numbers $\log \alpha_Y^{\S_1,\S_2}$ for all $(\S_1,\beta_1),(\S_2,\beta_2)\in \Sigma$ and $Y\in \S_2$.
            \item The \emph{associator algebra} of $(V,\A,\G)$ with respect to $\Sigma$ is the $k$-subalgebra $Z_\Sigma(V,\A,\G)$ of $\C$ generated by the coefficients of $G(\S_1,\S_2)$ for all maximal $\G$-nested sets $(\S_1,\beta_1), (\S_2,\beta_2)\in \Sigma$. 
        \end{enumerate}
    \end{definition}
    \begin{remark}
        \label{remark: logarithms reduce to 2 pi i}
        Notice that the presence of $2\pi i$ in $F_\Sigma(V,\A,\G)$ makes it unnecessary to specify a determination of the logarithm. Moreover, logarithms of roots of unity are rational multiples of $2\pi i$. Thus, if all $\alpha_Y^{\S_1,\S_2}$ are roots of unity, then $F_\Sigma(V,\A,\G)=k[2\pi i]$.
    \end{remark}
    \begin{remark}
        The name \emph{associator algebra} is not standard, but it will come in handy during the rest of the exposition. We opted for this name given the analogy with Drinfeld's associator \cite{Drinfeld-On_quasitriangular_quasiHopf_algebras_and_on_a_group_that_is_closely_connected_with_Gal_Qbar_Q}.
    \end{remark}
    \noindent
    We may define an increasing filtration $W_\bullet Z_\Sigma(V,\A,\G)$ in $Z_\Sigma(V,\A,\G)$ as follows. First, for $\S_1,\S_2\in \Sigma$ we let $W_n(\S_1,\S_2)$ be the $k$-vector space generated by the coefficients of the terms of degree at most $n$ of $G(\S_1,\S_2)$. We will see that $\sum_n W_n(\S_1,\S_2)$ is a $k$-algebra filtered by the $W_n(\S_1,\S_2)$'s in the next section. Then, $Z_\Sigma(V,\A,\G)$ is generated as a $k$-algebra by all the $k$-algebras $\sum_n W_n(\S_1,\S_2)$ for $\S_1,\S_2\in\Sigma$. Hence, it comes with an induced filtration, which will be referred to as \emph{weight filtration}.\\
    Our goal is to describe the associator algebra of $(V,\A,\G)$ in terms of the analogous algebras of the $1$-dimensional arrangements obtained from $(V,\A,\G)$ by iterated restrictions and deletions. To make this precise, let $\S$ be a maximal $\G$-nested set and fix $X\in \S$, $X\neq V^*$. Writing $\S'=\S\setminus \{X\}$, the closed subscheme $D_{\S'}$ of $\overline{Y}(V,\A,\G)$ has dimension one. Set $A$ to be the sum of all $Y\in \S'$ which are contained in $X^+$ and let $(V_{\S'},\A_{\S'},\G_{\S'})$ denote the arrangement
    \[
        \left(A^\perp/X^\perp, (\A\vert_{X^+})/A, (\G\vert_{X^+})/A \right).
    \]
    We then have an isomorphism $D_{\S'}\cong \overline{Y}(V_{\S'},\A_{\S'},\G_{\S'})$. \\
    For convenience of notation, we say that a $\G$-nested set $\S'$ is $\Sigma$-admissible if there are maximal $\G$-nested sets $\S,\mathcal{T}\in \Sigma$ containing $\S'$ such that for some $X\in \S$ we have $\S\setminus \{X\}=\S\cap \mathcal{T}$ or, equivalently, for some $Y\in \mathcal{T}$ we have $\mathcal{T}\setminus \{Y\} = \S\cap \mathcal{T}$. This means that $\S$ and $\mathcal{T}$ are connected by an edge in the graph of maximal $\G$-nested sets. In other words, the two distinct points $D_{\S}$ and $D_{\mathcal{T}}$ are both contained in the closed subscheme $D_{\S'}$.\\
    For a $\Sigma$-admissible $\G$-nested set $\S'$, one can define an induced family $\Sigma'$ of maximal $\G_{\S'}$-nested subsets with adapted bases. Indeed, given $(\mathcal{T},\beta)\in\Sigma$ such that $\S'\subseteq \mathcal{T}$, the unique element of $\mathcal{T}\setminus \S'$ yields a maximal $\G_{\S'}$-nested set, provided we also add $V_{\S'}^*$, of course. This nested set also carries an adapted basis induced by $\beta$ via restrictions and quotients.
    \begin{proposition}
        \label{prop: decomposition of associators}
        If $(V,\A)$ has enough $\G$-modular elements, then
        \[
            Z_{\Sigma}(V,\A,\G)=F_\Sigma(V,\A,\G)\otimes_k \bigotimes_{\S'} Z_{\Sigma'}(V_{\S'},\A_{\S'},\G_{\S'}),
        \]
        the last tensor product running over all $\Sigma$-admissible $\G$-nested sets $\S'$.
    \end{proposition}
    \begin{proof}
        Suppose first that at least two distinct maximal $\G$-nested sets appear in $\Sigma$. Take $(\S_1,\beta_1),(\S_2,\beta_2)\in \Sigma$. Since we assume $\Sigma$ to be a connected subgraph of the graph of maximal $\G$-nested subsets, there are $(\mathcal{T}_1,\gamma_1),\dots, (\mathcal{T}_n,\gamma_n)\in \Sigma $ such that $(\S_1,\beta_1)=(\mathcal{T}_1,\gamma_1)$, $(\S_2,\beta_2)=(\mathcal{T}_n,\gamma_n)$ and each pair $\mathcal{T}_i,\mathcal{T}_{i+1}$ is connected by an edge. It is apparent that 
        \[
            G(\S_1,\S_2)=G(\mathcal{T}_1,\mathcal{T}_2)\dots G(\mathcal{T}_{n-1},\mathcal{T}_n).
        \]
        Notice that, if $\S_1=\S_2$, by our assumption on $\Sigma$ there is $\S_3$ connected to $\S_1$, hence $G(\S_1,\S_2)=G(\S_1,\S_3)G(\S_3,\S_2)$. In all cases, we need only to compute $G(\S_1,\S_2)$ when $\S_1$ and $\S_2$ are connected by an edge, which means that there are $X_1\in \S_1$ and $X_2\in \S_2$ such that $X_1\neq X_2$ and $\S_1\setminus\{X_1\}=\S_2\setminus\{X_2\}=\S_1\cap \S_2$. By definition, $\S_1\cap \S_2$ is $\Sigma$-admissible.\\
        Let $L_{\S_1}$ and $L_{\S_2}$ be the solutions of the holonomy equation of $Y(V,A)$ localized at $\S_1$ and $\S_2$ respectively, with respect to the adapted bases $\beta_1$ and $\beta_2$. Also consider their regularized restrictions
        \begin{gather*}
            L_{\S_1}'=\left( L_{\S_1} \prod_{Y\in \S_1\cap \S_2} \left(u_Y^{\S_1}\right)^{-t_Y} \right)\Bigg\vert_{U_{\S_1}\cap D_{\S_1\cap\S_2}}, \\
            L_{\S_2}'=\left( L_{\S_2} \prod_{Y\in \S_1\cap \S_2} \left(u_Y^{\S_2}\right)^{-t_Y} \right)\Bigg\vert_{U_{\S_2}\cap D_{\S_1\cap\S_2}}.
        \end{gather*}
        Proposition~\ref{prop: restriction of holonomy equation to the boundary} implies that $L_{\S_1}'$ and $L_{\S_2}'$ are the solutions of the holonomy equation on $Y(V_{\S_1\cap \S_2},\A_{\S_1\cap\S_2})$ localized at the image of $X_1$ and $X_2$ in $\G_{\S_1\cap\S_2}$ respectively. As a result, there is a constant series $G(X_1,X_2)$ such that $L_{\S_2}'=L_{\S_1}'G(X_1,X_2)$. Notice that the coefficients of $G'(X_1,X_2)$ belong to the associator algebra $Z_{\Sigma'}(V_{\S_1\cap\S_2}, \A_{\S_1\cap\S_2},\G_{\S_1\cap\S_2})$, where $\Sigma'$ is the family of $\G$-nested sets with adapted bases induced on $D_{\S_1\cap \S_2}$ by $\Sigma$.\\
        We shall prove that $G(\S_1,\S_2)$ equals $G(X_1,X_2)$ up to multiplication by a constant power series with coefficients in $F_\Sigma(V,\A,\G)$. We argue inductively by repeatedly projecting onto the divisors $D_Y$ for $Y\in\S_1\cap \S_2$.\\
        Fix $Y\in \S_1\cap \S_2$. For $i=1,2$ the functions
        \[
            \tilde{L}_i=\left( L_i \left( u_Y^{\S_i}\right)^{-t_Y}\right)\Big\vert_{\left\{u_Y^{\S_i}=0\right\}}
        \]
        are solutions to the holonomy equation on $D_Y$ localized at the nested sets $\tilde{\S_1}$ and $\tilde{\S_2}$ induced by $\S_1$ and $\S_2$. Thus, there is a constant power series $G(\tilde{\S}_1,\tilde{\S}_2)$ in the holonomy algebra associated with $D_Y$ such that $\tilde{L}_2=\tilde{L}_1 G(\tilde{\S}_1,\tilde{\S}_2)$. Since $(V,\A)$ has enough $\G$-modular elements, by Proposition~\ref{prop: retractions to boundary divisors} there is a retraction $\pi_Y \colon  \overline{Y}(V,\A,\G)\to D_Y$ induced by a linear projection $V\to Y^\perp$. For $i=1,2$ let us write $L_{\S_i}=h_i\pi_Y^*\tilde{L}_i$ as in Proposition~\ref{prop: global decomp. of holonomy solutions}, where $\pi_Y^*\tilde{L}_i=\tilde{L}_i\circ \pi$. We then have
        \begin{align*}
            G(\S_1,\S_2)^{-1}\pi_Y^*G(\tilde{\S}_1,\tilde{\S}_2) &= (L_2)^{-1} L_1(\pi_Y^*\tilde{L}_1)^{-1}\pi_Y^*\tilde{L}_2\\
            &= (\pi_Y^*\tilde{L}_2)^{-1} h_2^{-1} h_1 \pi_Y^*\tilde{L}_1 (\pi_Y^*\tilde{L}_1)^{-1}\pi_Y^*\tilde{L}_2\\
            &= (\pi_Y^*\tilde{L}_2)^{-1} h_2^{-1} h_1\pi_Y^*\tilde{L}_2.
        \end{align*}
        Since this is a constant series, it agrees with its regularization along $D_{\S_1}$. For regularizations, we work with the manifold with corners $\{ 0\le u_Z^{\S_1}<\epsilon\mid Z\in \S_1\}$ for $\epsilon>0$ small enough, allowing logarithms along $\{u_Z^{\S_1}=0\}$ for all $Z\in \S_1$. By Proposition~\ref{prop: global decomp. of holonomy solutions}, it is clear that $\text{Reg}\,(h_1, D_{\S_1})=1$. On the other hand, if $Y^+$ is the successor of $Y$ in $\S_2$, we have
        \begin{align*}
            u_Y^{\S_2} &= \frac{x_{\beta_2(Y)}}{x_{\beta_2(Y^+)}}=\frac{x_{\beta_1(p_{\S_1}(\beta_2(Y)))}}{x_{\beta_1(p_{\S_1}(\beta_2(Y^+)))}}\frac{P^{\S_1}_{\beta_2(Y)}}{P^{\S_1}_{\beta_2(Y^+)}} \\
            &= \prod_{\substack{Z\in\S_1 \\ p_{\S_1}(\beta_2(Y)) \subseteq Z}} u_Z^{\S_1}
               \prod_{\substack{Z\in\S_1 \\ p_{\S_1}(\beta_2(Y^+)) \subseteq Z}} \left(u_Z^{\S_1}\right)^{-1}
               \frac{P^{\S_1}_{\beta_2(Y)}}{P^{\S_1}_{\beta_2(Y^+)}}.
        \end{align*}
        As a result,
        \begin{align*}
            \text{Reg}\left( h_2^{-1}, D_{\S_1}\right) &= \text{Reg}\left( \left(u_Y^{\S_2}\right)^{-t_Y} g_2^{-1}, D_{\S_1}\right) = \text{Reg}\left( \left(u_Y^{\S_2}\right)^{-t_Y}, D_{\S_1}\right) \\
            &=\exp\left( t_Y\left( \log \alpha_{\beta_2(Y)}^{\S_1,\S_2}-\log\alpha_{\beta_2(Y^+)}^{\S_1,\S_2}\right)\right).
        \end{align*}
        Since $t_Y$ commutes with $\pi_Y^*\tilde{L}_2$, we conclude that
        \begin{align*}
            \pi_Y^*G(\tilde{\S}_1,\tilde{\S}_2)&= \text{Reg}\left( h_2^{-1}, D_{\S_1}\right)G(\S_1,\S_2)\\
            &=\exp\left( t_Y\left( \log \alpha_{\beta_2(Y)}^{\S_1,\S_2}-\log\alpha_{\beta_2(Y^+)}^{\S_1,\S_2}\right)\right)G(\S_1,\S_2).
        \end{align*}
        We can iterate this argument for all $Y'\in \S_1\cap \S_2$, $Y'\neq Y $, choosing at each step a projection $\pi_{Y'} \colon D_Y\to D_Y\cap D_{Y'}$ and exchanging the role of $G(\S_1,\S_2)$ with $G(\tilde{\S}_1,\tilde{\S}_2)$ and of $Y$ with $Y'$. The composition of the maps $\pi_Y$ for $Y\in \S_1\cap \S_2$ yields a projection $\pi \colon \overline{Y}(V,\A,\G)\to D_{\S_1\cap \S_2}$ and it is clear that the last constant in the sequence coincides with $G(X_1,X_2)$. As a result, we conclude that $\pi^*G(X_1,X_2)=HG(\S_1,\S_2)$ for some power series $H$ with coefficients in $F_{\Sigma}(V,\A,\G)$. \\
        It remains to deal with the case when $\Sigma$ contains only one maximal $\G$-nested sets, say $\S$. Let $\beta$ and $\gamma$ be two adapted bases for $\S$ such that $(\S,\beta),(\S,\gamma)\in \Sigma$. We write $L_\beta$ and $L_\gamma$ for the solutions of the holonomy equation on $Y(V,\A)$ localized at $\S$ with respect to $\beta$ and $\gamma$. Also put coordinates $u_X^\beta$ and $u_X^\gamma$ for $X\in \S$ on $U_{\S,\beta}$ and $U_{\S,\gamma}$ respectively. There is a constant power series $G(\beta,\gamma)$ such that $L_\gamma=L_\beta G(\beta,\gamma)$. In what follows, we take regularizations with respect to $U_{\S,\beta}$.\\
        Since $G(\beta,\gamma)$ is constant, it coincides with its regularization along $D_\S$. Thus,
        \begin{align*}
            G(\beta,\gamma)&=\text{Reg}\left( \left(L_\beta\right)^{-1}L_\gamma, D_\S \right) = \text{Reg}\left( L_\gamma, D_\S \right) \\
            &= \text{Reg}\left( f_\gamma \prod_{X\in \S} \left( u_Y^{\gamma}\right)^{t_X}, D_\S \right) \\
            &= \prod_{X\in \S} \text{Reg}\left( \left( u_Y^{\gamma}\right)^{t_X}, D_\S \right).
        \end{align*}
        Computations analogous to the ones carried out in the previous case show that the last series has coefficients in $F_\Sigma(V,\A,\G)$.
    \end{proof}
    \begin{remark}
        Recall that we have defined a filtration on $Z_\Sigma(V,\A,\G)$ by first considering for all $\S_1,\S_2$ the $k$-vector space $W_n(\S_1,\S_2)$ generated by the coefficients of terms of degree at most $n$ in $G(\S_1,\S_2)$. We have claimed that $\sum_n W_n(\S_1,\S_2)$ is actually a $k$-algebra. In the proof of Proposition~\ref{prop: decomposition of associators} we have seen that $G(\S_1,\S_2)$ is the regularization of $L_{\S_2}$ at $D_{\S_1}$. Thus, the fact that $\sum_n W_n(\S_1,\S_2)$ is closed under products follows immediately once we show that the $k$-vector space generated by the coefficients of $L_{\S_2}$ is naturally a $k$-algebra. This will be a direct consequence of Proposition~\ref{prop: realizing reduced bar complex as functions} in the next section.
    \end{remark}
    \subsubsection{Considerations on monodromy}
    We conclude this section with some remarks about the monodromy of the solutions of the holonomy equation. Let $(V,\A)$ be a hyperplane arrangement and $\G$ a building set for $L(\A)$. Let $\S$ be a maximal $\G$-nested set and denote by $L_\S$ the solution of the holonomy equation associated with $Y(V,\A)$ localized at $\S$. Also write
    \[
        f_\S= L_\S \prod_{Y\in\S} \left(u_Y^\S\right)^{-t_Y}
    \]
    as in Theorem~\ref{theo: solutions of holonomy equation}.\\
    Given $X\in \G$, denote by $M_X$ the monodromy operator given by analytic continuation along a loop around $D_X$. In order to describe $M_XL_\S$, we first assume that $X\in \S$. Observe that $f_\S$ is single-valued on a neighborhood of the point $D_\S$, hence $M_Xf_\S=f_\S$. Moreover, 
    \[
        M_X \left(u_X^\S\right)^{t_X}= M_X \exp\left(t_X \log u_X^\S\right)= \exp\left(t_X \log u_X^\S +2\pi i t_X \right) = \left(u_X^\S\right)^{t_X} e^{2\pi i t_X}.
    \]
    As a result, we have
    \[
        M_XL_\S= L_\S e^{2\pi i t_X}.
    \]
    If $X\notin \S$, we may consider a maximal $\G$-nested set $\S'$ that contains $X$. If $L_{\S'}$ is the solution of the holonomy equation localized at $\S'$, we have $L_\S= L_{\S'}G(\S',\S)$. As a result,
    \[
        M_XL_\S= M_X  L_{\S'}G(\S',\S) = L_{\S'}e^{2\pi i t_X} G(\S',\S)= L_\S G(\S,\S') e^{2\pi i t_X} G(\S',\S).
    \]
    This shows that the monodromy ring of $L_\S$ can be computed through the coefficients of the associators $G(\S,\S')$.

    \subsection{Generalized multiple polylogarithms}
        
    Let $(V,\A)$ be an essential hyperplane arrangement over $k$. Fix a building set $\G$ for $L(\A)$. In this section, we prove that the $k$-algebra $\mathcal{L}_\S(V,\A)$ generated by the coefficients of the solution of the holonomy equation localized at a maximal $\G$-nested set $\S$ is isomorphic, as a differential $k$-algebra, to the reduced bar complex of $Y(V,\A)$. As a consequence, we deduce the existence of primitives in said algebra under the assumption that $(V,\A)$ is supersolvable. \\
    In \cite{Brown-Multiple_zeta_values_and_periods_of_moduli_spaces_of_curves}, this is carried out in the case of the braid arrangement by decomposing both $\mathcal{L}_\S(V,\A)$ and $B(V,\A)$ into a tensor product of $1$-dimensional reduced bar complexes using a sequence of fibrations provided by the assumption of supersolvability. However, the isomorphism between $\mathcal{L}_\S(V,\A)$ and $B(V,\A)$ is valid more in general for arbitrary hyperplane arrangements, as we now show. We start with a result by Kohno which relates the reduced bar complex to the holonomy algebra of $(V,\A)$.
    \begin{proposition}[Kohno]
    \label{prop: kohno}
        The $k$-linear dual of $B(V,\A)$ is isomorphic to $\widehat{A}(V,\A)$.
    \end{proposition}
    \begin{proof}
        By \cite[Proposition 2.3]{Kohno-Bar_complex_of-the_Orlik_Solomon_algebra} and \cite[Proposition 3.1]{Kohno-Bar_complex_of-the_Orlik_Solomon_algebra}, the dual of the reduced bar complex of the affine complement $X(V,\A)$ is isomorphic to a $k$-algebra defined in the same way as the holonomy algebra $\widehat{A}(V,\A)$ leaving out the relation $\sum_{H\in\A} t_H=0$. In degree $1$, for all $H\in \A$ the form $d\log x_H$ is dual to the element $t_H$. \\
        For all $H\in \A$, fix a vector $y_H\in H$. Also fix a hyperplane $H_0 \in \A$ and let $x_H=y_H/y_{H_0}$ for $H\in \A$, so the first de Rham cohomology group of $Y(V,\A)$ has a basis over $k$ given by $d\log x_H$ for $H\in \A\setminus\{H_0\}$. On the other hand, let $t_H\in H^1_{\text{dR}}(X(V,\A),k)^\lor $ be dual to $d\log y_H$ with respect to the standard basis of $ H^1_{\text{dR}}(X(V,\A),k)$ given by $d\log y_K$ for $K\in \A$. It is apparent that the canonical quotient morphism $X(V,\A)\to Y(V,\A)$ induces an injective map $H^1_{\text{dR}}(Y(V,\A),k)\to H^1_{\text{dR}}(X(V,\A),k)$ which sends $d\log x_H$ to $d\log y_H-d\log y_{H_0}$. It follows that the image of this map is the kernel of $\sum_{H\in \A} t_H$, so $H^1_{\text{dR}}(Y(V,\A),k)^\lor$ is the quotient of $H^1_{\text{dR}}(X(V,\A),k)^\lor$ modulo the line generated by $\sum_{H\in\A} t_H$.\\
        The reduced bar complex $B(V,\A)$ of $Y(V,\A)$ is a $k$-subalgebra of the free shuffle algebra of a basis of $H^1_{\text{dR}}(Y(V,\A),k)$ over $k$. Its dual is therefore a quotient of the dual of said tensor algebra, which is given by the $k$-algebra of non-commutative power series in the indeterminates $\{ t_H\mid H\in \A\}$ modulo the relation $\sum_{H\in\A} t_H=0$. Combined with the above result of Kohno, we conclude that the dual of $B(V,\A)$ is isomorphic to $\widehat{A}(V,\A)$.
    \end{proof}
    \begin{remark}
        Other than only the linear structure, also the Hopf algebra structure of $B(V,\A)$ is dual to the one of $A(V,\A)$. This can be easily checked simply at the level of the free shuffle algebra in which the reduced bar complex embeds and its dual.
    \end{remark}
    \noindent
    Let us turn to the study of the multi-valued functions arising as coefficients of a solution of the holonomy equation. Fix a graded basis $\beta$ of $B(V,\A)$ as a $k$-vector space. In view of Proposition~\ref{prop: kohno}, every element $F\in \widehat{A}(V,\A)\widehat{\otimes}_k\C$ can be written uniquely as $F=\sum_{w\in \beta} f_w w^\lor$ with $f_w\in \C$, where for all $w\in\beta$ we denote by $w^\lor$ its dual element with respect to the basis $\beta$.
    \begin{definition}
        Let $F$ be a multi-valued function on $Y(V,\A)$ with values in $\widehat{A}(V,\A)\widehat{\otimes}_k \C$. For a point $z$ in the universal covering of $Y(V,\A)(\C)$, let us write $F=\sum_{w\in\beta} f_w(z) w^\lor$ with $f_w(z)\in\C$.
        \begin{enumerate}
            \item We denote by $C(F)$ the $k$-algebra generated by the multi-valued functions of the form $z \mapsto f_w(z)$ for all $w\in \beta$.
            \item Let $\G$ be a building set for $L(\A)$ and $\S$ a maximal $\G$-nested set. If $L_\S$ is the solution of the holonomy equation on $Y(V,\A)$ with boundary condition at $D_\S$, we call $\mathcal{L}_\S(V,\A)=C(L_\S)$ the algebra of \emph{generalized multiple polylogarithms} on $Y(V,\A)$ localized at $\S$.
        \end{enumerate}
    \end{definition}
    \begin{remark}
        The definition of $C(F)$ does not depend on the choice of a graded basis $\beta$ for $B(V,\A)$. Indeed, let $\beta'$ be another such basis and write $F=\sum_{v\in \beta'}f'_v v^\lor$ with $f'_v\in \C$. for all $w\in \beta$ write $w=\sum_{v\in \beta'}\alpha_{w,v} v$ with $\alpha_{w,v}\in k$, where the sum is finite. We then have $f'_v=\sum_{w\in\beta} f_w\alpha_{w,v}$, which shows that the $k$-algebra generated by the $f_v$'s is contained in the one generated by the $f_w$'s. Exchanging the roles of $\beta$ and $\beta'$ yields the claim.
    \end{remark}
    \noindent
    Let $F$ be a multi-valued function on $Y(V,\A)$ with values in $\widehat{A}(V,\A)\widehat{\otimes}_k \C$ and, for a fixed graded basis $\beta$ of $B(V,\A)$, write $F(z)=\sum_{w\in\beta}f_w(z)w^\lor$. Consider the $k$-linear map $\phi_F \colon B(V,\A)\to C(F)$ which sends $w\in\beta$ to $f_w$. Our next goal is to prove that $\phi_F$ is a differential $k$-algebra isomorphism when $F=L_\S$ is a solution of the holonomy equation, after tensoring with the regular functions of $Y(V,\A)$. We first treat the $k$-algebra structure.
    \begin{lemma}
        If $F(z)$ is a group-like element for all $z$ in the universal covering of $Y(V,\A)(\C)$, then $\phi_F$ is a $k$-algebra homomorphism.
    \end{lemma}
    \begin{proof}
        Notice that for all $w\in B(V,\A)$ and all $z$ as in the statement we have $f_w(z)=\langle F(z), w\rangle$. As a result, for all $w_1,w_2\in \beta$
        \begin{align*}
            \phi_{F}(w_1 \,\sh\, w_2) (z) &= \langle F(z), w_1 \,\sh\, w_2 \rangle
            =\langle \Delta F(z), w_1\otimes w_2\rangle = \langle F(z)\otimes F(z), w_1\otimes w_2\rangle \\
            &= \langle F(z),w_1\rangle \langle F(z), w_2\rangle=\phi_F(w_1)(z)\phi_F(w_2)(z),
        \end{align*}
        hence the claim follows by $k$-linearity.
    \end{proof}
    \noindent
    Together with Theorem~\ref{theo: solutions of holonomy equation}, this lemma implies that for all maximal $\G$-nested sets $\S$ the map $\phi_{L_\S} \colon  B(V,\A)\to \mathcal{L}_\S(V,\A)$ is a $k$-algebra homomorphism.\\
    On taking the tensor product with the regular functions of $Y(V,\A)$, we get a map
    \[
        \phi_F \colon  B(V,\A)\otimes_k \mathcal{O}_{Y(V,\A)} \to C(F)\otimes_k \mathcal{O}_{Y(V,\A)}
    \]
    which is $k$-linear and, if $F$ is group-like, also a $k$-algebra homomorphism.
    \begin{proposition}
    \label{prop: realizing reduced bar complex as functions}
        Let $\G$ be a building set for $L(\A)$ and $\S$ a maximal $\G$-nested set. The $k$-algebra $\mathcal{L}_\S(V,\A)\otimes_k \mathcal{O}_{Y(V,\A)}$ has a natural structure of differential $k$-algebra. Moreover, the $k$-algebra homomorphism $\phi_{\mathcal{L}_\S}$ is an isomorphism of differential $k$-algebras.
    \end{proposition}
    \begin{proof}
        Let us decompose $A(V,\A)=\bigoplus_{n\ge 0} A_n$ and $B(V,\A)=\bigoplus_{n\ge 0} B_n$ according to the degree. Fix a basis $\beta_n$ of $B_n$ as a $k$-vector space for every $n\ge 0$ and let us write $L_\S=\sum_{n=0}^\infty \sum_{w\in \beta_n} f_w w^\lor$. For $H\in \A$, set $\omega_H=d\log x_H$ in projective coordinates. Given that $L_\S$ satisfies the holonomy equation, we have
        \[
        dL_\S=\Omega L_\S=\sum_{n=0}^\infty\sum_{w\in\beta_n} \sum_{H\in\A} f_w\omega_H \otimes t_Hw^\lor. 
        \]
        Let $T$ be the free shuffle algebra over the set $\{\omega_H\mid H\in \A\}$ and let $T_n$ denote its degree $n$ part. The embedding $B_n\hookrightarrow T_n $ yields a quotient map $T_n^\lor \to A_n$. For all $w\in \beta_n$ and $H\in \A$ the element $t_Hw^\lor\in T_{n+1}^\lor$ is sent via this map to a certain linear combination of the elements in $\beta_{n+1}$, say $t_Hw^\lor= \sum_{v\in\beta_{n+1}} \alpha_v^{w,H}v^\lor$ with $\alpha_v^{w,H}\in k$. It follows that for all $v\in \beta_{n+1}$
        \[
            df_v = \sum_{w\in \beta_{n}}\sum_{H\in\A} \alpha_v^{w,H} f_w \omega_H,
        \]
        which shows that $df_v\in \mathcal{L}_\S(V,\A)\otimes_k \Omega^1_{Y(V,\A)/k}$. As a result, $\mathcal{L}_\S(V,\A)\otimes_k \mathcal{O}_{Y(V,\A)}$ has a natural structure of differential $k$-algebra.\\
        For the second claim, let $\gamma_n$ be the standard basis of $T_n$ given by non-commutative monomials and consider $L_n=\sum_{u\in\gamma} u\otimes u^\lor\in T_n\otimes_k T_n^\lor$, the dual elements being taken with respect to $\gamma_n$. Via the isomorphism $T_n\otimes_k T_n^\lor\cong \text{Hom}_k(T_n,T_n)$, the element $L_n$ corresponds to the identity map of $T_n$. Under the restriction map $T_n\otimes_k T_n^\lor\cong \text{Hom}_k(T_n,T_n) \to \text{Hom}_k(B_n,T_n)\cong T_n\otimes_k B_n^\lor$ the identity is sent to the inclusion $B_n\hookrightarrow T_n$, which lies in the image of the canonical map $B_n\otimes_k 
        B_n^\lor\cong \text{Hom}_k(B_n,B_n)\to \text{Hom}_k(B_n,T_n)\cong T_n\otimes_k B_n^\lor$ induced by the same inclusion. As an element of $B_n\otimes_k B_n^\lor$, the image of $L_n$ is $\sum_{w\in \beta_n} w\otimes w^\lor$, the duals being taken with respect to $\beta_n$.\\
        The differential $d \colon B(V,\A)\to B(V,\A)\otimes_k\Omega^1_{Y(V,\A)/k}$ extends naturally to a $k$-linear map $d \colon  T\to T_n\otimes_k\Omega^1_{Y(V,\A)/k}$ by setting $d([\omega_{H_1}\vert \dots \vert \omega_{H_n}])=\omega_{H_1}\otimes[\omega_{H_2}\vert \dots \vert \omega_{H_n}]$ for $H_1,\dots, H_n\in \A$. Although the latter map is not a differential, as $d^2\neq 0$ outside of $B(V,\A)$ in general, tensoring with the identity of $T_n^\lor$ and of $B_n^\lor$ still gives a commutative diagram
        \[
        \begin{tikzcd}
            T_n \otimes_k T_n^\lor \ar[r, "d"] \ar[d] &  \Omega^1_{Y(V,\A)/k} \otimes_k   T_{n-1} \otimes_k T_{n-1}^\lor  \ar[d] \\
            T_n\otimes_k B_n^\lor \ar[r, "d"] & \Omega^1_{Y(V,\A)/k} \otimes_k  T_{n-1} \otimes_k B_{n-1}^\lor \\
            B_n\otimes_k B_n^\lor \ar[r, "d"] \ar[u, hook] & \Omega^1_{Y(V,\A)/k} \otimes_k  B_{n-1} \otimes_k B_{n-1}^\lor \ar[u, hook]
        \end{tikzcd}
        \]
        Consider the differential form $\Omega=\sum_{H\in \A} \omega_H \otimes [\omega_H]^\lor$, which we regard as an element of $\Omega^1_{Y(V,\A)/k}\otimes_k T_1^\lor$. We then have $dL_n=\Omega L_{n-1}$. Indeed, ordering the elements of $\A$ as $H_1,\dots, H_s$ for convenience of notation, we have
        \begin{align*}
            \Omega L_{n-1}&=\sum_{I=(i_1,\dots, i_{n-1})}\sum_{H\in A} \omega_H \otimes [\omega_{H_{i_1}}\vert \dots \vert \omega_{H_{i_{n-1}}}] \otimes [\omega_H]^\lor [\omega_{H_{i_1}}\vert \dots \vert \omega_{H_{i_{n-1}}}]^\lor\\
            &=\sum_{I=(i_1,\dots, i_{n-1})}\sum_{H\in A} \omega_H \otimes [\omega_{H_{i_1}}\vert \dots \vert \omega_{H_{i_{n-1}}}] \otimes [\omega_H\vert \omega_{H_{i_1}}\vert \dots \vert \omega_{H_{i_{n-1}}}]^\lor\\
            &= \sum_{I=(i_1,\dots, i_n)} \omega_{H_{i_1}}\otimes [\omega_{H_{i_2}}\vert \dots \vert \omega_{H_{i_{n}}}] \otimes [\omega_{H_{i_1}}\vert \dots \vert \omega_{H_{i_{n}}}]^\lor \\
            &= \sum_{I=(i_1,\dots, i_n)} d([ \omega_{H_{i_1}}\vert \dots \vert \omega_{H_{i_{n}}}]) \otimes [ \omega_{H_{i_1}}\vert \dots \vert \omega_{H_{i_{n}}}]^\lor \\
            &= \sum_{u\in \gamma_n} du\otimes u^\lor = dL_n.
        \end{align*}
        Here, we are using the fact that for $v,w\in T $ we have $v^\lor w^\lor =[v\vert w]^\lor$. Indeed, for all $u\in \gamma$ we have
        \[
            \langle v^\lor w^\lor, u \rangle =\langle v\otimes w, \Delta u \rangle = \langle v\otimes w, \sum_{[u_1\vert u_2]= u} u_1\otimes u_2 \rangle =  \sum_{[u_1\vert u_2]= u} \langle v, u_1\rangle \langle w, u_2\rangle,
        \]
        which equals $1$ if $u=[v\vert w]$ and $0$ otherwise.\\
        Following the diagram above, if we denote by $L_n$ also its image in $T_n\otimes B_n^\lor$, it is clear that the identity $d L_n=\Omega L_{n-1}$ takes place in $\Omega^1_{Y(V,\A)/k} \otimes_k  B_{n} \otimes_k B_{n}^\lor $ as well.\\
        On the other hand, denote by $L_{\S,n}$ the part of degree $n$ of $L_\S$. This can be regarded as an element of $\mathcal{L}_\S(V,\A)\otimes A_n= \mathcal{L}_\S(V,\A)\otimes B_n^\lor$. Tensoring with the identity of $B_n^\lor$ we have a map $\phi_{L_\S} \colon  B_n\otimes B_n^\lor \to \mathcal{L}_\S(V,\A)\otimes B_n^\lor $ which clearly maps $L_n$ to $L_{\S,n}$. Moreover, 
        \[
            \phi_{L_\S}(dL_n)=\phi_{L_\S}(\Omega L_{n-1})= \Omega\phi_{L_\S}(L_{n-1})=\Omega L_{\S,n-1}=dL_{\S,n}=d\phi_{L_\S}(L_n).
        \]
        By passing to the coefficients it follows that for all $w\in B_n$ we have $\phi_{L_\S}(dw)=df_w$. This means that $\phi_{L_\S} \colon  B(V,\A)\otimes_k \mathcal{O}_{Y(V,\A)}\to \mathcal{L}_\S(V,\A)\otimes_k \mathcal{O}_{Y(V,\A)}$ is a homomorphism of differential $k$-algebras.\\
        It is clear that $\phi_{L_\S}$ is surjective, thus non-zero. Since $B(V,\A)\otimes_k \mathcal{O}_{Y(V,\A)}$ is a differentially simple $k$-algebra and $\phi_{L_\S}$ respects the differential, said map is also injective, hence an isomorphism.
    \end{proof}
    \begin{remark}
        The weight filtration $W_\bullet B(V,\A)$ induces a corresponding weight filtration $W_\bullet \mathcal{L}_\S(V,\A)$ via the above isomorphism.
    \end{remark}
    \noindent
    In view of Theorem~\ref{theo: reduced bar complex is unipotent closure}, we deduce the following
    \begin{corollary}
        \label{cor: generalized polylogs are unipotent closure}
        Let $\G$ be a building set for $L(\A)$ and $\S$ a maximal $\G$-nested set. The differential $k$-algebra $\mathcal{L}_\S(V,\A)\otimes_k \mathcal{O}_{Y(V,\A)}$ of generalized multiple polylogarithms localized at $\S$ is a unipotent closure of $\mathcal{O}_{Y(V,\A)}$.
    \end{corollary}
    \begin{remark}
        It is worth noting that Proposition~\ref{prop: realizing reduced bar complex as functions} holds for arbitrary hyperplane arrangements, not necessarily supersolvable or with enough modular elements. This fact was proven in \cite{Brown-Multiple_zeta_values_and_periods_of_moduli_spaces_of_curves} for the braid arrangement using Proposition~\ref{prop: reduced bar complex for supersolvable arrangements} inductively and comparing the one-dimensional cases only. 
    \end{remark}
    \noindent
    Under fiber type assumptions, one can make the higher de Rham cohomology groups of the differential algebra $\mathcal{L}_\S(V,\A)\otimes_k \mathcal{O}_{Y(V,\A)}$ vanish through Corollary~\ref{cor: vanishing top degree cohomology reduced bar complex}.
    \begin{corollary}
        \label{cor: generalized polylogs have primitives}
        Let $\G$ be a building set for $L(\A)$ and $\S$ a maximal $\G$-nested set. If $(V,\A)$ is fiber type, then for all $i=1,\dots, l$
        \[
            H^i_{\textnormal{dR}}( \mathcal{L}_\S(V,\A)\otimes_k \mathcal{O}_{Y(V,\A)})=0.
        \]
        Moreover, a primitive of a closed differential form $F\in \mathcal{L}_\S(V,\A)\otimes_k \Omega^i_{Y(V,\A)/k}$ has weight at most one higher than $F$.
    \end{corollary}
    \noindent
    Thanks to Proposition~\ref{prop: realizing reduced bar complex as functions}, one can realize the reduced bar complex of $(V,\A)$ as an algebra of multi-valued functions on $Y(V,\A)(\C)$. Another approach would have been to exploit Chen's iterated integrals \cite{Chen-Iterated_path_integrals}, which we now briefly describe.\\
    Set for short $Y=Y(V,\A)(\C)$ and let $\widehat{Y}$ be its universal covering. Fix a point $x\in Y$ and consider a path $\gamma \colon  [0,1]\to Y$ starting at $x$. Let $\omega_1,\dots, \omega_n$ be closed algebraic differential $1$-forms on $Y$. As usual, we denote by $[\omega_1\,\vert\, \dots \,\vert\, \omega_n]$ the corresponding element in the tensor algebra over $\Omega^1_{Y(V,\A)/k}$. The \emph{iterated integral} of $[\omega_1\,\vert\, \dots \,\vert\, \omega_n]$ along $\gamma$ is defined as
    \[
        \int_\gamma [\omega_1\,\vert\, \dots \,\vert\, \omega_n] = \int_{0<t_1<\dots<t_n<1} \gamma^*\omega_1(t_1)\wedge\dots \wedge \gamma^*\omega_n(t_n).
    \]
    This definition can be extended by $k$-linearity to the tensor algebra over $\Omega^1_{Y(V,\A)/k}$. The iterated integral of $\sum_{I}c_I [\omega_{i_1}\,\vert\, \dots \,\vert\, \omega_{i_n}] $ only depends on the homotopy class of $\gamma$ if and only if $\sum_{I}c_I [\omega_{i_1}\,\vert\, \dots \,\vert\, \omega_{i_n}] $ belongs to the reduced bar complex of $(V,\A)$. In this case, it can be regarded as a multi-valued function on $Y$. A version with tangential base points is also possible.\\
    Associating to an element in the reduced bar complex the corresponding iterated integral yields an isomorphism of differential algebras. Thus, iterated integrals provide an explicit description of the solutions of the holonomy equation treated so far. The version with tangential base points corresponds to the boundary condition at points which are maximal intersections of boundary divisors in our setting.

    \subsection{Regularized primitives}
        
    Let $(V,\A)$ be an essential hyperplane $(l+1)$-arrangement over $k$, $\G$ a building set for $L(\A)$ and $\S$ a maximal $\G$-nested set. As we have seen in Proposition~\ref{prop: realizing reduced bar complex as functions}, there is an isomorphism of differential $k$-algebras 
    \[
        B(V,\A)\otimes_k \mathcal{O}_{Y(V,\A)} \cong \mathcal{L}_\S(V,\A)\otimes_k \mathcal{O}_{Y(V,\A)}.
    \]
    In analogy with the reduced bar complex, for $i=1,\dots, l$ let us set
    \[
        \Omega^i\mathcal{L}_\S(V,\A) = \mathcal{L}_\S(V,\A)\otimes_k \Omega^i_{Y(V,\A)/k}.
    \]
    If $(V,\A)$ is supersolvable, Corollary~\ref{cor: generalized polylogs have primitives} ensures that for every $F\in \Omega^l \mathcal{L}_\S(V,\A)$ there is some $G \in \Omega^{l-1} \mathcal{L}_\S(V,\A)$ such that $dG=F$. Our goal is to use the existence of primitives to reduce the computation of integrals on $\overline{Y}(V,\A,\G)$ to the one of integrals on the boundary divisors $D_X$ for $X\in \G$. The key ingredient is Stoke's theorem for functions with logarithmic singularities.\\
    However, to apply Stoke's theorem it is necessary to find a primitive which has no poles along the boundary of the integration domain. That is, suppose we want to integrate $F$ over a homology class with boundary contained in a set of divisors $D_X$ for $X\in \G$, along which $F$ has no poles. This integral equals the one of $G$ on the boundary, provided that $G$ has no poles either along the divisors $D_X$ which intersect the boundary. Thus, we need to remove possible spurious poles from $G$ along divisors where $F$ has at most logarithmic singularities. To achieve this, it is essential to assume that $(V,\A)$ has enough $\G$-modular elements.\\
    Let $\Sigma$ be a family of maximal $\G$-nested sets equipped with an adapted basis. We assume that $\Sigma$ defines a connected subgraph of the graph of all maximal $\G$-nested sets. There is a corresponding associator algebra $Z_\Sigma(V,\A,\G)$. If we define 
    \[
        \mathcal{L}_\Sigma(V,\A)= \mathcal{L}_\S(V,\A)\otimes_k Z_\Sigma(V,\A,\G),
    \]
    we obtain a $k$-algebra of multi-valued functions on $Y(V,\A)(\C)$ which include all the coefficients of the solution $L_\S$ for $\S\in \Sigma$ of the holonomy equation localized at $\S$. Similarly, we define $\Omega^i\mathcal{L}_\Sigma (V,\A)$ by taking the tensor product with $\Omega^i_{Y(V,\A)/k}$, as above. Notice that the content of Corollary~\ref{cor: generalized polylogs have primitives} applies to $\mathcal{L}_\Sigma(V,\A)$ as well.\\
    Our goal is to prove the following
    \begin{proposition}
        \label{prop: regularizing primitives}
        Suppose that $(V,\A)$ has enough $\G$-modular elements. Let $F\in\Omega^l\mathcal{L}_\Sigma(V,\A)$ and suppose that for every $\S\in \Sigma$ and $X\in \S$ the differential form $F$ has no poles along the boundary divisor $D_X$. Then there exists a primitive $G\in \Omega^{l-1}\mathcal{L}_\Sigma(V,\A)$ of $F$ which has no poles along the boundary divisors $D_X$ for all $X\in \S$, $\S\in \Sigma$.
    \end{proposition}
    \noindent
    Before providing a proof, we need one more lemma.
    \begin{lemma}
    \label{lemma: generalized polylogs contain regularized restrictions}
        Let $\S$ be a maximal $\G$-nested set and $X\in \S$. Suppose that there is a $\G$-modular element $M\in \G$ such that $V^*=X\oplus M$ and let $\pi \colon \overline{Y}(V,\A,\G)\to D_X$ be a retraction as in Proposition~\ref{prop: retractions to boundary divisors}.\\
        Then for all $f\in L_\S(V,\A)\otimes_k \mathcal{O}_{Y(V,\A)}$ the regularized restriction
        \[
            \pi^*\textnormal{Reg}\left( f, \left\{u_X^\S=0\right\} \right)
        \]
        belongs to $L_\S(V,\A)\otimes_k \mathcal{O}_{Y(V,\A)}$.
    \end{lemma}
    \begin{proof}
        We argue by induction on the weight of $f$. In case $f$ has weight $0$, we have $f\in \mathcal{O}_{Y(V,\A)}$, hence the claim follows from Corollary~\ref{cor: embedding regualr functions with enough modular elements}. Notice that the content of Corollary~\ref{cor: embedding regualr functions with enough modular elements} still works assuming the existence of the retraction $\pi$, independently of the presence of enough modular elements. \\
        Suppose that the statement holds for all functions in $W^iL_\S(V,\A)\otimes_k \mathcal{O}_{Y(V,\A)}$ for some $i\ge 1$ and let $f$ have weight $i+1$. Set for short
        \[
            g=\pi^*\textnormal{Reg}\left( f, \left\{u_X^\S=0\right\} \right).
        \]
        For all $Y\in \S$ we have
        \[
            \frac{\partial g}{\partial u_Y^\S} = \pi^*\textnormal{Reg}\left( \frac{\partial f}{\partial u_Y^\S}, \left\{u_X^\S=0\right\} \right).
        \]
        Since all derivatives of $f$ have weight $i$, by induction hypothesis all derivatives of $g$ belong to $L_\S(V,\A)\otimes_k \mathcal{O}_{Y(V,\A)}$. By Corollary~\ref{cor: generalized polylogs are unipotent closure} there is $h\in L_\S(V,\A)\otimes_k \mathcal{O}_{Y(V,\A)}$ whose derivatives agree with those of $g$. It follows that $h$ equals $g$ up to an additive constant in $k$. We conclude that $g\in L_\S(V,\A)\otimes_k \mathcal{O}_{Y(V,\A)}$, as desired.
    \end{proof}
    \noindent
    We are now ready to prove Proposition~\ref{prop: regularizing primitives}. The presence of enough modular elements allows us to exploit the same strategy as in \cite{Brown-Multiple_zeta_values_and_periods_of_moduli_spaces_of_curves}.
    \begin{proof}
        Let $F\in\Omega^l\mathcal{L}_\Sigma(V,\A)$ be as in the statement. By Corollary~\ref{cor: generalized polylogs have primitives} we may find $G\in \Omega^{l-1}\mathcal{L}_\Sigma(V,\A)$ such that $dG=F$. Let $\S\in \Sigma$ and fix $X\in D_\S$. We show that there is $G'\in \Omega^{l-1}\mathcal{L}_\Sigma(V,\A) $ such that $dG'=F$ and without poles along $D_X$. Moreover, we choose $G'$ in such a way that it has no poles along all irreducible boundary divisors where $G$ has at most logarithmic singularities. The statement then follows by iterating this argument for all $X\in \G$ appearing in the maximal $\G$-nested sets in $\Sigma$, replacing at each time $G$ with the new primitive $G'$. \\
        Since $(V,\A)$ has enough $\G$-modular elements, by Proposition~\ref{prop: retractions to boundary divisors} there is a retraction $\pi \colon  \overline{Y}(V,\A,\G)\to D_X$ induced by a linear projection at the level of arrangements. We work in the local chart $U_\S$ with respect to an adapted basis $\beta$ such that $(\S,\beta)\in \Sigma$. For all $Y\in \S$ write $\omega_Y=\bigwedge_{Z\neq Y} du_Z^\S$. Thus,
        \[
            G=\sum_{Y\in \S} g_Y \omega_Y
        \]
        for some $g_Y\in \mathcal{L}_\Sigma(V,\A)\otimes_k \mathcal{O}_{Y(V,\A)}$. Functions in $\mathcal{L}_\Sigma(V,\A)$ have at most logarithmic singularities along $D_X$, hence we may write $g_Y=\sum_{i=0}^{n_Y} a_i^Y\log^i u_X^\S$ for some $a_{i}^Y\in \mathcal{L}_\Sigma(V,\A)\otimes_k \mathcal{O}_{Y(V,\A)}$. By expanding the $a_i^Y$'s into Laurent series in a neighborhood of $D_\S$, we have
        \[
            G=\sum_{Y\in \S}\sum_{i=0}^{n_Y}\sum_{j=-N_Y}^\infty b_{i,j}^Y\left( u_X^\S\right)^{j} \log^i u_X^\S
        \]
        for some $N_Y\ge 0$ and functions $b_{i,j}^Y$ which do not depend on $u_X^\S$. Notice that
        \[
            b_{i,j}^Y=\frac{1}{(N_Y+j)!}\left(\frac{\partial^{N+j}}{\partial\left( u_X^\S\right)^{N+j}}\left(\left( u_X^\S\right)^{N_Y}a_i^Y\right)\right)\Bigg\vert_{\left\{ u_X^\S=0\right\}}.
        \]
        By Proposition~\ref{prop: realizing reduced bar complex as functions} $\mathcal{L}_\Sigma(V,\A)\otimes_k \mathcal{O}_{Y(V,\A)}$ is closed under taking derivatives, so Lemma~\ref{lemma: generalized polylogs contain regularized restrictions} implies that $\pi^*b^Y_{i,j}\in \mathcal{L}_\Sigma(V,\A)\otimes_k \mathcal{O}_{Y(V,\A)}$. Consider now
        \[
            h_Y=
            \begin{cases}
                \sum_{i=0}^{n_Y} \sum_{j=-N_Y}^{-1} \pi^*b^Y_{i,j} \left( u_X^\S\right)^{j} \log^i u_X^\S & \text{if $Y\neq X$}; \\
                \sum_{i=0}^{n_X} \sum_{j=-N_Y}^{-1} \pi^*b^X_{i,j} \left( u_X^\S\right)^{j} \log^i u_X^\S + \sum_{i=1}^{n_X} \pi^* b_{i,0}^X \log^iu_X^\S & \text{if $Y=X$},
            \end{cases}
        \]
        and set $H=\sum_{Y\in \S} h_Y \omega_Y\in \Omega^{l-1} \mathcal{L}_\Sigma(V,\A)$.\\
        It is straightforward to see that $dH$ is either zero or has a pole along $D_X$. Indeed, for $Y\neq X$
        \[
            d(h_Y\omega_Y) = \pm \frac{\partial h_Y}{\partial  u_Y^\S}\bigwedge_{Z\in \S} du_Z^\S= \sum_{i=1}^{n_Y} \sum_{j=-N_Y}^{-1} \frac{\partial \pi^*b^Y_{i,j}}{\partial  u_Y^\S} \left( u_X^\S\right)^{j} \log^i u_X^\S\bigwedge_{Z\in \S}du_Z^\S,
        \]
        while, for $Y=X$, up to the sign $d(h_X\omega_X)$ equals
        \[
            \left(\sum_{i=0}^{n_X} \sum_{j=-N_Y}^{-1} \pi^*b^X_{i,j} 
             \frac{(j-i\log u_X^\S) \log^{i-1} u_X^\S}{\left( u_X^\S\right)^{1-j}} +
            \sum_{i=1}^{n_X} i\pi^* b_{i,0}^X \frac{\log^{i-1}u_X^\S}{u_X^\S}\right) \bigwedge_{Z\in \S}du_Z^\S.
        \]
        On the other hand, set $G'=G-H$. Notice that $G'$ has no pole along $D_X$. To check this, it is enough to prove that $g_Y-h_Y$ has no pole along $D_X$ for all $Y\in X$. For $Y\neq X$, the Laurent expansion of $g_Y-h_Y$ at $\left\{u_X^\S=0\right\}$ reads
        \[
            g_Y-h_Y=\sum_{i=0}^{n_Y}\sum_{j=0}^\infty b_{i,j}^Y\left( u_X^\S\right)^{j} \log^i u_X^\S,
        \]
        which has at most logarithmic singularities along $\left\{u_X^\S=0\right\}$. For $Y=X$ we have
        \[
            g_X-h_X=\sum_{j=0}^\infty b_{0,j}^Y\left( u_X^\S\right)^{j}+\sum_{i=0}^{n_Y}\sum_{j=1}^\infty b_{i,j}^Y\left( u_X^\S\right)^{j} \log^i u_X^\S,
        \]
        which presents no poles along $D_X$ either.\\
        Moreover, $dG'$ also has at most logarithmic singularities along $D_X$. This is clear for $d(g_Y-h_Y)\omega_Y$ for $Y\neq X$, since it only involves taking the derivative in $u_Y^\S$. If instead $Y=X$, all powers of $\log u_X^\S$ in the above expansion of $g_X-h_X$ are multiplied by a positive power of $u_X^\S$. Since 
        \[
            \frac{\partial }{\partial u_X^\S}\left( \left(u_X^\S\right)^{j}\log^i u_X^\S\right)= \left(u_X^\S\right)^{j-1}\log^{i-1}u_X^\S \left(i+j\log u_X^\S \right), 
        \]
        for every $j\ge 1$ this function has no pole along $\left\{u_X^\S=0\right\}$, hence the same holds for $d(g_X-h_X)\omega_X$.\\
        By assumption, $F$ has no poles along $D_X$. Thus, $dH=dG-dG'=F-dG'$ has no poles along this divisor either. By our previous analysis of $dH$, we conclude that $dH=0$, hence $dG'=F$.\\
        Finally, if $G$ has at most logarithmic singularities along a divisor $D_Y$ for $Y$ in a maximal $\G$-nested set contained in $\Sigma$, the functions $\pi^*b_{i,j}$ satisfy the same property. This implies that also $H$ has no pole along $D_Y$, and thus the same is true for $G'$. We have therefore constructed a primitive $G'\in \Omega^{l-1}\mathcal{L}_\Sigma(V,\A)$ of $F$ which has at most logarithmic singularities along $D_X$ and which does not introduce further poles where $G$ has at most logarithmic singularities, as desired.
    \end{proof}
    \begin{remark}
    In the proof, it is crucial to use the retraction $\pi$ to turn the coefficients $b^Y_{i,j}$ into functions globally defined on $Y(V,\A)$. In the absence of such retraction, the differential form $H$ could only be defined locally in a neighborhood of $D_X$, thus preventing to run the same strategy to remove poles at the remaining boundary divisors.
    \end{remark}

    \subsection{Periods of arrangements with enough modular elements}
        
    This section is devoted to the proof pf Theorem~\ref{theo: periods of hyperplane arrangements with enough modular elements}. We start by introducing the precise setup in which to expose the main theorem.\\
    Suppose $(V,\A)$ is an essential $(l+1)$-hyperplane arrangement over $k$. Let $\mathcal{P}$ and $\mathcal{R}$ be disjoint subsets of $\G$ not containing $V^*$. We define
    \[
        D^\mathcal{P}=\bigcup_{P\in \mathcal{P}} D_P, \quad D^\mathcal{R}=\bigcup_{R\in \mathcal{R}} D_R.
    \]
    Consider a homology class
    \[
        \Delta\in H_l\left(\overline{Y}(V,\A,\G)\setminus D^\mathcal{P}, D^\mathcal{R}\setminus \left( D^\mathcal{R}\cap D^\mathcal{P}\right), k\right).
    \]
    In order to compute integrals of algebraic differential forms over $\Delta$, we need to impose a few restrictions on $\Delta$.\\
    Since our goal is to compute integrals over $\Delta$, by linearity of the integral we may assume that $\Delta$ is represented by a compact manifold with corners inside $\overline{Y}(V,\A,\G)$ whose boundary lies in $D^\mathcal{R}$ and which does not intersect $D^\mathcal{P}$. By abuse of notation we shall denote this manifold by $\Delta$ as well. \\
    Moreover, it is possible to assume that $\Delta$ is stratified according to the canonical stratification of $\overline{Y}(V,\A,\G)$ induced by the simple normal crossings divisor $D^\mathcal{R}$. Thus, $\Delta$ can be written as the disjoint union of locally closed submanifolds $\Delta^{(i)}$ for $i=0,\dots, l$ of real codimension $i$ in $\Delta$ such that $\Delta^{(i)}$ is contained in the intersection of exactly $i$ irreducible divisors $D_R$ with $R\in \mathcal{R}$. In particular, $\Delta^{(0)}$ consists of a finite set of points which are maximal intersections of irreducible boundary divisors. Consequently, each of these points corresponds to a maximal $\G$-nested set $\S\subseteq \G$; we denote by $\Sigma_\Delta$ the collection of these maximal $\G$-nested sets arising from $\Delta^{(0)}$. We also attach to every $\mathcal{S}\in \Sigma_\Delta$ an adapted basis in such a way that the open subsets $U_\S\cap \Delta$ of $\Delta$ for all $\S\in \Sigma_\Delta$ induce an atlas for $\Delta$ as a manifold with corners.\\
    From now on, let us assume that $(V,\A)$ has enough $\G$-modular elements. In order to avoid polydromy problems with the generalized multiple polylogarithms introduced in the previous sections, we introduce the following restriction on $\Delta$.
    \begin{definition}
        We say that $\Delta$ is \emph{admissible} if for all $x\in \Delta$ the map
        \[
            \pi_1(\Delta, x)\to \pi_1(\overline{Y}(V,\A,\G)\setminus D^\mathcal{P}, x)
        \]
        induced by the inclusion $\Delta\hookrightarrow \overline{Y}(V,\A,\G)\setminus D^\mathcal{P} $ is the zero map.
    \end{definition}
    \begin{lemma}
        \label{lemma: admissible homology classes}
        Assume that $(V,\A)$ has enough $\G$-modular elements; let $X\in \mathcal{R}$. If $\Delta$ is admissible, then the homology class
        \[
            \partial \Delta\cap D_X \in H_{l-1}\left(D_X\setminus (D_X\cap D^\mathcal{P}), D_X\cap D^\mathcal{R}\setminus \left(D_X\cap D^\mathcal{R}\cap D^\mathcal{P}\right), k\right).
        \]
        is admissible for the hyperplane arrangement associated with $D_X$.
    \end{lemma}
    \begin{proof}
        For all $x\in \partial \Delta \cap D_X$ we look at the diagram
        \[
        \begin{tikzcd}
            \pi_1(\partial\Delta \cap D_X,x) \ar[r] \ar[d] & \pi_1(\Delta, x) \ar[d] \\
            \pi_1(D_X\setminus D_X\cap D^\mathcal{P},x) \ar[r] & \pi_1(\overline{Y}(V,\A,\G)\setminus D^\mathcal{P},x)
        \end{tikzcd}
        \]
        induced by the inclusions. The map in the bottom row is injective, since the inclusion $D_X\hookrightarrow \overline{Y}(V,\A,\G)$ admits a retraction $\pi \colon  \overline{Y}(V,\A,\G)\to D_X$ by Proposition~\ref{prop: retractions to boundary divisors}. Since the map in the right column is zero, it follows that also the one in the left column is zero.  
    \end{proof}
    \noindent
    We are now ready to state and prove our main result.
    \begin{theorem}
        \label{theo: periods of hyperplane arrangements with enough modular elements}
        Let $(V,\A)$ be an essential $(l+1)$-arrangement over $k$. Let $\G$ be a building set for $L(\A)$ and suppose that $(V,\A)$ has enough $\G$-modular elements. Fix two disjoint subsets $\mathcal{P},\mathcal{R}$ of $\G$ not containing $V^*$.\\
        Consider an admissible homology class 
            \[
                 \Delta\in H_l\left(\overline{Y}(V,\A,\G)\setminus D^\mathcal{P}, D^\mathcal{R}\setminus \left( D^\mathcal{R}\cap D^\mathcal{P}\right), k\right).
            \]
        Let $F\in W_i\Omega^l\mathcal{L}_{\Sigma_\Delta}(V,\A)$ be a differential form with possible poles only along the divisors $D_X$ for all $X\in \mathcal{P}$. Then
        \[
            \int_\Delta F \in W_{l+i}Z_{\Sigma_\Delta}(V,\A,\G).
        \]
    \end{theorem}
    \begin{proof}
        By linearity of the integral, we may assume that $\Delta$ is connected. We argue by induction on the dimension $l$ of $Y(V,\A)$.\\
        Let us first deal with the case $l=1$, so we have $Y(V,\A)(\C)=\mathbb{P}^1_k(\C)\setminus \{ a_1,\dots, a_n\}$ with $a_1,\dots, a_n\in k\cup\{\infty\}$. Up to renaming the $a_i$'s and writing $\Delta$ as a sum of several paths, we may assume that $\mathcal{R}$ corresponds to the points $\{a_1,a_2\}$ and $\Delta$ is a path from $a_1$ to $a_2$. Since $\Delta$ is admissible, it does not wind around any $a_i$ with $i\neq 1,2$.\\
        Let $F$ be given as in the statement. According to Proposition~\ref{prop: regularizing primitives} there exists a primitive $G\in W_{i+1}\mathcal{L}_{\Sigma_\Delta}(V,\A)\otimes_k \mathcal{O}_{Y(V,\A)}$ of $F$ which has no poles at $a_1$ and $a_2$. Thus, $G$ extends continuously to the boundary of $\Delta$ and by Theorem~\ref{theo: Stokes}
        \[
            \int_\Delta F=\int_{\partial \Delta} G = G(a_2)-G(a_1).
        \]
        Let $L_1$ and $L_2$ be the solutions of the holonomy equation localized at $a_1$ and $a_2$ respectively. It is apparent that $G(a_r)$ belongs to the $k$-algebra $C_r$ generated by the coefficients of the series $\text{Reg}(L_1, a_r)$ for $r=1,2$. Since $\text{Reg}(L_1, a_1)=1$ by the boundary condition imposed on $L_1$, we have $C_1=k$. On the other hand, there is a constant series $G_{1,2}$ such that $L_2=L_1G_{1,2}$. Given that $G_{1,2}$ is constant, we have
        \[
            G_{1,2}=\text{Reg}(G_{1,2},a_2)=\text{Reg}(L_2^{-1}L_1,a_2) = \text{Reg}(L_1,a_2).
        \]
        This shows that $C_2$ coincides with $Z_{\Sigma_\Delta}(V,\A,\G)$, hence the validity of the claim for the one-dimensional case.\\
        Suppose now that the statement holds up to dimension $l-1$. Consider a differential form $F\in W_i\Omega^l\mathcal{L}_{\Sigma_\Delta}(V,\A)$ with possible poles only along the irreducible components of $D^\mathcal{P}$. By Proposition~\ref{prop: regularizing primitives} there is $G\in W_{i+1}\Omega^{l-1} \mathcal{L}_{\Sigma_\Delta}(V,\A)$ that has no poles along $D_X$ for $X\in \mathcal{R}$ and satisfies $dF=G$. Since $\Delta$ is admissible, both $F$ and $G$ are single-valued on $\Delta$. By Theorem~\ref{theo: Stokes}, $G$ extends continuously to the boundary of $\Delta$ and 
        \[
            \int_\Delta F= \int_{\partial\Delta} G= \sum_{X\in \mathcal{R}}\int_{\partial \Delta \cap D_X} G.
        \]
        For all $X\in \mathcal{R}$ we know that 
        \[
            D_X\cong \overline{Y}(V/X^\perp, \A\vert_X, \G\vert_X)\times \overline{Y}(X^\perp, \A/X, \G/X).
        \]
        Let $\Sigma_{\Delta, X}$ be the subset of $\Sigma$ consisting of those maximal $\G$-nested sets which contain $X$. We further define
        \[
            \Sigma_{\Delta,X}^{(1)}=\left\{\S\vert_X \mid \S\in \Sigma_{\Delta,X} \right\}, \quad 
            \Sigma_{\Delta,X}^{(2)}=\left\{\S/X \mid \S\in \Sigma_{\Delta,X} \right\},
        \]
        with the convention that the associated adapted bases are naturally induced by the adapted bases in $\Sigma$. Notice that $\Sigma_{\Delta,X}^{(1)}$ and $\Sigma_{\Delta, X}^{(2)}$ are families of maximal $\G\vert_X$- and $\G/X$-nested sets, respectively.\\
        By Proposition~\ref{prop: restriction of holonomy equation to the boundary}, $G\vert_{D_X}$ lies in
        \begin{gather*}
            W_{i+1}\Omega^{l-1}\left(\mathcal{L}_{\Sigma_{\Delta, X}^{(1)}}(V/X^\perp, \A\vert_X)\otimes_k \mathcal{L}_{\Sigma_{\Delta,X}^{(2)}}(X^\perp, \A/X) \right)\\
            \cong W_{i+1}\left(\Omega^{l_1} \mathcal{L}_{\Sigma_{\Delta, X}^{(1)}}(V/X^\perp, \A\vert_X)\otimes_k \Omega^{l_2}\mathcal{L}_{\Sigma_{\Delta,X}^{(2)}}(X^\perp, \A/X) \right),
        \end{gather*}
        where 
        \begin{gather*}
            l_1=\dim \overline{Y}(V/X^\perp, \A\vert_X, \G\vert_X)=\dim X -1, \\
            l_2=\dim \overline{Y}(X^\perp, \A/X, \G/X)=\dim V-\dim X-1=l-1-\dim X.
        \end{gather*}
        We may thus write
        \[
            G\vert_{D_X}=\sum_{r=1}^{n_X} G_{X,r}^{(1)}G_{X,r}^{(2)}
        \]
        with $G_{X,r}^{(1)}\in \Omega^{l_1}\mathcal{L}_{\Sigma_{\Delta, X}^{(1)}}(V/X^\perp, \A\vert_X)$, $G_{X,r}^{(2)}\in \Omega^{l_2}\mathcal{L}_{\Sigma_{\Delta,X}^{(2)}}(X^\perp, \A/X) $ in such a way that the weight of their product is at most $i+1$.\\
        On the other hand, let $\mathcal{P}_X$ be the set of $P\in \mathcal{P}$ such that $\{X,P\}$ is $\G$-nested, that is, $D_P\cap D_X\neq \emptyset$. Notice that $\mathcal{P}_X$ defines two subsets $\mathcal{P}_X^{(1)}$ and $\mathcal{P}_X^{(2)}$ of $\G\vert_X$ and $\G/X$ respectively. Define $\mathcal{R}_X$, $\mathcal{R}_X^{(1)}$, $\mathcal{R}_X^{(2)}$ in an analogous way. We then have an isomorphism
        \begin{gather*}
            H_{l-1}\left( D_X\setminus D^{\mathcal{P}_X}, D^{\mathcal{R}_X}\setminus \left(D^{\mathcal{R}_X}\cap D^{\mathcal{P}_X} \right),k \right) \cong \\
            \cong \bigoplus_{a+b=l-1} H_{a}\left( \overline{Y}(V/X^\perp, \A\vert_X, \G\vert_X)\setminus D^{\mathcal{P}_X^{(1)}}, D^{\mathcal{R}^{(1)}_X}\setminus \left(D^{\mathcal{R}^{(1)}_X}\cap D^{\mathcal{P}^{(1)}_X} \right), k \right)\otimes_k\\
            \otimes_k H_{b}\left(\overline{Y}(X^\perp, \A/X, \G/X)\setminus D^{\mathcal{P}_X^{(2)}}, D^{\mathcal{R}^{(2)}_X}\setminus \left(D^{\mathcal{R}^{(2)}_X}\cap D^{\mathcal{P}^{(2)}_X} \right), k \right).
        \end{gather*}
        The summand of $\partial \Delta \cap D_X$ corresponding to the pair $(a,b)=(l_1,l_2)$ may be written as
        \[
            (\partial\Delta \cap D_X )_{(l_1,l_2)}= \sum_{s=1}^{m_X} \Delta_{X,s}^{(1)}\otimes \Delta^{(2)}_{X,s}. 
        \]
        with $\Delta_{X,s}^{(1)}$ and $\Delta^{(2)}_{X,s} $ in the homology groups suggested by the notation. This summand of $\partial\Delta\cap D_X$ is the only one appearing in the computation of the integral for dimension reasons.\\
        By Fubini's theorem we therefore deduce that 
        \[
            \int_\Delta F = \sum_{X\in \mathcal{R}}\int_{\partial \Delta \cap D_X} G = 
            \sum_{X\in \mathcal{R}} \sum_{r=1}^{n_X} \sum_{s=1}^{m_X} \int_{\Delta^{(1)}_{X,s}}  G_{X,r}^{(1)} \int_{\Delta_{X,s}^{(2)}} G_{X,r}^{(2)}.
        \]
        For $t=1,2$ the integral
        \[
            \int_{\Delta^{(t)}_{X,s}}  G_{X,r}^{(t)} 
        \]
        satisfies the assumptions needed for the induction hypothesis. Indeed, the holomogy class $\Delta^{(t)}_{X,s}$ is admissible by Lemma~\ref{lemma: admissible homology classes}. Moreover, the differential form $G_{X,r}^{(t)}$ lies in $ \Omega^{l_1}\mathcal{L}_{\Sigma_{\Delta, X}^{(1)}}(V/X^\perp, \A\vert_X)$ or $ \Omega^{l_2}\mathcal{L}_{\Sigma_{\Delta,X}^{(2)}}(X^\perp, \A/X)$ and has possible poles only along $D^{\mathcal{P}_X^{(1)}}$ or $D^{\mathcal{P}_X^{(2)}}$. The claim then follows by induction thanks to Proposition~\ref{prop: decomposition of associators}.
    \end{proof}

    \subsection{One-dimensional arrangements}
        
    In this section, we study in more detail the holonomy equation for one-dimensional hyperplane arrangements. This is a classical topic; we follow the exposition and the notation of \cite{Brown-Multiple_zeta_values_and_periods_of_moduli_spaces_of_curves}.\\
    Consider a one-dimensional hyperplane arrangement over $k$, corresponding to $\mathbb{P}^1$ with the points $\sigma_0,\sigma_1,\dots, \sigma_n\in k$ and $\infty$ removed. We may assume that $\sigma_0=0$ and $\sigma_1=1$. In our usual notation for hyperplane arrangements, this means to consider a two-dimension $k$-vector space $V=kx_1\oplus kx_2$ with arrangement $\A$ given by the lines $\langle x_1\rangle$, $\langle x_2\rangle$, $\langle x_1-\sigma_i x_2\rangle$ for all $i=1,\dots, n$. The only building set for $L(\A)$ is $L(\A)$ itself. To work in local charts, we always endow the nested set $\{\langle x_1-\sigma_i x_2\rangle, V^*\} $ with the adapted basis $\gamma$ such that $\gamma(\langle x_1-\sigma_i x_2\rangle)= x_1-\sigma_i x_2$ and $\gamma(V^*)=x_2$.\\
    Let $\omega_i=d\log (z-\sigma_i)$. Since $\mathbb{P}^1$ has dimension $1$, the reduced bar complex of this arrangement is isomorphic to the free shuffle $k$-algebra in the indeterminates $\omega_0,\dots,\omega_n$. Let $L$ be the solution to the holonomy equation with boundary condition at $0$, that is, at the nested set $\{\langle x_1\rangle, V^*\}$. This can be then written as
    \[
        L(z)=\sum_{m=0}^\infty \sum_{I=(i_0,\dots,i_m)} L_I(z) \omega_{i_0}^\lor\dots \omega_{i_m}^\lor
    \]
    and the algebra of generalized multiple polylogarithms is generated by the functions $L_I$'s, which are called \emph{hyperlogarithms}. For a simpler notation, let $\beta$ be the standard basis of the free shuffle algebra over $\omega_0,\dots,\omega_n$ given by monomials. We then write 
    \[
        L(z)=\sum_{w\in\beta} L_w(z) w^\lor,
    \]
    so $L_I(z)=L_w(z)$ for $w=[\omega_{i_0}\vert\dots \vert \omega_{i_m}]$.\\
    In this case, the holonomy equation is equivalent to the system of differential equations
    \[
        \frac{d}{dz} L_{\omega_jw}(z) = \frac{L_{w}(z)}{z-\sigma_{j}}
    \]
    for all $j=0,\dots, n$ and $w\in\beta$.\\
    Taking into account the boundary condition at $\sigma_0=0$, one can check that for all $m\ge 1$
    \[
        L_{[\omega_0^m]}(z)=\frac{1}{m!}\log^m z.
    \]
    On the other hand, take $w=[\omega_{i_0}\vert\dots \vert \omega_{i_m}]$ with $i_m\neq 0$. We then have
    \[
        L_{\omega_j w}(z)=\int_0^z\frac{L_{w}(t)}{z-\sigma_j}dt,
    \]
    which we think of as an equality between multi-valued functions. The condition $i_m\neq 0$ is there to ensure convergence of the integral in the right-hand side. Inductively, this represents $L_w(z)$ as the iterated integral
    \[
        L_w(z)=\int_0^z \frac{1}{t_0-\sigma_{i_0}}\int_0^{t_0}\frac{1}{t_1-\sigma_{i_1}}\dots \int_0^{t_{m-1}}\frac{1}{t_m-\sigma_{i_m}}dt_1\dots dt_m,
    \]
    which will also be denoted by $\int_0^z w$ for short.\\
    Let $\beta_c$ be the subset of $\beta$ consisting of those monomials which do not end with $\omega_0$. To extend the definition of $L_w(z)$ to all elements of the reduced bar complex, it suffices to observe that every $w\in\beta$ can be written as 
    \[
        w=\sum_{v\in\beta_c}\sum_{n\ge0} a_{v,n} v \,\sh\, [\omega_0^n]
    \]
    for some uniquely determined $a_{v,n}\in k$. Since the functions $L_w$ are $k$-algebra homomorphisms in $w$, this extends the definition of $L_w$ to all elements of the reduced bar complex.\\
    We may describe these functions more explicitly locally around zero by introducing multiple polylogarithms. Take $r\ge 1$ and let $n_1,\dots, n_r\in \Z$ be positive integers. The multiple polylogarithm associated with $(n_1,\dots, n_r)$ is defined in a neighborhood of the origin by the power series
    \[
        \text{Li}_{n_1,\dots, n_r}(z_1,\dots, z_r)=\sum_{0<k_1<\dots <k_r} \frac{z_1^{k_1}\dots z_r^{k_r}}{k_1^{n_1}\dots k_r^{n_r}}  
    \]
    in the complex variables $z_1,\dots, z_r$. This power series converges absolutely for $|z_i|<1$. If $n_r\ge 2$, convergence extends to $|z_i|\le 1$. Our convention for the notation agrees with that of \cite{Brown-Multiple_zeta_values_and_periods_of_moduli_spaces_of_curves}, while the definition of \cite{Zhao-Analytic_continuation_of_multiple_polylogarithms} coincides with $\text{Li}_{n_r,\dots, n_1}(z_r,\dots, z_1)$. The quantity $n_1+\dots+n_r$ is called \emph{weight} of the multiple polylogarithm $\text{Li}_{n_1,\dots, n_r}(z_1,\dots,z_r)$.\\
    Given $w\in \beta_c$, we may write
    \[
        w=[\omega_0^{n_r-1}\vert \omega_{j_r}\vert \omega_0^{n_{r-1}-1}\vert \omega_{j_{r-1}}\vert\dots \vert \omega_0^{n_1-1}\vert \omega_{j_1}]
    \]
    with $\omega_{j_1\neq 0}$. For $k=1,\dots, n$, set $h_k=n-j_k+1$ and define $y_1,\dots, y_n$ as follows:
    \[
        y_1=(\sigma_2\dots \sigma_n)^{-1},\;\dots,y_{n-2}=(\sigma_2\sigma_3)^{-1}, \;y_{n-1}=\sigma_2^{-1}.
    \]
    Thus, we have
    \[
        \sigma_2=y_{n-1}^{-1},\; \sigma_3=(y_{n-2}y_{n-1})^{-1}, \; \dots,\; \sigma_l=(y_1\dots y_{n-1})^{-1}.
    \]
    By computing explicitly the derivatives of the polylogarithms using their series expansion, it is possible to see that 
    \[
        L_w(y_n)= (-1)^r\text{Li}_{n_1,\dots, n_r}\left( \frac{y_{h_1}\dots y_{n}}{y_{h_2}\dots y_n},\dots, \frac{y_{h_{r-1}}\dots y_n}{y_{j_r}\dots y_n},y_{h_r}\dots y_n\right).
    \]
    Although in the right-hand side we are computing a multiple polylogarithm at some points of absolute value possibly greater than $1$, the corresponding series still converges for $|y_n|$ sufficiently small. \\
    Indeed, notice that the variable $y_n$ always appears in the last argument of the right-hand side. Let $z_1,\dots, z_r$ be complex numbers such that for $i=1,\dots, r-1$ we have a uniform upper bound of the form $|z_i|<C$ for some $C>0$. Suppose that $|z_r|<C^{-r+1}$. For all $k_1,\dots k_r$ with $0<k_1<\dots < k_r$ we then have
    \[
        \left|\frac{z_1^{k_1}\dots z_r^{k_r}}{k_1^{n_1}\dots k_r^{n_r}} \right|\le \frac{C^{k_1+\dots +k_{r-1}-(r-1)k_r}}{k_1^{n_1}\dots k_r^{n_r}}<\frac{1}{k_1^{n_1}\dots k_r^{n_r}}.
    \]
    Hence, the convergence of the series defining $\text{Li}_{n_1,\dots, n_r}(z_1,\dots, z_r)$ for this choice of parameters is ensured.\\
    The formula above expresses hyperlogarithms in terms of multiple polylogarithms. Conversely, along the same lines one can express multiple polylogarithms in terms of iterated integrals, thus recovering analytic continuation for polylogarithms.\\
    For $i=1,\dots, n$ let $L_i$ be the solution of the holonomy equation localized at $\sigma_i$. Consider the associator $G_i$ such that $L=L_iG_i$. To compute the coefficients of $G_i$, we may observe that 
    \[
        G_i=\text{Reg}(G_i,\sigma_i)=\text{Reg}(LL_i^{-1},\sigma_i)=\text{Reg}(L,\sigma_i).
    \]
    It is easily seen that regularizing $L$ with respect to $\sigma_i$ means to only consider the words $w$ in the reduced bar complex that do not start with $\omega_i$. The coefficients of $G_i$ are therefore obtained by evaluating at $z=\sigma_i$ the functions $L_w(z)$ for all $w$ which do not start with $\omega_i$. When allowed by convergence, this means that the coefficients of $G_i$ are values of multiple polylogarithms evaluated at a specific set of points. It is worth mentioning that these coefficients are the main ingredients of the description of the integrals computed by Theorem~\ref{theo: periods of hyperplane arrangements with enough modular elements}.\\

    Let us now focus on the main case of interest. Let $q\ge 1$ be an integer and $\mu$ a primitive $q$-th root of unity in $\C$. Consider the one-dimensional arrangement defined over $\Q(\mu)$ obtained by removing from $\mathbb{P}^1$ the points $0$, $\infty$ and all the $q$-th roots of unity.\\
    The discussion above implies that the coefficients of the associators of this arrangement are numbers of the form 
    \[
        \text{Li}_{n_1,\dots, n_r}\left( \mu^{s_1},\dots, \mu^{s_r}\right)=
        \sum_{0< k_1<\dots<k_r}\frac{\mu^{k_1s_1+\dots + k_rs_r}}{k_1^{n_1}\dots k_r^{n_r}}
    \]
    for arbitrary $s_1,\dots, s_r\in \Z$. These numbers are called \emph{multiple polylogarithmic values}. \\
    A special case worth mentioning is $q=1$, which yields \emph{multiple zeta values}:
    \[
        \zeta(n_1,\dots, n_r)=\sum_{0<k_1<\dots<k_r} \frac{1}{k_1^{n_1}\dots k_r^{n_r}}
    \]
    for integers $n_1,\dots, n_r\ge 1$ with $n_r\ge 2$. The $\Q$-vector space generated by multiple zeta values is a $\Q$-algebra naturally filtered by the weight. It is conjectured that the weight actually defines a grading of this algebra. Combined with some known results \cite{Brown-Mixed_Tate_motives_over_Z}, this conjecture would imply the algebraic independence of odd zeta values together with $\pi$. There are several known relations between multiple zeta values of the same weight; for a detailed account on the subject, we refer to \cite{Burgos-Gil-Fresan-Multiple_zeta_values_from_numbers_to_motives}.\\
    For $q\ge 1$, among noticeable examples of numbers which are $\Q(\mu)$-linear combinations of multiple polylogarithmic values there is Catalan's constant
    \[
        G=\sum_{n=0}^\infty \frac{(-1)^n}{(2n+1)^2}
        =-\int_{0}^1\frac{\log x}{1+x^2} \,dx,
    \]
    whose irrationality is still an open problem. In general, values of the $\beta$-function
    \[
        \beta(s)=\sum_{n=0}^\infty \frac{(-1)^n}{(2n+1)^s}
    \]
    at positive integers greater than $1$ can also be expressed in terms of multiple polylogarithmic values. \\
    The ones just mentioned are particular instances of Dirichlet $L$-values. Given a Dirichlet character $\chi$, the associated Dirichlet $L$-function is defined as
    \[
        L(\chi,s)=\sum_{n=1}^\infty \frac{\chi(n)}{n^s}
    \]
    for $\text{Re}(s)>1$. Since Dirichlet characters only assume values in zero and roots of unity, it is not difficult to express $L(\chi, n)$ for $n\ge 2$ as a $\Q(\mu)$-linear combination of multiple polylogarithmic values. To draw this connection more explicitly, it is first useful to consider the Hurwitz zeta function
    \[
        \zeta(s,x)=\sum_{n=0}^\infty \frac{1}{(n+x)^s},
    \]
    which converges absolutely for $\text{Re}(s)>1$ and $x\in \C$, $x\neq 0,-1,-2,\dots$ and has analytic continuation to $s\in \C\setminus\{1\}$. \\
    By manipulating the involved series, one can check that 
    \[
        L(\chi,s)=\frac{1}{k^s}\sum_{m=1}^k\chi(m)\zeta\left(s,\frac{m}{k} \right),
    \]
    where $k$ is the conductor of $\chi$. On the other hand, the Hurwitz zeta function satisfies a functional equation \cite{Fine-Note_on_the_Hurwitz_zeta_function} of the following form:
    \[
        \frac{(2\pi)^s}{\Gamma(s)}\zeta(1-s,x)=i^{-s}\text{Li}_s\left( e^{2 \pi i x} \right) + i^s \text{Li}\left( e^{-2 \pi i x} \right)
    \]
    for all $x\in \C$ with $0<\text{Re}(x)<1$. Evaluating this identity at rational $x$ and integer $s$ yields the desired relation between Dirichlet $L$-values and multiple polylogarithmic values.\\
    Recent irrationality proofs for odd zeta values \cites{Fischler-Sprang-Zudilin-Many-odd-zeta_values_are_irrational, Lai-Yu-A_note_on_the_number_of_irrational_odd_zeta_values} exploit linear forms in Hurwitz zeta values. Writing these in terms of polylogarithms draws a connection with periods of algebraic varieties, which should, in turn, give geometric inputs to irrationality proofs.

    \subsection{The reflection arrangement of the full monomial group}
        
    In this section, we sketch how to deduce Theorem~\ref{theo: intro full monomial group} from Theorem~\ref{theo: intro periods of arrangements with enough modular elements} in the introduction. We also hint at possible applications to irrationality proofs.\\
    Let $q\ge 1$ and fix a primitive $q$-th root of unity $\mu$. Set $k=\Q(\mu)$ and let $V$ be a $k$-vector space of dimension $l$ with a fixed basis $x_1, \dots, x_l$ for the dual space $V^*$. Recall the reflection arrangement of the full monomial group from Definition~\ref{def: arrangement full monomial group}
    \[
        \A_{l,q}=\left\{\langle x_1\rangle ,\dots,\langle x_{l+1}\rangle \right\} \cup \{ \langle x_i-\mu^n x_j\rangle \mid n=1,\dots, q,\; i,j=1,\dots,l+1, \;i\neq j \}.
    \]
    As we have seen throughout the exposition, this arrangement has enough $\F(\A_{l,q})$-modular elements, and every one-dimensional hyperplane arrangement obtained from $\A_{l,q}$ by iterated restrictions and quotients is isomorphic to $\A_{1,q}$ or $\A_{1,1}$. \\
    We now isolate some admissible homology classes of $\overline{Y}(V,\A_{l,q}, \F(\A_{l,q}))$ relative to the boundary divisor. Given $\sigma\in \text{Sym}_{l+1}$ and $n=(n_0,\dots, n_l)$ with $n_i\in\{1,\dots, q\}$, we define $\Delta_{\sigma, n}$ as the subset of the complex points of $Y(V,\A)$ defined by the following conditions in projective coordinates $[x_1:\dots:x_{l+1}]$:
    \begin{gather*}
        0<|x_{\sigma(1)}|<\dots < |x_{\sigma(l+1)}|, \\
        \text{$\text{arg}(\mu^{-n_i}x_i)=\text{arg}(\mu^{-n_j}x_j)$ for all $i,j=1,\dots,l+1$}.
    \end{gather*}
    In affine coordinates $t_i=\frac{x_i}{x_{l+1}}$, with $t_{l+1}=1$, these conditions are equivalent to
    \begin{gather*}
        0<|t_{\sigma(1)}|<\dots < |t_{\sigma(l+1)}|, \\
        \text{$\text{arg}(t_i)=\frac{2\pi }{q}(n_i-n_0)$ for all $i=1,\dots, l$},
    \end{gather*}
    choosing the standard branch of the argument. Let us denote by $\overline{\Delta}_{\sigma,n}$ the closure of $\Delta_g$ in $\overline{Y}(V,\A_{l,q},\F_{l,q})(\C)$ in the analytic topology. This defines an admissible $l$-homology class of $\overline{Y}(V,\A_{l,q},\F(\A_{l,q}))$ relative to the boundary divisor. \\
    The irreducible divisors $D_X$ for $X\in \F(\A_{l,q})$ that intersect the boundary of $\overline{\Delta}_{\sigma,n}$ are precisely the ones with 
    \begin{enumerate}
        \item either $X=X(\delta)$ for some $\delta=\{\sigma(1),\sigma(2),\dots, \sigma(i)\}$ with $i\le l$;
        \item or $X=X(\lambda,\nu)$ with $\lambda=\{\sigma(i),\sigma(i+1),\dots, \sigma(j)\}$ for some $i<j$ and $\nu(k)=\mu^{n_k}$.
    \end{enumerate}
    For such choices of $\delta$ we assign the vector $ \mu^{-n_{\sigma_i}}x_{\sigma(i)}$, while for such $\lambda,\nu$ the vector $\mu^{-n_{\sigma(j)}}x_{\sigma(j)}-\mu^{-n_{\sigma(i)}}x_{\sigma(i)}$. This allows us to define adapted bases for each maximal $\F(\A_{l,q})$-nested set that contain only elements of $\F(\A_{l,q})$ whose associated divisor intersects the boundary of $\overline{\Delta}_{\sigma,n}$. The standard local charts on $\overline{Y}(V,\A_{l,q},\F(\A_{l,q}))$ for these adapted bases induce a structure of manifold with corners on $\overline{\Delta}_{\sigma,n}$. \\
    Let $\Sigma$ be the set of these maximal $\F(\A_{l,q})$-nested set with such adapted bases. By working explicitly in coordinates, one may check that $Z_{\Sigma}(V,\A_{l,q},\F(\A_{l,q}))=\Q(\mu)[2\pi i]$, in view of Remark~\ref{remark: logarithms reduce to 2 pi i}. Combining all these properties together with the observations of the previous section, this leads to the following
    \begin{theorem}
    \label{theo: full monomial group}
        Let $\omega$ be an algebraic $l$-differential form on $Y(V,\A_{l,q})$ which has no poles along the irreducible boundary divisors intersecting the boundary of $\overline{\Delta}_{\sigma,n}$. Then the integral
        \[
            \int_{\overline\Delta_{\sigma,n}} \omega
        \]
        is a $\Q(\mu) [2\pi i]$-linear combination of multiple polylogarithmic values of weight at most $l$.
    \end{theorem}
    \noindent
    Period integrals over the arrangement of the full monomial group have appeared several times in the literature in connection with irrationality proofs. For example, Zudilin \cite{Zudilin-Arithmetic_of_Catalan's_constant_and_its_relatives} considers integrals of the form
    \[
        I_{l,n}=\int_{[0,1]^l} \frac{(1-s_1\dots s_l)\prod_{j=1}^l s_j^{n-\frac{1}{2}} (1-s_j)^n }{(1+s_1\dots s_l)^{3n+2}}ds_1\dots ds_l
    \]
    and shows that they are rational linear combinations even $\beta$-values which yield the irrationality of one among $\beta(2),\beta(4),\dots, \beta(12)$. This result was further improved by Lai and Yu \cite{Lai-Yu-A_note_on_the_number_of_irrational_odd_zeta_values}.\\
    On a different note, Zlobin \cite{Zlobin-Expansion_of_multiple_integrals_in_linear_forms} computes explicitly in terms of multiple zeta values integrals of the form
    \[
        \int_{[0,1]^l} \frac{\prod_{i=1}^l s_i^{a_i-1}(1-s_i)^{b_i-a_i-1} }{\prod_{j=1}^m (1-s_{k_j}\dots s_l)^{c_j}}ds_1\dots ds_l.
    \]
    for suitable choices of integer parameters $a_i,b_i,c_i $. These are generalizations of the classical linear forms in zeta values used by Ball and Rivoal \cite{Ball-Rivoa-Irrationalité_d'une_infinité_de_valuers_de_la_fonction_zeta_aux_entiers_impairs} to show the infinitude of the dimension of the $\Q$-vector space generated by odd zeta values. When the parameters $a_i,b_i,c_i$ are taken to be rational, one obtains period integrals over the reflection arrangement of the full monomial group. By Theorem~\ref{theo: full monomial group}, these are linear combinations of multiple polylogarithmic values. \\
    As already mentioned in the previous section, other irrationality proofs for zeta values \cites{Fischler-Sprang-Zudilin-Many-odd-zeta_values_are_irrational, Lai-Yu-A_note_on_the_number_of_irrational_odd_zeta_values} make use of linear combinations of Hurwitz zeta values to sharpen previous asymptotic lower bounds. So far no integral representation of these linear forms has been exploited, and it would be interesting to study these examples in the geometric setup suggested by Theorem~\ref{theo: full monomial group}. Brown \cite{Brown-Irrationality_proofs_for_zeta_values_moduli_spaces_and_dinner_parties} has conducted a study of this kind  on the moduli spaces of curves in relation to previous irrationality proofs. To this extent, he has isolated promising families of periods of these spaces, the so-called \emph{cellular integrals}, which make geometric symmetries more apparent in irrationality proofs. However, these integrals cannot recover the linear forms used in \cites{Fischler-Sprang-Zudilin-Many-odd-zeta_values_are_irrational, Lai-Yu-A_note_on_the_number_of_irrational_odd_zeta_values}. The reflection arrangement of the full monomial group seems a natural candidate to extend Brown's approach to multiple polylogarithmic values. \\
    A study of the symmetries of the arrangements in this section and their role in irrationality proofs would also deserve attention. Given the length of this exposition, we defer a more thorough investigation of these inputs to upcoming work.

    \bibliographystyle{alpha}
    \bibliography{Literatur.bib}
\end{document}